\documentclass[11pt,letterpaper]{amsart}
\usepackage[english]{babel}
\usepackage{amssymb,amsmath,amsfonts,mathrsfs, amsthm}
\usepackage{units}
\usepackage{ulem}
\usepackage{lipsum}
\usepackage{enumerate}
\usepackage[utf8]{inputenc}
\usepackage{cite}
\usepackage{tikz}
\usetikzlibrary{matrix}
\usepackage[pdftex]{hyperref}
\RequirePackage{amscd}
\RequirePackage{epic}
\RequirePackage[all]{xy}
\RequirePackage{url}
\RequirePackage[shortlabels]{enumitem}
\usepackage{empheq}
\usepackage{epsfig}
\usepackage{cite}
\RequirePackage{amscd}
\RequirePackage{epic}
\RequirePackage{eepic}
\RequirePackage[all]{xy}
\RequirePackage{url}
\RequirePackage[shortlabels]{enumitem}
\usepackage{empheq}
\usepackage{nomencl}
\usepackage{epsfig}
\usepackage{graphicx}

\newcommand{\comment}[1]{}
   
    \newcommand{\set}[1]{{\left\{#1\right\}}}
\newcommand{\pa}[1]{{\left(#1\right)}}
\newcommand{\sq}[1]{{\left[#1\right]}}

\newcommand{\abs}[1]{{\left|#1\right|}}
\newcommand{\norm}[1]{{\left |#1\right |}}

\newcommand{\T}{\mathbb{T}}
\newcommand{\Z}{\mathbb{Z}}
\newcommand{\R}{\mathbb{R}}
\newcommand{\C}{\mathbb{C}}

\newcommand{\dg}{{\mathtt{D}_{\g,\cS}}}
\newcommand{\dgp}{{\mathtt{D}_{\g,\cS}}}

\newcommand{\eps}{\varepsilon}

\renewcommand{\Im}{\operatorname{Im}}

\newcommand{\pan}{{\mathcal Q}}
\newcommand{\na}{\widehat{n}}

\newcommand{\co}[1]{\textit{#1}}
\newcommand{\gr}[1]{\textbf{#1}}

\newcommand{\id}{\operatorname{Id}}

\newcommand{\ad}{\operatorname{ad}}

\newcommand{\jjap}[1]{\lfloor #1 \rfloor }

\RequirePackage{url}
\newtheorem{prop}{Proposition}[section]
    
      \newtheorem{them}{Theorem}[]
    \newtheorem*{thm*}{Theorem}
    \newtheorem*{cor*}{Corollary}
    \newtheorem*{teo normal}{"Twisted Conjugacy" Theorem}

    \newtheorem{cor}{Corollary}
    \newtheorem{lemma}{Lemma}
    \theoremstyle{remark}
     \newtheorem{ex}{Example}
\newtheorem{rmk}{Remark}[section]
\theoremstyle{definition}
\newtheorem{defn}{Definition}

\numberwithin{equation}{section}
\numberwithin{thm}{section}
\numberwithin{defn}{section}
\numberwithin{prop}{section}
\numberwithin{cor}{section}
\numberwithin{lemma}{section}
\numberwithin{rmk}{section}

\newcommand{\g}{\gamma}

\newcommand{\s}{{\sigma}}

\newcommand{\scH}{{\mathscr{H}}}

\newcommand{\E}{{\mathbb E}}

\newcommand{\N}{{\mathbb N}}

\newcommand{\cB}{{\mathcal B}}
\newcommand{\cC}{{\mathcal C}}

\newcommand{\cF}{{\mathcal F}}
\newcommand{\cG}{{\mathcal G}}
\newcommand{\cH}{{\mathcal H}}
\newcommand{\cI}{{\mathcal I}}

\newcommand{\cK}{{\mathcal K}}
\newcommand{\cL}{{\mathcal L}}
\newcommand{\cM}{{\mathcal M}}
\newcommand{\cN}{{\mathcal N}}
\newcommand{\cO}{{\mathcal O}}

\newcommand{\cQ}{{\mathcal Q}}
\newcommand{\cR}{{\mathcal R}}
\newcommand{\cS}{{\mathcal S}}
\newcommand{\cT}{{\mathcal T}}
\newcommand{\cU}{{\mathcal U}}
\newcommand{\cV}{{\mathscr V}}

\newcommand{\fm}{{\mathfrak{m}}}

\newcommand{\tc}{{\mathtt{c}}}
\newcommand{\td}{{\mathtt{d}}}

\newcommand{\tq}{{\mathtt{q}}}

\newcommand{\sob}{{\mathtt{w}}}

\newcommand{\tD}{{\mathtt{D}}}

\newcommand{\tK}{{\mathtt{K}}}

\newcommand{\tM}{{\mathtt{M}}}

\newcommand{\be}{{\bf e}}

\usepackage{bm}

\newcommand{\al}{{\alpha}}
\newcommand{\bt}{{\beta}}

\newcommand\norma[1]{\left\lVert#1\right\rVert}

\newcommand{\im}{{\rm i}}
\newcommand{\jap}[1]{\langle #1 \rangle}
\newcommand{\und}[1]{\underline{#1}}

\newcommand{\e}{{\varepsilon}}

\newcommand{\meas}{{\rm meas}}

\definecolor{aquamarine}{rgb}{0,0.5,0.5}

\newcommand{\tw}{{\mathtt{w}}}
\newcommand{\nnorm}[1]{{\left\vert\kern-0.25ex\left\vert\kern-0.25ex\left\vert #1 
    \right\vert\kern-0.25ex\right\vert\kern-0.25ex\right\vert}}

\newcommand{\bcoeffu}[1]{#1_{\bal',\bbt'}}
\newcommand{\bcoeffd}[1]{#1_{\bal'',\bbt''}}

\newcommand{\uuu}{u^{\bal'}\bar{u}^{\bbt'}}
\newcommand{\uud}{u^{\bal''}\bar{u}^{\bbt''}}

\newcommand{\pon}{{{\Pi^{0,\mathcal{K}}}}} 

\newcommand{\modi}[1]{\abs{u_{#1}}^2}

\newcommand{\es}{e^{\set{S,\cdot}}}
\newcommand{\bal}{{\bm \al}}
\newcommand{\bbt}{{\bm \bt}}
\newcommand{\baluno}{\bal'}
\newcommand{\baldue}{\bal''}
\newcommand{\bbtuno}{\bbt'}
\newcommand{\bbtdue}{\bbt''}

\newcommand{\zero}[1]{#1^{(0)}}
\newcommand{\zeroR}[1]{#1^{(0,\cR)}}
\newcommand{\zeroK}[1]{#1^{(0,\cK)}}
\newcommand{\due}[1]{#1^{(-2)}}
\newcommand{\buon}[1]{#1^{\ge 2}}
\newcommand{\call}{\mathcal{L}}

\usepackage{framed,enumitem} 

\newcommand{\dueK}[1]{#1^{(-2,\cK)}}

\renewcommand{\o}{{\omega}}

\newcommand{\betta}{{}}

\newcommand{\ub}{v}
\newcommand{\zb}{z}

\newcommand{\bO}{{ \Omega}}

\newcommand{\cachi}{{\zeta}}

\newcommand{\ur}{{\bf u}}
\newcommand{\rr}{{\rm r}}

\newcommand{\VS}{{V_\cS}}
\newcommand{\VSc}{{V_{\cS^c}}}
\newcommand{\VSf}{{{\mathscr V}_\cS}}
\newcommand{\VSs}{{{\mathscr V}_\cS^*}}
\newcommand{\Vs}{{{\mathscr V}^*}}

\newcommand{\gatta}{{\td}}

\begin{document}
\author{Luca Biasco}
\address{Università degli Studi Roma Tre}
\email{biasco@mat.uniroma3.it}

\author{Jessica Elisa Massetti}
\address{Università degli Studi Roma Tre}
\email{jmassetti@mat.uniroma3.it}

\author{Michela Procesi}
\address{Università degli Studi Roma Tre}
\email{procesi@mat.uniroma3.it}
 
 \title{Small amplitude weak almost periodic solutions for the 1d NLS}
\begin{abstract}
 All the almost periodic solutions for non integrable PDEs 
found in the literature are
very regular (at least $C^\infty$) and, hence, very close to quasi periodic ones.
This fact is deeply exploited in the existing proofs. Proving the existence of almost periodic solutions with finite regularity is a main open problem in KAM theory for PDEs.
Here we  consider  the 
one dimensional   NLS with external parameters
and 
 construct almost periodic solutions which have only  Sobolev regularity 
both in time and space. 
Moreover many of our solutions are so only in a weak sense.  
This is the 
first result on existence of weak, i.e. non classical,  solutions 
for non integrable PDEs
in KAM theory. 
\end{abstract} 
\maketitle
\setcounter{tocdepth}{1} 
\tableofcontents

\section{Introduction}

We present here the first result on existence of weak, i.e. non classical,  solutions 
for non integrable PDEs
in KAM theory.
More precisely  we study  NLS equations on the circle
finding almost periodic solutions, 
i.e. solutions which are limit (in the uniform topology in time) of time-quasi-periodic functions.
We work on models with external parameters of the form:
\begin{equation}\label{NLS}\tag{${\rm NLS}_V$}
\im \ur_t + \ur_{xx} - V\ast \ur +
f(|\ur|^2)\ur=0\,,\qquad \ur(t,x)=\ur(t,x+2\pi)\,,
\end{equation}
where  $\im=\sqrt{-1}$
and $f(y)$ is real analytic in $y$ in a neighborhood of $y=0$ with $f(0)=0$.
{Given $V=\pa{V_j}_{j\in\Z}\in 
[-\nicefrac14,\nicefrac14]^\Z\subset \ell^\infty(\R)$,} the Fourier multiplier $V\ast$ is defined as
the bounded\footnote{Clearly the operator norm is $\|V\|_{\mathcal L(\mathcal F(\ell^1),\mathcal F(\ell^1))}=|V|_\infty:=\sup_j|V_j|.$} 
linear operator $V\ast:\mathcal F(\ell^1)\to\mathcal F(\ell^1)$
by\footnote{More precisely $(V\ast \ur)(t,x):=(V\ast \ur(t,\cdot))(x)$ for every $t\in\R.$}
\begin{equation}
\label{vu}
u(x)=\sum_{j\in\Z}  u_j e^{\im j x}\mapsto
(V\ast u)(x) = \sum_{j\in\Z} V_j u_j e^{\im j x}\,,
\end{equation}
where $\mathcal F(\ell^1)$ is the Wiener algebra
 of $2\pi$-periodic 
functions\footnote{Note that all such functions are continuous.} 
having Fourier coefficients in $\ell^1=\ell^1(\C):=\{u:= \pa{u_j}_{j\in\Z}\in \C^\Z\; : \norm{u}_{\ell^1}:= \sum_{j\in\Z}\abs{u_j} < \infty\}$ endowed with the 
corresponding
norm $|u(\cdot)|_{\mathcal F(\ell^1)}:=\norm{u}_{\ell^1}$. 
 Note that the choice of the ball 
 $[-\nicefrac14,\nicefrac14]^\Z$
 is rather arbitrary and any other  ball  is 
 acceptable.
 
Besides classical solutions, namely functions $\ur$ 
admitting continuous derivatives $\ur_t$ and $\ur_{xx}$ solving
\eqref{NLS}, we are particularly interested in
 weak solutions according to the following

 \begin{defn}[Global in time weak solutions]\label{franefritte} A function $\ur:\R^2\to\C$
 which is $2\pi$-periodic in $x$ and 
 such that the map
 $t\mapsto \ur(t,\cdot)\in\mathcal F(\ell^1)$ is continuous
  is a weak solution of  \eqref{NLS}
 if for any smooth compactly supported function $\chi:\R^2\to\R$ one has
  \begin{equation}\label{ranefritte}
 	\int_{\R^2} (-\im \chi_t + \chi_{xx})\ur - (V\ast \ur -
 	f(|\ur|^2)\ur)\chi\, dx\, dt\,=0\,.
 \end{equation}
 \end{defn}
\noindent Note that according to our definition a weak solution is a 
 continuous\footnote{Indeed $|\ur(t,x)-\ur(t_0,x_0)|
 \leq |\ur(t,\cdot)-\ur(t_0,\cdot)|_{\mathcal F(\ell^1)}+|\ur(t_0,x)-\ur(t_0,x_0)|$.}
  function
 on $\R^2$.
 
Local well posedness  for equations as \eqref{NLS} is  well known even in lower regularity 
 (see, e.g. \cite{Bou93} and \cite{Bou94}).  In the context of integrable PDEs  there are 
 various results on weak almost-periodic solutions, we mention \cite{KM,KM2} for the KdV and mKdV,  \cite{GK} for the Szego equation and \cite{GeKa} for the Benjamin-Ono.
In the present paper we work close to the elliptic fixed point $\ur=0$ and consider  \eqref{NLS} as a small (non integrable) perturbation of the (integrable) linear Schr\"odinger equation and prove the persistence of  almost-periodic solutions.
  \begin{them}\label{deltoide}
 For almost every Fourier multiplier $V$
there exist infinitely many small-amplitude weak almost-periodic solutions  of \eqref{NLS}.
Infinitely many of such solutions are not classical and infinitely many are classical.
 \end{them}
 
 \noindent
Here by {\co{almost every}} we mean a full measure 
set with respect to  the {\sl product  probability 
measure} on $[-\nicefrac14,\nicefrac14]^\Z$.
The Borel sets on such measure are respect to the 
product topology on $[-\nicefrac14,\nicefrac14]^\Z$
(see Appendix \ref{ciavatta}). {This theorem will be derived in Section 3 from Theorem \ref{teo kam} and Theorem \ref{toro Sobolevp}, stated below. }

\bigskip

Before describing in more detail our results, we briefly motivate why one is interested in almost-periodic solutions and particularly in those of low regularity.
\\
In the study of finite dimensional nearly-integrable Hamiltonian systems  KAM Theorems play a pivotal role,
casting a light that illuminates the picture quite clearly.
 Indeed under some (generic) non-degeneracy assumptions
 most of the phase space of such systems  is foliated by maximal invariant tori,
whose dimension is half of the one of the whole space.
In particular the system is not ergodic and the majority of initial data give rise  to
quasi-periodic solutions that densely fill some invariant torus and are, therefore, 
perpetually stable.   
  Possible chaotic behavior is restricted to a set of small measure. 
  \\
On the other hand,  in the infinite dimensional setting, for example in the PDEs case,
the general picture is so far rather obscure and
the main questions still remain unanswered.
The typical  solutions
of an infinite dimensional integrable system are the
\co{almost-periodic} ones that lie on maximal infinite dimensional invariant tori;
what is their fate under perturbation?
Is it still true that the majority of initial data produce perpetually stable solutions?
The meaning and possible answers to these questions are deeply related to the regularity
of the phase space in which one looks for solutions.
These questions are completely open
and only very partial answers  are available.
In particular the known results are on spaces of very high regularity (contained in $C^\infty$).
On the other hand,  finite regularity solutions appear naturally and are widely studied in the PDE-context. 
For instance the persistence of Sobolev almost periodic solutions 
is the first open problem on KAM for PDEs mentioned 
 by S. Kuksin  in \cite{Kuk15} (see e.g. Problem 7.1). Even though our solutions are quite special, being mainly supported on sparse set of Fourier's modes, the present work can be seen as a first step forward in this direction.
\\
{As is usual in KAM Theory, a key point in the study of the dynamics in the neighborhood of these invariant tori, consists in controlling the spectral properties of appropriate linear operators and dealing with the connected problem of small-divisor. The main difficulty is then to guarantee that suitable arithmetic (Diophantine) conditions on the frequencies are fulfilled all along the scheme, so that small-divisors can be bounded accordingly and the almost-periodic dynamics controlled.}
Extending Diophantine (or similar) estimates, which strongly depend on the dimension,
to the infinite dimension is not straightforward.
In fact all the results on almost-periodic solution for PDEs only deal with 
 special model cases and consider 
 parameter dependent equations.
\vskip3pt
{\it Parameter vs. regularity issues.} 
 Quoting Bourgain \cite{Bourgain:2005}: ``the role of this parameter is essential to ensure appropriate non-resonance properties of the (modulated) frequencies along the iteration.''
\\
As we shall see below, in all the existent literature, the external parameters have rather low regularity (giving a good frequency modulation) while the almost periodic solutions have an \co{extremely fast} decay in their Fourier coefficients, which approach zero 
  super-exponentially, exponentially or sub-exponentially (Gevrey). This means that those solutions are ``very close" to quasi-periodic ones.
 For example  P\"oschel  in \cite{Poschel:2002} studies an NLS with a multiplicative potential in $L^2$(producing an infinite set of free parameters) and smoothing non-linearity and constructs almost-periodic solutions iteratively,   through successive small perturbations of finite (but at each step higher) dimensional invariant  tori. This leads to  a very strong compactness property: in order to overcome the dependence  of the KAM estimates on the  dimension, the distances  of these tori have to shrink super-exponentially, this leading to very regular solutions. See also \cite{GX13} for a generalization of P\"oschel's approach to the analytic cathegory, by using T\"oplitz-Lipschitz function techniques. 
 In his pioneering work \cite{Bourgain:2005} on the quintic NLS with Fourier multipliers (providing external parameters in $\ell^\infty$),  Bourgain  proposed a different approach 
 which does not rely on  approximations  by quasi-periodic functions 
 by working directly in Fourier space, and relying on a Diophantine condition which is tailored for the infinite dimension. For most choices of the parameters, this leads to the construction of almost periodic invariant tori which support Gevrey solutions
 (see also \cite{BMP:almost},\cite{CoY}). 
\noindent
In any case, all the above results are valid for most choices of the (infinitely many) external parameters. An open question is whether one can achieve a similar result by modulating finitely many parameters. This would lead to applications to more natural PDEs. \\
Of course, the most challenging scenario is represented by a fixed PDE, where the parameters should be ``extracted" from the initial data.

{\it Almost-periodic solutions for a fixed PDE.} Again, quoting Bourgain ``In the absence of external parameters, these (non resonance, \co{ed.})  conditions need to be realized from amplitude-frequency modulation and suitable restriction of the action-variables. This problem is harder. Indeed, a fast decay of the action-variables (enhancing convergence of the process) allows less frequency modulation and worse small divisors". \\
Summarizing, the possibility of constructing almost-periodic solutions for a \co{fixed} PDE, i.e. ``eliminating" the external parameters through amplitude-frequency modulation, appears to be intimately related to the regularity issues. 
 {Moreover, in the context of completely integrable PDEs, the invertibility of the amplitude-frequency map is known only in spaces of very low regularity {(see \cite{KapMas})}.
It then becomes fundamental to look for  almost-periodic solutions in lower regularity spaces if we want to bypass the introduction of external parameters.} While it is possible to lower the regularity beyond the Gevrey class (see  \cite{Co})
up to now
all the known solutions are at least $C^\infty$.
On the other hand
 finding {\sl finite regularity} solutions appears to be a very difficult question, due to extremely small divisors.
\\
{Comparable difficulties in tackling the Sobolev case appear in Birkhoff Normal Form theory for PDEs leading to two different types of results.
In the analytic or Gevrey case one has sub-exponential stability times
(see \cite{Faou-Grebert:2013} and \cite{Cong}),
whereas in the Sobolev case the stability times appear to be controlled by the Sobolev exponent
(see \cite{Bambusi-Grebert:2006}, \cite{FI}, \cite{Berti-Delort},\cite{BMP:2019}, \cite{FM:weak}).}
\\
  The counterpart of total and long time stability results is the construction of unstable trajectories, which undergo  growth of the Sobolev norms, see \cite{Bou1,CKSTT, GuaKa, GHHMP, GGMP}.
  \vskip3pt
{\it Comparison with the quasi-periodic case.} In the context of quasi-periodic solutions there is a wide literature regarding the finite regularity case (starting from the seminal paper \cite{Kuk88}). 
A good  strategy, which works also for fully-nonlinear PDEs,  is to apply a Nash-Moser scheme and prove tame estimates on the inverse of the linearized equation at an approximate solution.
This method was proposed in
\cite{BB1} (generalizing the seminal works \cite{CW}, \cite{Bo4} concerning the analytic case) via multi-scale analysis, see also
\cite{BB3}
or   \cite{BBM1} for a reducibility approach.
\\
Note that in the  existing KAM literature \co{finite regularity} solutions are always classical and due to \co{finite regularity} nonlinearities.
Typically one
looks for an {\sl invariant  embedded finite} dimensional torus 
in a fixed phase space of $x$-dependent functions. 
Then the regularity of the embedding is clearly related to the one of the non-linearity,
which, in the Nash-Moser schemes, is required to be high enough increasing with
the dimension of the torus. 
This is not surprising since in  $n$-dimensional Hamiltonian systems
the minimal regularity of the Hamiltonian  vector field in order to construct a 
maximal Diophantine  torus is essentially $C^n$
(see \cite{Her},\cite{Kou},\cite{Poe}).
For analytic nonlinearities one obtains analytic (or, at least, $C^\infty$)
embedded tori and hence smooth in time solutions. 
By bootstrap arguments, smoothness in space follows.
We finally stress that in the integrable case 
one can construct even periodic solutions with very low regularity,
see for instance \cite{GeKa}. 
\\
In comparison, most of the almost-periodic literature concentrates on the construction of the solutions rather than the invariant objects. Inspired by \cite{Bourgain:2005}, 
 a conceptual novelty of \cite{BMP:almost}
 was to look directly for 
     an{\sl \footnote{In the infinite-dimensional case, whether this is an embedding depends strongly on the chosen topology. See discussion in Subsection \ref{armando}.} infinite dimensional invariant torus}. 
     The strategy developed in   \cite{BMP:almost} applies also to non maximal tori
     both of infinite and finite dimension.
     The persistence result was achieved through an abstract normal form theorem ``à la Herman" (in finite dimensional systems see \cite{Massetti:ETDS, Massetti:APDE}), whose estimates are  \co{uniform in the dimension} of the considered torus
     (see \cite[Theorems 3 and 7.1]{BMP:almost}).
     \\
     Note that in looking for infinite dimensional tori the topology of the phase space
     plays a fundamental role.  In particular the analyticity of the embedded torus 
     does not necessarily imply the analyticity in time of the solution,  at 
	least if the frequencies are unbounded as it is typical in PDEs. 

{\it Small divisors for infinite frequencies  in finite regularity.}
As discussed above, lowering the regularity of the phase space 
requires dealing with very small divisors.
In order to apply the general approach of  \cite{BMP:almost} in a phase space    with finite regularity one needs to look at the problem from a novel perspective and introduce new ideas.
The starting point is 
 to look for {\sl special tori which are approximately supported, in Fourier space, on a sparse subset of $\Z$} called ``tangential sites''(see Definition \ref{orey}), then prove that our solutions are
 supported on such sites, see \eqref{legna}, up to a close to identity change of variables.
The crucial fact is that  the choice of tangential sites provides an extra set of parameters which can be used, for instance,  in order to avoid resonances or simplify small divisor estimates (see for instance \cite{PP3}, \cite{HP}).
In fact,  we prove that \co{an appropriate choice} of the tangential sites, see definition \ref{orey}, allows us to impose {\it very strong} Diophantine conditions, see Definition \ref{diomichela},  so that our small divisors can be controlled similarly to the Gevrey case of \cite{Bourgain:2005, BMP:almost}. 
\\
Having constructed an invariant torus contained in the phase space, 
we  show that  it is the support of weak almost periodic solutions (according to Definition \ref{franefritte}). 
Finally we prove that many
 tori support  almost periodic solutions which are not classical.
 
 \medskip
 
 	\gr{Acknowledgments.} We wish to thank D. Bambusi, M. Berti, L. Corsi, R. Feola, E. Haus, T. Kappeler, J.-P. Marco, and A. Maspero for useful discussions and suggestions. 
	We are also grateful to the anonymous referees 
	for their careful reading of the paper as well as for
	pointing out some bugs in the first version; 
	their suggestions greatly helped us to improve the exposition.
	\\ 
	The authors are partially supported, by PRIN 2020XB3EFL, Hamiltonian and Dispersive PDEs".
	 J.E. Massetti and M. Procesi acknowledge the support of the INdAM-GNAMPA grant ``Spectral and
dynamical properties of Hamiltonian systems”. J.E. Massetti acknowledges the support of the INdAM-GNAMPA grant ``Chaotic and unstable behaviors of infinite-dimensional dynamical systems" .

 \smallskip

\section{Main results}

It is convenient to exploit the Hamiltonian nature of \eqref{NLS}. As was shown by Bourgain in \cite{Bourgain:2005} it is most convenient to use  a product space as phase space. To this purpose
we  introduce the scale of Banach 
spaces\footnote{Obviously one could also take the more standard weight $\jap{j}:= \max\{1, |j|\}$
	instead of $\jjap{j}$, which
	generates the same Banach space, namely $\tw_p\equiv \ell^{\infty, p}$.
	We made such choice for merely technical reasons (see proof of the monotonicity property given in Proposition \ref{che testicoli} contained in \cite[Proposition 6.3]{BMP:2019}).} 
\begin{equation}\label{pecoreccio}
	\tw_p
	:= 
	\set{u:= \pa{u_j}_{j\in\Z}\in\ell^1(\C)\; : \quad \norm{u}_{p}:= \sup_{j\in\Z}\abs{u_j} \jjap{j}^{p}< \infty},\,  p>1\,.
\end{equation}
where $
	\jjap{j} := \max\{2, |j|\}.$
We look for solutions in the Fourier-Lebesgue space $\cF(\tw_p)$ of  $2\pi$-periodic functions whose Fourier coefficients belong to $\tw_p$. In the following we will identify functions and sequences
writing, e.g. $\ell^1$ instead of $\mathcal F(\ell^1)$ or $\tw_p$  instead of $\mathcal F(\tw_p)$ and so on.
Moreover, we have the following standard immersion properties\footnote{For $u\in\tw_p$ we have
	$
	\|u\|_{L^\infty}\,, \,\|u^{(k)}\|_{L^\infty}\leq \sum_{j} \jap{j}^k|u_j|\leq \sum_{j} \jjap{j}^{k-p}|u|_{\tw_p}.
	$
	Moreover if $u\in H^p$ then,  by Parseval equality, 
	$\sum_j |j|^{2p}|u_j|^2=\|u^{(p)}\|_{L^2}^2$ and 
	$\sum_j |u_j|^2=\|u\|_{L^2}^2$, so that
	$\sup_j |j|^{p}|u_j|\leq\|u^{(p)}\|_{L^2}$ and 
	$|u_0|\leq\|u\|_{L^2}$, proving the second inequality in \eqref{katsendorff}.}.
	For $0\leq k<p-1$ we have
	$C^p\subset H^p\subset \tw_p\subset C^{k}$ 
	with estimate
	\begin{equation}\label{katsendorff}
		\tc
		\max\{\|u^{(k)}\|_{L^\infty},\|u\|_{L^\infty}\}
		\leq |u|_{\tw_p}
		\leq 2\max\{\|u^{(p)}\|_{L^2},\|u\|_{L^2}\}\,,
	\end{equation}
	where the first inequality holds for $u\in\tw_p$ and the second one for $u\in C^p$
	and where  $\tc^{-1}:=\sum_{j} \jjap{j}^{k-p}$.

We endow $\sob_{p}\subset \ell^2$ with the symplectic structure 
$\im\sum_j d {u}_j\wedge d\bar{u}_j$  inherited from $\ell^2$.
We define  the NLS Hamiltonian
\begin{equation}\label{Hamilto0}
\begin{aligned}
	& H_V(u):=\sum_{j\in\Z} (j^2+V_j) |u_j|^2+P\,,\, \, \mbox{with } \\ 
	&P:=- \int_\T F(|\sum_j u_j e^{\im j x}|^2) dx \,,\quad F(y):=\int_0^y f(s) ds\,.
	\end{aligned}
\end{equation}
Note that  $\tw_p$ is a Banach algebra w.r.t convolution and moreover
$V\ast:\tw_p\to\tw_p$ is a  linear bounded operator with norm $|V|_\infty.$ This implies  that \eqref{Hamilto0} is an analytic function on
 $\tw_p\cap H^1$, see Proposition \ref{neminchia}.
 \\
 By \eqref{katsendorff}, if $p>3/2$ then $\tw_p\subset H^1$. Otherwise, for  $1<p<3/2$ we might have  infinite energy.
Anyway, for our purposes, we only need
 that the Hamiltonian vector field 
$$
X_{H_V}^{(j)} := \im \partial_{\bar{u}_j}H_V = \im (j^2 + V_j) u_j +\im \partial_{\bar{u}_j} P
$$
is well defined component-wise.
For completeness  we remark that
 the flow of $H_V$ is locally well-posed on  $\tw_p$, for $p>1$. 
 Indeed in Proposition \ref{neminchia}
 (see also \eqref{panettone})
 we will show that the Hamiltonian vector 
 field $X_P$ is a uniformly bounded map from 
 a suitable ball around the origin of 
$\tw_p$ to $\tw_p$;
therefore, by standard variation of constants,
we obtain the 
local  well-posedness on  $\tw_p$.
\smallskip

The proof of Theorem \ref{deltoide} is based on the construction of an analytic change of variables on the phase space $\tw_p$, which conjugates the nonlinear dynamics to a linear one, supported on an invariant flat torus. The regularity of the solution in the original variables is then deduced from the dynamics on the flat torus.
Since these issues require some care, we find it instructive to first discuss the linear case,  where the key difficulties become transparent.

\subsection{The linear case}\label{armando}
When $f=0$ the Hamiltonian reduces to its quadratic part $\sum_{j\in\Z} (j^2+V_j) |u_j|^2,$ so that the linear actions $|u_j|^2$
are constants of motions and the dynamics is
$$
u_j(t)=  u_j(0)e^{\im  \nu_j t}\,, \quad j\in\Z\,, \qquad
\nu=(\nu_j)_{j\in\Z}\,,\ \ \nu_j:=j^2+V_j\,.
$$

\noindent
Let us call $\mathcal S_*:=\{j\in\Z\ \mbox{s.t.}\ u_j(0)\neq 0\}$.
If $\mathcal S_*$ is a finite set,
the corresponding solution $\ur(t,x):=\sum_{j\in\Z}u_j(t) e^{\im jx}$ is quasi-periodic and 
analytic both in time and space.
If $\mathcal S_*$ is  infinite, the regularity of $\ur(t,x)$
 obviously depends on the one of the initial datum.
If $u(0):=(u_j(0))_{j\in\Z}\in\ell^1$ then
$\ur(t,x)$ is a  weak solution of \eqref{NLS}  in the sense of \eqref{ranefritte}.  
Moreover, such solution is a time almost-periodic function,
being limit in $\mathcal F(\ell^1)$ of the quasi periodic truncations
$\sum_{|j|\leq n}u_j(0)e^{\im \nu_j t+\im j x}$ as $n\to\infty$.
Furthermore if $(u_j(0))_{j\in\Z}\in\tw_p$
then $t\mapsto \ur(t,\cdot)\in\mathcal F(\tw_{p})$
and its truncations converge\footnote{Note that 
the convergence
does not necessarily hold in $\tw_p$.  See the examples below.} in 
$\tw_{p'}$ with $1<p'<p$.
\\
Note that if $p>3$ we have a classical solutions
(recall \eqref{katsendorff}).
Otherwise if $p\le 2$ one can easily produce non classical solutions. Indeed 
take  any infinite subset $\cS_*$ of $\Z$ and
 set $u_j(0)=\jap{j}^{-p}$
for $j\in \cS_*$ and $u_j(0)=0$ otherwise.	
Then, if $p\le 2$
\begin{equation}\label{piscia}
\limsup_{|j|\to\infty} j^2|u_j(t)|
=\limsup_{|j|\to\infty} j^2|u_j(0)|> 0\,,
\end{equation} 
so  $\ur(t,\cdot)\notin H^2$. Finally, if $\cS_*=\Z$  and $p\le 3/2$ then $\ur(t,\cdot)\notin H^1$.

The support of each solution is an invariant torus in the following sense.	Given
$I:=(I_j)_{j\in\Z}\in \tw_{2p}$ with $I_j\geq 0.$
  We define the flat torus
	\begin{equation}\label{pippero}
			\cT_I:=\{u\in \tw_{p}: |u_j|^2= I_j \quad \forall j\in \Z\}
		\end{equation}
		and 
	the set
	\begin{equation}\label{shula}
\cS_I:=\{j\in\Z\ \mbox{s.t.}\ I_j>0\}\,.
\end{equation}	
Then 	 the map  $\mathfrak i: \T^{\cS_I} \to \cT_I\subset \tw_p,\quad \varphi=(\varphi_j)_{j\in\cS_I}\mapsto \mathfrak i (\varphi)\,,$ with
	\begin{equation}\label{manodedios}
	\begin{cases}
	\mathfrak i_j (\varphi):= \sqrt{I_j} e^{\im \varphi_j} & \ {\rm for } \ j\in\cS_I\,,\\  
	\mathfrak i_j (\varphi):=0 & \ {\rm otherwise}\,,
	\end{cases}
	\end{equation}
  is an analytic 
	immersion\footnote{Assuming also that 
	$\inf_j \sqrt{I_j}\jap{j}^p>0$ the map $\mathfrak i$ is an embedded torus. Otherwise, $\mathfrak i$ is an homeomorphism on the image only if we endow both source and target spaces with the product topology} provided that we endow $\T^{\cS_I}$ with the $\ell^\infty$-topology
	(see Lemma \ref{lemmadedios} below for details). 
By construction the linear dynamics on the torus $\cT_I$ is $\varphi \to \varphi+ \nu t$. 
\\	
Since the map $ t\mapsto \nu t\in \T^{\cS_I}$ is not even continuous (endowing $\T^{\cS_I}$ with the $\ell^\infty$-topology),  the regularity of $t\mapsto\mathfrak i(\nu t) $ depends on the choice of the actions $I_j$. In the examples discussed above
it  is not continuous w.r.t. the strong\footnote{Note that the map is continuous endowing $\tw_p$ with  the product topology, which coincides with the weak $*$ topology on bounded sets.} topology, see also \cite{KM}. 
{In contrast with the finite dimensional case, even if $\nu$ has rationally independent entries,  it is not straightforward to understand whether this invariant torus is densely filled\footnote{In the product topology such solutions are always dense. } by the solution's orbit or not. In fact, this issue is related to the asymptotic behavior of $\nu$. }
We will discuss this in Lemma \ref{uccellagione}.
	
\begin{defn}[Invariant flat torus]\label{uccellone}
Let $I\in\tw_{2p}$
and $\cT_I$ defined in \eqref{pippero}.
Consider a vector field $X=(X^{(j)})_{j\in\Z}:\cT_I\to \C^\Z$.
We say that $\cT_I$ is an 
invariant flat torus for $X$ with frequency
$\nu\in\R^\Z$
if 
		$$
		\dot u_j(t)=X^{(j)}(u(t))\,,
		\quad u_j(t):=u_j(0) e^{\im \nu_j t}\,, \quad   \forall t\in \R\,, \quad |u_j(0)|^2 = I_j, \quad
		j\in\Z\,.
		$$	
\end{defn}
\noindent
Note that the frequency $\nu_j$ is uniquely defined only for $j\in 
\mathcal S_I$.
\smallskip

\subsection{The nonlinear case}\label{alfonso}
In the nonlinear setting, as explained in the introduction, our invariant tori are approximately supported on $\mathcal S_I$
contained in an appropriately  sparse set $\cS\subset\Z$.

\begin{defn}[Admissible tangential sites]\label{orey}
\smallskip

a) \, An infinite subset $\cS$ of  $\N$, referred to as the {\it set of tangential sites in $\N$}, is said to be {\it admissible}  if there exists a
a smooth strictly increasing
function \footnote{The function $s(i)$ is obviously not unique but 
its restriction to $\cS$ is unique. The same holds for its inverse $i(s)$.}  
$s:[0,+\infty)\to[0,+\infty)$  with the following properties:

\begin{itemize}[leftmargin=*]
\item $\cS=s(\N)$
\item  there exists $i_*\geq 21$  such that
\begin{subequations}
\begin{equation}
\label{cicoria}
s(i)\geq e^{(\log i)^{1+\eta}}\,, \qquad
\forall i\geq i_*\,,\qquad
\mbox{for some given}\ \ \ 
1<\eta\leq 2\,;
\end{equation}
\begin{equation}\label{theshow}
s(i+i')\geq s(i)+s(i')\,,\qquad
s(h i)\geq h s(i)\,,\qquad
\forall\,h\geq 1\,,\ \ i, i'\geq i_*\,;
\end{equation}
\begin{equation}\label{mustgoon}
s(i^2)\geq s^2(i)\,,\quad\forall i\geq i_*\,.
\end{equation}
\end{subequations}
\end{itemize}
We denote by $i(s)$ the inverse function of $s(i).$

\smallskip
\noindent
b)\,  An infinite subset $\cS$ of  $\Z$, referred to as the {\it set of tangential sites}, is said to be {\it admissible}  if there exist a finite subset $\cS_0$ of $\Z$ and two admissible sets of tangential sites  $\cS_1,\cS_2\subset \N$ so that $\cS$ is one of the following subsets
$$
\cS_1,\qquad -\cS_1, \qquad (-\cS_1)\cup \cS_0,\qquad \cS_1\cup\cS_0 \quad (-\cS_1) \cup \cS_2\cup \cS_0\,.
$$
\end{defn}

\smallskip
\begin{rmk}\label{diego}
Definition \ref{orey} above  gives a quantitative control on how "sparse" the set $\cS$ should be.
In particular by \eqref{cicoria} the function $s(i)$ must grow faster than any polynomial.
\end{rmk}

\noindent
\textit{Examples of admissible tangential sites in $\N$}. $\cS := \set{ 2^i\,:\, i\in\N}$; choose\footnote{See \cite{BMP:linceisob}} $s(x) = 2^x$. Or, given $\eta >0$ as above, set $\cS := \set{[e^{(\log i)^{1+\eta}}]\,: i\in\N}$  ($[\cdot]$ being the integer part), and choose $s(x) = [e^{(\log x)^{1+\eta}}]$, for any $x>0$.


In what follows, {given a sequence indexed over $\Z$ we systematically decompose it over $\cS$ and $\cS^c:=\Z\setminus\cS$};
for example,
for the potential $V$ we write
$$
V=(\VS,\VSc)\,,\qquad \VS:=(V_j)_{j\in\cS}\,,\ \ \ \VSc:=(V_j)_{j\in\cS^c}\,.
$$
Analogously, we define the set of {\sl tangential frequencies} as
\begin{equation}
\label{panc}
\pan_\cS := \set{\nu = (\nu_j)_{j\in\cS}\in \R^\cS\, : \, \abs{\nu_j- j^2} < \frac12}\,.
\end{equation}
The cube 
$\pan_\cS$ inherits the product topology\footnote{Recall that if $X_i$, $i\in I$, are topological metrizable spaces, the product topology on $X:= \prod_{i\in I} X_i$ is the topology of the point-wise convergence, meaning that a sequence
$x^{(k)}=(x^{(k)}_i)_{i\in I}$
 converges iff $x^{(k)}_i$ converges for all $i\in I.$
  } and the product probability  measure ${\rm meas}_{\cQ_\cS}$
from $[-\nicefrac12,\nicefrac12]^\cS$
through the map
\begin{equation}\label{piramide}
\VSs:\pan_\cS\ \to \ [-\nicefrac12,\nicefrac12]^\cS\,,\qquad
\mbox{where}\qquad
\mathscr V^*_{\cS,j}(\nu):=\nu_j-j^2\,,\ \ \ j\in\cS\,.
\end{equation}
Moreover the above map  also endows  $\cQ_\cS$ with the 
$\ell^\infty$-metric, which induces a finer topology.
\\
Finally, we denote by $B_{r}(\tw_p)$ the open ball of radius $r$ centered at the origin of $\tw_p$, and   for $\rr>0$ we define 
define the \co{tangential actions}
\begin{equation}\label{pifferello}
\mathcal I(p,\rr):=\{ 
 		I\in { B}_{\rr^2}(\tw_{2p})\,: \, I_j=0\,\mbox{ for}\quad j\in\cS^c\,,\ \  I_j \geq 0\, \mbox{ for}\, j\in\cS
		\}\,.
 	\end{equation} 
We are now ready to state our main dynamical result regarding the existence of invariant KAM tori, parameterized by $I,\VSc$ and by the frequency $\nu$.
	\begin{them}\label{teo kam}
	Assume that $\cS\subset\Z$ is an admissible set of tangential sites,
		 $p>1$ and
	$\rr>0$ is sufficiently small.
	Moreover assume that
	 $$
	 I\in\cI(p,\rr)\quad\mbox{and}\quad  
	 \VSc=(V_j)_{j\in\cS^c}\in [-\nicefrac14,\nicefrac14]^{\cS^c}\,.
	 $$
	 Then there exists a Cantor-like set 
	 $\cC\subset\pan_\cS$ (recall \eqref{panc}) of positive  measure such that for any frequency $\nu\in\cC$, 
		there exists
		$$
		V_{\cS}=(V_j)_{j\in\cS}\in [-1, 1]^{\cS}
		$$ and a symplectic diffeomorphism $\Phi$ analytic on a small ball in $\tw_{p}$ such that,  for all  $\nu\in \cC$,  
		$$
		\cT_I=\{u\in \tw_{p}: |u_j|^2= I_j \quad \forall j\in \Z\}
		$$
		is { an invariant flat  torus} of frequency $\nu$ (recall Definition \ref{uccellone}) for the Hamiltonian vector field of
		$
		H_{V} \circ \Phi$,  with $V= (\VS,\VSc)$. 
		\end{them}		
		Theorem \ref{teo kam} above contains cubes of three different sidelengths, $ [-\nicefrac14,\nicefrac14]^{\cS^c}$, $\pan_{\cS}$ and $[-1, 1]^{\cS}$. While the choice of the lengths  $\nicefrac 12,1,$ and $2$ is rather arbitrary, the fact that we need three different scales is unavoidable, see Remark \ref{brut}.
	\\
	By Theorem \ref{teo kam}, equation\eqref{NLS} has invariant tori on which the dynamics is the linear translation by $\nu t$.
	The statement  above is a typical KAM Theorem, regarding the existence of an invariant torus. The fact that we are looking for an infinite torus however introduces various new difficulties, in particular related to the regularity of the dependence on $\nu,I,\VSc$. This appears immediately when one wishes to prove that the Cantor set $\cC$
	of ``good frequencies'' is measurable (and of positive measure) with respect to the product probability
	measure ${\rm meas}_{\cQ_\cS}$ on $\pan_\cS$. The $\sigma$-algebra 
	of such measure, which is the natural one in this context, is given by the Borel sets of the product topology,
	which is coarser than  the one induced by the $\ell^\infty$-metric.
	Then a crucial point is that a function $f:\pan_\cS\to \R$ which is  Lipschitz (with respect to the 
	$\ell^\infty$-metric) might   be non continuous with respect to the product topology and, hence, 
	non measurable. 
	 As typical in KAM schemes  $\cC$ is defined as the intersection of sets of the form 
	 $\{|f(\nu)|>\alpha>0\}$ (see \eqref{unbraccio} and \eqref{labirinto}) 
	 and $f:\pan_\cS\to \R$ is a Lipschitz function.
	 As explained above this does not assure 
	 measurability\footnote{This problem does not appear if $f$ actually depends only on a 
	 finite number of variables as in the case of maximal tori, where
	 $f(\nu)=\nu\cdot \ell$, with $|\ell|<\infty.$ }.
	 \\
	So we reformulate our theorem in a more technical way, carefully keeping track of the regularity w.r.t. all the parameters. To avoid working with functions defined only on a Cantor set, we 
	suitably (see Lemma \ref{godiva}) extend all the functions so that they are defined for $\nu\in \pan_\cS$.   To summarize, the parameter dependences we need to control are:
	\\
	- continuity w.r.t. the product topology for measure estimates both in $\nu$ and in $V$;\\
	- Lipschitz dependence (w.r.t. $\ell^\infty$-metric) for implicit function 
	 theorems /contractions/extensions
	(see Lemma \ref{zagana}).

The smallness condition  (i.e. how close are our solutions to zero) is tied to the size of  the non-linearity and to the regularity of the solution we are looking for. Recall that, since $f$ is analytic, for some $\mathtt R>0$ we have
\begin{equation}\label{analitico}
	f(y)= 
	\sum_{d=1}^\infty f^{(d)} y^d\,,\quad\qquad
	|f|_{\mathtt R}:=\sum_{d=1}^\infty|f^{(d)}|\mathtt R^d <\infty \,.
\end{equation}


\begin{them}\label{toro Sobolevp}
	Let $p_*>1$ and  $0<\g\leq \min\{\nicefrac14,\abs{f}_{\mathtt R}\}$.
	  There exists $\e_*=\e_*(p_*)>0$ and $C=C(p_*)>1$ such that,
	for all $\rr>0$
	satisfying 
	\begin{equation}
		\label{cornettonep}
		\e:=\frac{\abs{f}_{\mathtt R}}{\g \mathtt R} \rr^2 \le \e_*
			\end{equation}
	and for every 
	   \begin{equation}\label{diavola}
	   \frac{p_*+1}{2}\leq p\leq p_*
	   \end{equation}
	   the following holds.
 There exist:
 \\
 i) a  map 
 {
 \begin{equation}\label{labello}
 \begin{aligned}
 \VSf:\; &\cQ_{\cS}\times
 [-\nicefrac14,\nicefrac14]^{\cS^c}\times
  \mathcal I(p,\rr) \to  [-1,1]^{\cS},\\ & (\nu, \VSc, I)\mapsto \VSf(\nu,\VSc,I)
 \end{aligned}
 \end{equation} }
   which is continuous  in $\nu,\VSc$ w.r.t. the product topology and Lipschitz  in all its variables w.r.t.  the $\ell^\infty$ metric.
In particular $\VSf$ is  Lipschitz  $C\e\g$-close w.r.t. $\nu$ to the map
  $\VSs$; 
  	\\
ii)	a map {
	$$
\begin{aligned}	
	\Phi: \quad & {B}_{3\rr}\pa{\tw_{p}}\times\pan_\cS\times
	 [-\nicefrac14,\nicefrac14]^{\cS^c}\times
	 \mathcal I(p,\rr)  \longrightarrow
	{B}_{4\rr}\pa{\tw_{p}}\,,\,		\\
	 &  (u;\nu,\VSc,I) \quad \mapsto  \quad\Phi(u;\nu,\VSc,I)
	\end{aligned}
	$$}
	which is Lipschitz  in all its variables and $\rr/16$-close to the identity w.r.t.\footnote{Namely
$|\Phi(u;\nu,\VSc,I)-u|_p\leq \rr/16$ on 
${B}_{3\rr}\pa{\tw_{p}}\times\pan_\cS\times
	 [-\nicefrac14,\nicefrac14]^{\cS^c}\times
	 \mathcal I(p,\rr)$.	} $u$; \\
iii)
 a Cantor-like Borel set $\cC=\cC(\VSc,I,\gamma)\subset \cQ_{\cS}$,
	with\footnote{By ${\rm meas}_{\cQ_\cS}$ and ${\rm meas}_{[-\nicefrac14,\nicefrac14]^\cS}$
	we denote the  product probability  measures on $\cQ_\cS$ and $[-\nicefrac14,\nicefrac14]^\cS$,
	respectively. For brevity we write ${\rm meas}_{\cQ_\cS}(A)$ for
	${\rm meas}_{\cQ_\cS}(A\cap \cQ_\cS)$ and similarly for ${\rm meas}_{[-\nicefrac14,\nicefrac14]^\cS}$.} 
	\begin{equation}\label{mirto}
	\begin{aligned}
	&{\rm meas}_{\cQ_\cS}(\cQ_\cS \setminus \cC(\VSc,I,\g))
	 \le C_0\g\,,\qquad\\
	&{\rm meas}_{[-\nicefrac14,\nicefrac14]^\cS} \Big(\VSf\big(\cQ_\cS \setminus \cC(\VSc,I,\g),\VSc,I\big)\Big) \le C_0\g\,,
	\end{aligned}
	\end{equation} 
	for a suitable absolute constant $C_0>0$.
	\\
	 Moreover for any $I\in \mathcal I(p,\rr)$, $\VSc\in [-\nicefrac14,\nicefrac14]^{\cS^c}$ and 
	  $\nu\in \cC(\VSc,I,\gamma)$, 
	  the map $\Phi(\cdot;\nu,\VSc,I)$
		is an  analytic symplectic change of variables.		\\
	Finally, for  $\nu\in \cC(\VSc,I,\gamma)$, $\cT_I$ defined in \eqref{pippero}
		is { a KAM torus} of frequency $\nu$ for
		$
		H_{V} \circ \Phi$ with $V= (\VSf(\nu,\VSc,I),\VSc)$. 		
\end{them}
	\begin{rmk}\label{brut}
In 	Theorem \ref{toro Sobolevp} above  the map $\VSf$ has oscillations of order $C\e\gamma$ so the target cube could be $[-\nicefrac12-C\e\g, \nicefrac12+C\e\g]$, but in any case must have side-length greater than $1$. On the other hand, in order to achieve the second estimate in \eqref{mirto}, we need that the sidelength of $\pan_{\cS}$ is greater than $\nicefrac{1}{2}$. This is the reason for choosing three distinct sidelengths (see comments after Theorem \ref{teo kam}).
\end{rmk}
Note that the parameter $\gamma$ comes from a Diophantine condition, see Definition \ref{diomichela}, and consequently controls the measure estimates \eqref{mirto}. Of course, this estimate is meaningful only for small $\gamma$.\\
We now  reformulate our result  in terms of the Fourier multiplier $V$
(for the proof see Section \ref{bellogrosso}). 
\begin{cor}\label{ga}
For $I\in \cI(p,\rr)$, with $p,\rr$ as in the previous Theorem, let
	\begin{equation*}
		\cG(I,\g):= \set{ V\equiv(\VS, \VSc)\in [-\nicefrac14,\nicefrac14]^\Z\;: \quad \VS\in \VSf(\cC(\VSc,I,\g),\VSc,I) }\,.
	\end{equation*}
 $\cG(I,\g)$ is a Borel set in $[-\nicefrac14,\nicefrac14]^\Z$ with measure greater than $1- C_0\g$ ($C_0$ is the constant in \eqref{mirto}).
\end{cor}
\noindent
Therefore for all $I\in \cI(p,\rr)$ and all  $V\in  \cG(I,\g)$  the equation \eqref{NLS} has  an invariant torus\footnote{Of course, in the linear case one has $\Phi = \id, \mathscr{V}_\cS = \mathscr{V}^*_\cS,  \cC = \cQ_\cS, \cG = [-\frac14, \frac14]^\Z$.}.
%

\noindent
\textbf{Invariant tori and regularity of our
almost periodic solutions.}
The flat torus $\cT_I$ in the original variables is realized
by the analytic immersion 
$\varphi\mapsto\Phi(\mathfrak i(\varphi);\nu,\VSc,I )$
with 
 $\mathfrak i$
defined in 
	\eqref{manodedios}.
By construction the NLS dynamics on the torus $\cT_I$ is $\varphi \mapsto \varphi+ \nu t$ so our candidate for an almost periodic solution is $\ur(t,\cdot):= \Phi(\mathfrak i(\nu t);\nu,\VSc, I )$ as discussed in  Section \ref{prove}. 
We wish to stress that analyticity of $\Phi(\mathfrak i(\varphi))$ in the angles does not imply analyticity in time,
	since the map $ t\mapsto \nu t\in \T^\cS$ is not even continuous (endowing $\T^\cS$ with the $\ell^\infty$-topology). 
	On the other hand, since $\Phi$ is analytic,
	{\it the regularity issues are the same as 
	in the linear case}, 
	recall Subsection \ref{armando}. In particular,   for $p>3$ our solutions are classical, while for $p\le 2$ we can construct also merely weak solutions, see the proof of Theorem \ref{deltoide}.
	
\noindent
\textbf{The Cantor set $\cC$}.
We can be rather explicit in our description of the set $\cC$ of Theorem \ref{toro Sobolevp}. We start by 
fixing the hypercube 
\begin{equation}
\label{pan}
\pan := \set{\omega = (\omega_j)_{j}\in \R^\Z\, : \, \abs{\omega_j- j^2} < \frac12}\, , \qquad \cQ=\cQ_{\cS}\times \cQ_{\cS^c}\,,
\end{equation}
endowed with the product topology,
and by introducing the following closed set
\begin{defn}[Diophantine condition] \label{diomichela} Let $\g>0$.  We say that a vector $\omega\in\pan$ belongs to ${\mathtt D}_{\g,\cS}$ if it satisfies\footnote{Instead of $3/2$
one can put every exponent $\tau>1$.}
\begin{equation}\label{diofantino nuc}
\begin{aligned}
&|\omega\cdot \ell| \ge  
	\gamma \prod_{s\in \cS}\frac{1}{(1+|\ell_s|^2 \langle i(s)\rangle^2)^{3/2}}, \quad \\
&	 \forall  \ell : 0<|\ell|<\infty \,,
	\sum_{j\in \cS^c}|\ell_j|\le 2 \,,\quad \pi(\ell)= \fm(\ell)=0\,,
\end{aligned}
	\end{equation}
	where\footnote{As usual for integer vector $\ell\in\Z^\Z$ we set $|\ell|=\sum_{j\in\Z}|\ell_j|$.} $i(s)$ is the inverse function of $s(i)$, $\pi(\ell) :=\sum_{j\in \Z} j\ell_j$ is the ``momentum'' and $\fm(\ell):=\sum_{j\in \Z} \ell_j$ is the ``mass''. 
\end{defn}
\begin{them}\label{torobolev}
	Under the hypotheses of Theorem \ref{toro Sobolevp},
	there exists a Lipschitz map
	$\bO: \cQ_\cS\times [-\nicefrac14,\nicefrac14]^{\cS^c} \times \mathcal I(p,\rr) \to \cQ_{\cS^c}$
	which is continuous with respect to the product topology on 
	$\cQ_\cS \times [-\nicefrac14,\nicefrac14]^{\cS^c}$
	and\footnote{The constant $C$ is the one of Theorem \ref{toro Sobolevp}.}
	for every $j\in \cS^c$ satisfies
	\begin{equation}
	\label{omegone}
	 |{\bO_j(\nu, \VSc, I)-j^2-V_j}| 
	\le C \g \e
	\end{equation}
	and the Lipschitz estimates
	\begin{equation}\label{batacchio}
	\begin{aligned}
& \sup_{\nu'\neq\nu}\frac{|\bO_j(\nu, \VSc, I)-\bO_j(\nu', \VSc, I)|}{|\nu-\nu'|_\infty}
	\le C  \e\,,\qquad\\
	&
\sup_{I'\neq I}\frac{|\bO_j(\nu, \VSc, I)-\bO_j(\nu, \VSc, I')|}{ \abs{I - I'}_{2p}}\leq
	 C      \g \rr^{-2}\e\,.
	 	\end{aligned}
\end{equation}	
	Moreover we can choose in Theorem \ref{toro Sobolevp}
	\begin{equation}\label{arancini}
	\begin{aligned}
	&\cC(\VSc,I,\gamma):= \{ \nu\in \cQ_{\cS}: \,\omega(\nu, \VSc, I)\in {\mathtt D}_{\g,\cS} \}\,, \quad\\
	&\mbox{where } \omega(\nu, \VSc, I):=(\nu,\Omega(\nu, \VSc, I))\,.
	\end{aligned}
	 \end{equation}
	 In this way, the torus $\cT_I$ defined in \eqref{pippero} is an
	   elliptic invariant torus in the sense that its linearized  dynamics in the ``normal'' directions 
	    is  $\dot{u}_j=\im\Omega_j u_j$ for $ j\in\cS^c$. 
\end{them}
\begin{rmk}\label{tana}
	Note that \eqref{diofantino nuc} is a much stronger Diophantine condition that the one proposed in \cite{Bourgain:2005} (or \cite{BMP:almost}), where the denominators were of the form $1+|\ell_j|^2 j^2$. Of course the reasons why we can impose such strong diophantine conditions, still 
	obtaining a positive measure set, are the structure of the set $\cS$ and the fact that we only need to consider denominators with $\sum_{j\in \cS^c} |\ell_j|\le 2$. 
\end{rmk}
\subsection{Plan of the paper, strategy and main novelties}
Deducing  Theorem \ref{deltoide} from Theorems \ref{toro Sobolevp} and \ref{torobolev} is a self contained argument, {which we present in section \ref{prove}. The proof  is developed as follows:  $\VSc$ and
$I\in \mathcal I(p,\rr)$ being fixed, we first construct a set $\cC'(\VSc,I,\g)\subseteq \cQ_\cS$ of large relative measure such that for all $\nu\in \cC'(\VSc,I,\g)$  the map 
$t\mapsto \ur(t,\cdot):= \Phi(\mathfrak i(\nu t);\nu,\VSc, I )
\in \tw_p$ is almost-periodic in 
$\tw_{p'}$ with $1<p'<p$
since it is uniform limit of  quasi periodic functions
just as in the linear case (see Subsection \ref{armando}). Such  approximating functions are in fact classical solutions of  the approximate equation $($NLS$_{V_n})$   with $V_n = V_n(\nu)\to V(\nu)$.  Secondly we show that $\ur$ is a weak solution in the sense of \eqref{ranefritte},  for such $V = V(\nu)$.  
We next reformulate our result in terms of the external parameters $V$ instead of the frequencies $\nu$,  namely we prove that for every $V$ in a large measure set $\mathcal G(I,\gamma)\subseteq [-\nicefrac14,\nicefrac14]^\Z $ we can solve \eqref{ranefritte}. A main point is to show that $\mathcal G(I,\gamma)$ is measurable. Next, taking the union over $\g$ we obtain a full measure set, thus showing
 that for a.e. $V$ there exists at least one solution. Finally moving $\cS$
 and suitably choosing $I$ we produce, for a.e. $V$, countably many different  solutions. }
 \smallskip
 
 In order to prove Theorems \ref{toro Sobolevp} and \ref{torobolev} (see end of Section \ref{cinque}) we follow
 the general strategy of  \cite{BMP:almost}. 
{As recalled in Section \ref{cinque} this consists in three main steps:  the introduction of a suitable functional setting and degree decomposition, a formulation in terms of a counter-term theorem, solving a Homological equation, and proving  the fast convergence of a KAM scheme.  } 
Regarding the first issue, the main additional difficulty in the present setting is to keep track of the regularity w.r.t. all the parameters. To this purpose
 in section \ref{tre}
 we  start by defining  the Poisson algebra of ``regular Hamiltonians''
(basically a space of normally analytic functions $H$ such that the corresponding Hamiltonian vector field $X_H$  is a normally analytic 
map
$B_{r}(\tw_p)\to \tw_p$) and further introduce the space of parameter--depending
regular Hamiltonians (see definitions \ref{lamai} and \ref{ala}). 
{In particular, we set our ambient space as the closure, w.r.t. a suitable norm (see \eqref{solenoide} together with \eqref{entusiasmo}-\eqref{gorma}), of functions depending only on a finite number of frequencies. A Lipschitz extension result is given in Lemma \ref{godiva}.}

{Theorems \ref{toro Sobolevp}-\ref{torobolev} will be derived in two steps,  through a technique known as ``elimination of parameters",  firstly introduced by R\"ussmann and Herman in the 80's.  This strategy consists in introducing some extra parameters (i.e.  the counter-terms $\lambda_j$) in the perturbed Hamiltonian, in order to  compensate its degeneracies (absence of twist properties for instance) and conjugate the modified Hamiltonian to one that admits the desired invariant KAM torus.  This first normal form step contains the hard analysis. Then,  we prove that we can \textit{eliminate} the counter-terms  using the potential,  by  means of the classical implicit function theorem in Banach spaces.  \\ Hence,  in our case,  we first put $H_V$ in \eqref{Hamilto0} in a suitable normal form with counter-terms in the spirit of Herman  through Theorem \ref{allaMoserbis},  {which contains the main KAM difficulties in dealing with Sobolev regularity. }  }

Roughly speaking,  this normal form theorem states that,  under appropriate smallness conditions,  for all $\omega\in \dgp$,  $\VSc\in [-\nicefrac14,\nicefrac14]^{\cS^c}$ and $I\in \cI(p,r)$ there exist $\lambda\in \ell^\infty$ and a symplectic  change of variables $\Psi$
(see \eqref{pocafede})  such that
\[
( \sum_{j\in\Z} (\omega_j+\lambda _j) |u_j|^2+P)\circ\Psi = \sum_{j\in\Z} \omega_j |u_j|^2+ O(|u|^2 -I)^2,
\]
where  $R= O(|u|^2 -I)^2$ means that $R$ is a regular Hamiltonian which has a zero of order at least 2 at the torus $\cT_I$ defined in \eqref{pippero} (we formalize this definition in Section \ref{proiezioni},  introducing an appropriate degree decomposition).  In order to deduce Theorems \ref{toro Sobolevp}-\ref{torobolev} we need to eliminate the counter terms (see equations in \eqref{equaOme}): the main point is that we can solve with respect only to the tangential variables $V_{\cS} \leftrightsquigarrow \nu$.

Coming back to the counter-term Theorem \ref{allaMoserbis},  let us explain the main new issues related to Sobolev regularity.\\
As is habitual,  the map $\Psi$ is constructed as the composition of a sequence of changes of coordinates  whose generating function $S$ at every step   is determined by solving a Homological equation 
\begin{equation}
\label{homolu}
L_\omega S:=\{\sum_j \omega_j|u_j|^2, S\} = F\,,
\end{equation}
where  $F$ is a given analytic function on the phase space $\tw_p$ (see definition \ref{ala}) which is at most quadratic in the normal variables ($j\notin \cS$).
{At a formal level,  a solution $S=L_\omega^{-1} F$ of \eqref{homolu}  is readily determined.  In order to prove that  $S$ is in fact analytic,  one has to control the contribution given by small divisors (i.e.  the eigenvalues of $L_\omega$). This is possible by imposing the arithmetic Melnikov conditions on the frequencies \eqref{diofantino nuc},  where the constraint $\sum_{j\in\cS^c}|\ell_j| \le 2$  comes from the fact that $F$  is at most quadratic in the normal variables,  while the 
zero mass and momentum conditions come from the presence of the corresponding quadratic constants of motion.}\\
{
Also in this Sobolev context,  due to the presence of small divisors,  one bounds the solution $S$ at the cost of some ``loss of information".  More explicitly,
 if the Hamiltonian vector field $X_F$ maps $B_{r}(\tw_p)\to \tw_p$,  then   $X_S$  maps $B_{r}(\tw_{p+\delta})\to \tw_{p+\delta}$ (with $\delta>0$); this means that $S$ is analytic  in a \co{smaller} domain,  since $\tw_{p+\delta}\subset \tw_p$.   Then,  at each iteration,  one is able to define the solutions $S$ only on a (ball of a) smaller phase space of the Banach scale $(\tw_p)_{p>1}$,  the target space shrinking accordingly. } Resembling the finite dimensional case, we call $\delta$ the ``loss of regularity "

{Of course,  the convergence of a KAM scheme is achieved only if one loses a {\sl summable, amount of regularity} $\delta_n$ at each step $n$.
Nonetheless the typical characteristic of  the Sobolev case is the presence of a {\it  lower bound } on the loss of regularity,  so that a KAM scheme based only on such bound cannot converge.  
In KAM schemes for quasi-periodic Sobolev solutions a similar problem arises but it is bypassed by using tame estimates, or approximation by analytic functions.  In the present context
it is not clear if one can exploit similar ideas,
  since we are already working with analytic nonlinearities.}
{Our point of view is to bypass this problem by taking full advantage of the fact that the tangential sites set $\cS$  is sparse (recall Definition
\ref{orey}). 
In order to get an intuition  of our strategy
take for simplicity $\cS=\{s(i)=2^i\,,\ i\in\N\}$
 and  start by considering  a toy model  with a non-linearity for which the set
\begin{equation}\label{tobia}
\cU_\cS:=\set{u\in \tw_p:\quad u_j=0 \quad \forall j\notin\cS}
\eqsim\{v\in\ell^\infty\ \mbox{s.t.}\  \sup_{i\in\N}|v_i|2^{ip}<\infty\}
\end{equation}
is invariant for the dynamics; so that we may study the equation restricted to $\cS$. 
Then we are essentially in the analytic case
(or Gevrey or slightly less if we take a slower growth for $s(i)$)
and the Bourgain strategy in \cite{Bourgain:2005} (or \cite{BMP:almost}, \cite{Co})
applies\footnote{Note that Bourgain's 
 Diophantine condition reads $|\omega\cdot\ell|\ge \g\prod_{i\in\N}(1+ \ell_i^2 \jap{i}^2)^{-1}$. This coincides with \eqref{diofantino nuc} when $\ell$ is supported only on $\cS$, modulo the renaming of the 
indexes in \ref{tobia}}. 
Of course for the NLS equation \eqref{NLS} the set $\cU_\cS$ is not invariant and the main difficulties  arise from interaction between tangential  and normal modes.}

In the Diophantine estimates, to deal with the terms $\ell$ not  supported only on the tangential sites, we use the constants of motion and the dispersive nature of the equation ($\omega_k\sim k^2$).
Once one has guessed the correct Diophantine conditions \eqref{diofantino nuc}, the proof of  Proposition \ref{mah} (i.e. controlling the solution of the Homological equation \ref{homolu}) is the real core of  our result. 
Again the proof  is simple if $F$ is supported
only on $\cS$ or $\cS^c$, on the other hand dealing with the interaction 
between tangential and normal sites requires a careful case analysis. 
\\
The final goal is to control the norm of $X_S$  as a map $B_{r}(\tw_{p+\delta})\to \tw_{p+\delta}$, for arbitrarily small $\delta$.
 This should be compared with the corresponding estimate on $X_S$ in \cite{BMP:2019} Proposition 7.1
item ($\mathtt M$). In the latter paper we take $\cS=\Z$ and then, in order to control $L_\omega^{-1} F$ we cannot take any $\delta>0$ but instead must require $\delta \ge \delta_*$, where $\delta_*>0$ is some fixed quantity. As one can expect the less {\it sparse} is $\cS$ the worst bounds one gets.
The quantitative condition in Definition \ref{orey} is needed in order to ensure convergence of the iterative KAM scheme.
We suitably choose the values of the parameter at each iterative step $n\in\mathbb N$, 
in particular the loss of regularity
$\delta_n$ must be summable, e.g. $\delta_n\sim n^{-c},$ for some $c>1$.
Then   the divergence due to small divisors, which is of order
$\exp(\exp(n^{c/\eta}))$ by \eqref{ss148} and with $\eta$ defined in \eqref{cicoria},
 must be compensated by 
the super-exponential convergence $\exp(-\exp(C n))$ 
given by the KAM quadratic scheme. This forces $c<\eta$ and  $\eta>1$.  
\\
The super-linearity assumptions \eqref{theshow}
 and 	\eqref{mustgoon} are essential for our estimate on the Homological equation to work.
 The asymptotic growth in \eqref{cicoria} is only needed in the KAM step.
 A slower growth would give rise to a too large estimate on the solution
 of the Homological equation which would not be compensated anymore
 by the quadratic convergence of the KAM scheme.

\vskip10pt

\section{Proof of Theorem \ref{deltoide}}\label{prove}

The solutions of Theorem \ref{deltoide} are constructed as the limit of sequences of smooth quasi-periodic functions.
Fix  $p_*>1$ and an admissible set of tangential sites $\cS$. Fix $\gamma>0$ and take $\rr$ such that \eqref{cornettonep} holds. 
For any  potential $\VSc\in [-\nicefrac14,\nicefrac14]^{\cS^c}$ and   $I\in\mathcal I(p_*,\rr)$, 
we  apply Theorem \ref{toro Sobolevp} and obtain, for all frequencies $\nu\in \cC(\VSc,I,\g)$, an  invariant torus.   
\\
Now   
	define $I^{(n)}=\pa{I^{(n)}_j} $ by setting $I^{(n)}_j=I_j$ if $|j|\le n$ and $I^{(n)}_j=0$ otherwise.
We apply Theorem \ref{toro Sobolevp}   with
	$I\rightsquigarrow I^{(n)},$ 
	$\gamma\rightsquigarrow\gamma/2$
	and set 
	$\mathscr V^{(n)}(\cdot):=\big(\VSf(\cdot,\VSc,I^{(n)}),\VSc\big)$  
	a $\Phi_n(\cdot;\cdot):=\Phi(\cdot;\cdot,\VSc, I^{(n)})$.
We have obtained a sequence of NLS equations with potentials $\cV^{(n)}(\nu)$
for $\nu\in [-\nicefrac14,\nicefrac14]^{\cS}$;
each equation admits a finite dimensional invariant torus with frequency $\nu$, for all  
$\nu \in \cC(\VSc,I^{(n)},\gamma/2)$. It is not hard to see that each torus supports a smooth quasi-periodic solution.

The idea is to show that (at least up to a sub-sequence)  the limit over $n$ is the desired almost-periodic solution.

\smallskip

\noindent
\co{First step (construction of the Cantor set).} {In order to apply   Theorem \ref{toro Sobolevp} for each $n$ (with the same frequency)  we take the ``good frequencies"  in the (countable) intersection of the Cantor-like sets where all the tori are defined. Correspondingly, we define the set of ``good potentials". This is the content of the following result, which is proved in Section \ref{bellogrosso}.}
	
	\begin{lemma}\label{culonembo}	
	There exists a sub-sequence $n_k\to\infty$ (independent of $\VSc$) such that the following holds. 
Defining the Borel sets
		\begin{eqnarray*}
		\cC'(\VSc,I,\gamma)&:=&\cC(\VSc,I,\gamma)\bigcap_{k\in\N}  
	\cC(\VSc,I^{(n_k)},\gamma/2)\,,
\\
\mathcal G_\cS(\VSc,I,\gamma)
&:=&
[-\nicefrac14,\nicefrac14]^\cS\ \bigcap\ 
	\VSf\big( \cC'(\VSc,I,\g),\VSc,I\big)\,,
	\\
\mathcal G(I,\gamma)
&:=&
\{ V=(\VS,\VSc)\in[\nicefrac14,\nicefrac14]^\Z\ \mbox{s.t.}\ \VS\in\mathcal G_\cS(\VSc,I,\gamma)  \} \,,
\label{Sborel}
\end{eqnarray*}
we have the estimates
	\begin{equation}\label{giuncata}\begin{aligned}
	&{\rm meas}_{[-\nicefrac14,\nicefrac14]^\cS} \Big(\mathcal G_\cS(\VSc,I,\gamma)\Big) \geq 1-C_0\gamma
	\,,\qquad\\
&{\rm meas}_{[-\nicefrac14,\nicefrac14]^\Z} \Big(\mathcal G(I,\gamma)\Big) \geq 1-C_0\gamma\,,
\end{aligned}
\end{equation}
where $C_0$ is defined in Theorem \ref{toro Sobolevp}.
	\end{lemma}
\noindent
	\co{Second step (construction of the convergent sub-sequence).} We now prove that for every $V\in \mathcal G(I,\gamma)$ 
	we can solve \eqref{ranefritte} with $V=(\VS,\VSc)$.
Indeed	for  every $\VS\in \mathcal G_\cS(\VSc,I,\gamma)$
there exists 
$\nu\in \cC'(\VSc,I,\g)$ such that
$\VSf\big(\nu,\VSc,I\big)=\VS.$
We claim that
	 \begin{equation}\label{legna}
 \ur(t,x):=\Phi(v(t,x);\nu,\VSc,I)\,,\quad
 \mbox{where}\quad 
 v(t,x):=
\sum_{j\in\cS} \sqrt{I_j} e^{\im (j x+\nu_j t)}\,,
\end{equation}
satisfies \eqref{ranefritte}
and, moreover, it is the uniform limit of the $C^\infty$-quasi-periodic functions
	$$
	 \ur_k(t,\cdot):=\Phi_{n_k}(v_k(t,\cdot); \nu)
	\qquad
	\mbox{where}\qquad 
	v_k(t,x):=\sum_{j\in\cS, |j|\leq n_k} \sqrt{I_j} e^{\im (j x+\nu_j t)}
	\,.
	$$
	Note that by construction $v,\ur\in \tw_{p_*}$ while
	 $t\mapsto\Phi_{n_k}(v_k(t,\cdot);\nu)=:\ur_k(t,\cdot)$ is a classical
	(actually $C^\infty$)
	quasi-periodic solution of 
	(NLS$_{V^{(n_k)}}$) with $V^{(n_k)}=\big(\mathscr V^{(n_k)}(\nu),\VSc\big)$.
	Thus  each $\ur_k$ satisfies  \eqref{ranefritte} with $V^{(n_k)}$
	in place of $V$.
	\\
	Taking
	$p<p_*$ satisfying  \eqref{diavola}
	 we get
	$$
	c_k:=\sup_{|j|>n_k}\sqrt{I_j}|j|^{p}\ \to\ 0\,,\qquad
	\mbox{as}\ \ \ k\to\infty
	$$
	and, for every $t\in\R,$
	$$
	|v(t,\cdot)-v_k(t,\cdot)|_{p}=c_k\,,\qquad
	|I-I^{(n_k)}|_{2p}=c_k^2\,.
	$$
	Since the map $I\to \VSf(\nu,\VSc,I)$ is Lipschitz
	then $V^{(n_k)}\to V$ in $\ell^\infty$.
	Moreover since $\Phi$ is also Lipschitz (w.r.t. $u$ and $I$), 
	 for every $t\in \R$ 
	$$
	\begin{aligned}
	| \ur(t,\cdot)- \ur_k(t,\cdot)|_{p} &=& | \Phi(v(t,\cdot);\nu,\VSc,I) - \Phi(v_k(t,\cdot);\nu,\VSc,I^{(n_k)})|_{p} 
\\&	\le& L\big(|v(t,\cdot)-v_k(t,\cdot)|_{p}+ 
	|I-I^{(n_k)}
	|_{2p}\big)
	\leq L(c_k+c_k^2)
	\end{aligned}
	$$
	for a suitable\footnote{We can take $L=2$ since $\Phi$  is $C\e$-close to the identity.} $L>0.$
	Then  $\ur_k\to \ur$ uniformly\footnote{Given $f:\R^2\to\C$ we have
	$\|f\|_{L^\infty(\R^2)}=\sup_{t\in\R}\|f(t,\cdot)\|_{L_x^\infty(\R)}
	\leq \sup_{t\in\R}|f(t,\cdot)|_{p}$ for every $p\geq 1.$} in $\R^2$.
	In order to prove
	  that
	$\ur$ satisfies \eqref{ranefritte} we 
	have to show that
	$$
	\int_{\R^2} (-\im \chi_t + \chi_{xx})(\ur-\ur_k) - (V\ast \ur -V^{(n_k)}\ast \ur_k)\chi
	-
 	\big(f(|\ur|^2)\ur-f(|\ur_k|)\ur_k\big)\chi\, dx\, dt\,
	$$
	tends to zero as $k\to \infty$.
	This follows since $\ur_k\to \ur$ uniformly and observing that
	\begin{eqnarray*}
	&&\|V^{(n_k)}\ast \ur_k-V\ast \ur\|_{L^\infty(\R^2)}
	\leq 
	|V^{(n_k)}\ast \ur_k-V\ast \ur|_{p}
	\\
	&&
	\leq
	|V^{(n_k)}-V|_{\ell^\infty}| \ur_k|_{p} +|V|_{\ell^\infty}| \ur-\ur_k|_{p}
	\ \stackrel{k\to\infty}\longrightarrow\ 0\,.
\end{eqnarray*}	
This proves the claim in \eqref{legna}.

\smallskip

\noindent
\co{Third step (A set of good potentials).} We now show that
for almost every $V\in[-\nicefrac14,\nicefrac14]^\Z$ there exists at least one solution of \eqref{ranefritte}.
For every integer $h\geq 1$  take $\gamma_h=\abs{f}_{\mathtt R}/h$  and  $\rr_h$ such that \eqref{cornettonep} holds as an equality, namely
$\rr_h:=  \sqrt{\e_*\g_h \mathtt R/\abs{f}_{\mathtt R}}=
\sqrt{\e_* \mathtt R/h}$. For every given sequence $I_h\in\cI(p_*,\rr_h)$
	we set
	\begin{equation}\label{acquaazzurra}
	\mathcal G:=\bigcup_{h\geq 1} \mathcal G(I_h,\gamma_h)\,.
	\end{equation}
By \eqref{giuncata}
$\mathcal G$
has full measure in $[-\nicefrac14,\nicefrac14]^\Z$.
This implies that for almost every $V\in[-\nicefrac14,\nicefrac14]^\Z$
there exists an integer $h\geq 1$ such that $V\in \mathcal G(I_h,\gamma_h).$
Then  \eqref{legna} (with $I=I_h$) gives a solution of \eqref{ranefritte}.

\smallskip

\noindent
\co{Fourth step (Abundance of solutions).}
In order to find infinitely many solutions for almost every $V$ in $[-\nicefrac14,\nicefrac14]^\Z$ 
 we proceed as follows.
 First we choose in \eqref{acquaazzurra}
 $\sqrt{(I_h)_j}:=\frac12
 \rr_h\jjap{j}^{-p_*}$.
 All the above construction depends on the choice of the set $\cS$ of admissible tangential sites; in particular this holds
 for the set $\mathcal G$ above. Let us now consider two distinct admissible sets $\cS$ and $\cS'$, and let us denote by $\cG,\cG'$ the corresponding sets of potentials. We claim that for each $V\in\cG\cap \cG'$ there exists (at least) two distinct almost-periodic solutions. 
Indeed, there exist $h,h'$ such that
 $V\in \mathcal G(I_h,\gamma_h)\cap \mathcal G'(I_{h'},\gamma_{h'})$,
 where $\mathcal G(I,\gamma),\mathcal G'(I,\gamma)$  are the sets defined in
Lemma \eqref{culonembo} corresponding to $\cS, \cS'$.
 Let us call $\ur,v,\Phi$ and $\ur',v',\Phi'$ the functions defined in \eqref{legna}, respectively
 for $\cS$ and $\cS'$. 
 Since by Theorem \ref{toro Sobolevp} (point {\it ii)}) the maps $\Phi$ and $\Phi'$ are
  $\bar\rr/16$-close to the identity where $\bar\rr:=\max\{\rr_h,\rr_{h'}\}$,
 then
  $$
  |\ur(t,\cdot)-\ur'(t,\cdot)|_{p_*}\geq  |v(t,\cdot)-v'(t,\cdot)|_{p_*} -
  \bar\rr/8\geq 
 \bar\rr/2-\bar\rr/8>0\,.
  $$
The same holds for any translations $\ur(t+t_0,x+x_0)$ and $\ur'(t,x)$. 
 Let us now consider a countable infinity of  distinct admissible sets $\cS$, and call $\mathcal G_*$
 the countable intersection of the corresponding sets $\mathcal G$. 
 The set $\mathcal G_*$ has full measure in $[-\nicefrac14,\nicefrac14]^\Z$ as well.
 Then for every $V\in \mathcal G_*$ we construct infinitely many solutions corresponding to different $\cS$'s.

\smallskip

\noindent
\co{Regularity}.
By \eqref{legna} and  \eqref{katsendorff} we know that if $p_*>3$ then $v_t$ and $v_{xx}$ are continuous functions on $\R^2.$
Moreover, by analyticity of $v\to\Phi(v;\nu,I)$, also $\ur_t$ and $\ur_{xx}$ are continuous function on $\R^2.$
Therefore $u$ is a classical solution.
On the other hand, when $p_*\leq 2$ and $\sqrt{I_j}=(\rr/2)\jap{j}^{-p_*}$ for all $j\in\cS$, for every $t$  {\sl the function $\ur_{xx}(t,\cdot)$
is not in $L^2$.} Otherwise its Fourier coefficients $(\ur_{xx}(t,\cdot))_j=- j^{2}(\ur(t,\cdot))_j$ would belong to $\ell_2$.
As in the linear case (recall \eqref{piscia})
we claim that
$$
\limsup_{|j|\to	\infty}
j^{2}|(\ur(t,\cdot))_j|>0\,.
$$
	Indeed, since $\Phi$ is close to the identity (in $\tw_{p_*}$) $\Phi=Id+\Phi_1$
		with $\|\Phi_1\|$ uniformly $\rr/16$-small, we have that   
		$\ur(t,\cdot):=\Phi(v(t,\cdot))=v(t,\cdot)+\Phi_1(v(t,\cdot))$ 
		and its
		 Fourier coefficients satisfies
		$|(\ur(t,\cdot))_j|\geq (\rr/4)\jap{j}^{-p_*}$  for all $j\in\cS$;
		since $\cS$ is an infinite set this is a contradiction. 
			 \qed

\section{Functional setting}\label{tre}

\noindent
Let us introduce the spaces of Hamiltonians used in the paper.

\noindent
\begin{defn}[Multi-index notation]\label{baccala}
  In the following we denote, with abuse of notation, by $\N^\Z$ the set of 
   multi-indexes $\bal,\bbt$ etc. such that
  $|\bal|:=\sum_{j\in\Z}\bal_j$ is finite.
  As usual $\bal!:=\prod_{j\in\Z,\, \bal_j\neq 0}\bal_j!$.
  Moreover $\bal\preceq\bbt$ means $\bal_j\leq\bbt_j$
  for every $j\in\Z$, then $\binom{\bbt}{\bal}:=\frac{\bbt!}{\bal!(\bbt-\bal)!}.$
   Finally take $j_1<j_2<\ldots<j_n$
  such that $\bal_j\neq 0$ if and only if $j=j_i $ for some $1\leq i\leq n$,
  as usual we set $\partial^{\bal} f:=\partial^{\bal_{j_1}}_{u_{j_1}}
  \ldots \partial^{\bal_{j_n}}_{u_{j_n}} f
 \,;$
  analogously for $\partial_{\bar u}^\bbt f.$
  \end{defn}
  
\smallskip
\begin{defn}[regular Hamiltonians]\label{Hr} 
	Consider a formal
	 power series expansion
	\begin{equation}\label{mergellina}
H(u)  = \sum_{(\bal,\bbt)\in \cM }H_{\bal,\bbt}u^\bal \bar u^\bbt\,,
	\qquad
	u^\bal:=\prod_{j\in\Z}u_j^{\bal_j}\,,
\end{equation}
where
\begin{equation}\label{zorro}
\cM:=\left\{
(\bal,\bbt)\in \N^\Z\times\N^\Z\,, \ \ {\rm s..t.}\ \ 
|\bal|=|\bbt|<+\infty\,,\ \ 
\sum_{j\in \Z} j (\bal_j-\bbt_j)=0
\right\},
\end{equation}
satisfying the reality condition
		\begin{equation}\label{real}
		H_{\bal,\bbt}= \overline{ H}_{\bbt,\bal}\,, \qquad
		\forall\, (\bal,\bbt)\in \cM\,.
		\end{equation}
	We say that $H\in \scH_{r,p}$ for $p>1,$ $r>0$ if
	\begin{equation}\label{panettone}
	 \abs{H}_{r,p}
 :=
  \frac1r \pa{\sup_{\norm{u}_{p}\leq r} \norm{{X}_{\underline H}}_{p} } 
 <\infty\,,
	\end{equation}
	where ${X}_{\underline H}$ denotes the Hamiltonian vector field,  intrinsically defined through the standard symplectic form $\varpi = \im\sum_j du_j\wedge d\bar{u}_j$ as
  $$
  i_{X_{\underline H}}\varpi(\cdot) := \varpi(\cdot, X_{\underline H}) = d\underline H(\cdot),
  $$
  $\underline H$ being the associated majorant Hamiltonian  $\underline H(u):=
	\sum_{(\bal,\bbt)\in \cM }|H_{\bal,\bbt}|u^\bal \bar u^\bbt$ of $H(u)$, and $i_{X_{\underline H}}$ being the usual contraction of a differential form with a vector field. 

\smallskip

We denote by $\scH_{r,p}(\C)$ the space of $H$ satisfying  \eqref{mergellina} 	
and \eqref{panettone} but not necessarily the reality condition \eqref{real}. \\
Finally, given any $F, G \in \scH_{r,p}(\C) $, we define in the usual manner the Poisson brackets as {$\set{F,G} := \varpi(X_F,X_G)$}.\\
 For details on the symplectic structure and Hamiltonian vector field see for instance \cite{BMP:2019}.
\end{defn}	
\noindent
	Note that by Lemma 2.1 of \cite{BMP:almost}
		\begin{equation}\label{norma1}
 \abs{H}_{r,p}
 =
\frac12  \sup_j  
\sum_{(\bal,\bbt)\in \cM} \abs{H_{\bal,\bbt}}\pa{ \bal_j + \bbt_j}u_p^{\bal + \bbt - 2e_j}
\,,
 \end{equation}
  where $u_p=u_p(r)$ is defined as
 \begin{equation}\label{giancarlo}
 u_{p,j}(r):= r  \jjap{j}^{-p}  \,.
 \end{equation}

\begin{rmk}
Regarding $\cM$ in \eqref{zorro} we note that the condition 
$|\bal|=|\bbt|$, i.e. $\fm(\bal-\bbt)=0$, corresponds to mass conservation,
namely 
 the $H$ Poisson commutes with the { \sl mass} $\sum_{j\in \Z}|u_j|^2$;
 moreover
 $\sum_{j\in \Z} j (\bal_j-\bbt_j)=0$, i.e.   $\pi(\bal-\bbt)=0$, corresponds to
	 momentum conservation, 
	 namely $H$ Poisson commutes with the 
		{ \sl momentum} $\sum_{j\in \Z}j|u_j|^2$.
\end{rmk}

\noindent
	Note that 
$\abs{\cdot}_{r,p}$ is a semi-norm
on  $\scH_{r,p}$ and a norm on its subspace
\begin{equation}\label{lenticchia}
 \scH_{r,p}^0:=
 \{\
		H\in\scH_{r,p}\ \ {\rm with}\ \ 
H(0)=0\ \}\,,
\end{equation}
endowing $\scH_{r,p}^0$ with a Banach space structure.
Moreover the space $\scH_{r,p}$ enjoys the following algebra property with respect to Poisson brackets of Hamiltonians.
	\begin{prop}[Poisson structure]\label{fan}
For any $F, G \in\scH_{r + \rho,p}$, with $\rho>0$,  one has
$
 \{F,G\}\in\scH_{r,p}^0\,$
with the bound
\begin{equation}\label{menate}
	\norm{\{F,G\}}_{r,p}
	\le 
	8\max\set{1, \frac{r}{\rho} }
	\norm{F}_{r+\rho,p}\norm{G}_{r+\rho,p}\,.
\end{equation}
\end{prop}

The proof is given in  Appendix \ref{appendice tecnica}. Of course the same estimates hold 
in $\scH_{r,p}(\C)$.

\medskip
The next result  controls the norm of the NLS non-linearity 
$P$ defined in \eqref{Hamilto0}
and it is based
on the algebra property of $\tw_p$, $p>1$, with respect to convolution.

\begin{prop}[Proposition 2.1 of \cite{BMP:almost}]\label{neminchia}
There exist $c_1,c_2>1$ continuously depending  on $p>1$ such that
if $c_1 r^2\leq {\mathtt R}$ then, recalling \eqref{analitico},
$$
|P|_{r,p}	\leq 	c_2 \frac{\abs{f}_{\mathtt R}}{\mathtt R} r^2 \,.
	$$
 \end{prop}

\medskip

\subsection{Parameter families of regular Hamiltonians}
Throughout the paper our Hamiltonians will depend on two parameters, the frequency $\omega\in \dg\subset\cQ$ and the action $I\in \cI(p,r)$. In  order to control the regularity w.r.t these parameters throughout the iterative scheme, we will introduce  an appropriate weighted norm as follows.

\medskip
Given $\gamma>0$,  a closed (w.r.t. the product topology) subset $\cO\subseteq\cQ$, and an open subset $\cI$ of some Banach space,
for $f:\cO\times\cI\to\C\,,\, (\omega,I )\mapsto  f(\omega,I)$ 
we set\footnote{We put $\gamma$ in front of the Lipschitz semi-norm because of dimensional reason. See Definition \ref{diomichela}.}
\begin{equation}
	\label{solenoide}
|f|^\g=|f|^{\g,\cO\times\cI}:=  \sup_{\substack{\omega\in\cO\\ I\in\cI} } \abs{f(\omega,I)} + \gamma  \sup_{\substack{\omega\ne \omega'\in\cO\\ I\in\cI}} \abs{\Delta_{\o,\o'}f},
\end{equation}
	where as usual 
\begin{equation}
	\label{deltab}
	\Delta_{\omega,\omega'}f := \frac{f(\omega,I) - f(\omega',I)}{\abs{\omega - \omega'}_{\infty}}\,.
\end{equation}

\begin{defn}\label{lamai}
Given $\cO$ and $\cI$ as above 
let
 $C_{\rm Lip}(\cO\times\cI)$ be  the Banach space of functions $f:\cO\times \cI \to \C,$
 which have finite norm $|f|^{\g,\cO\times\cI}$ and which are:
 \begin{itemize}[$\bullet$]
 	\item 
 continuous w.r.t  the product topology in $\cO$;
 \item
  analytic in  $\cI$. 
 \end{itemize}
In $C_{\rm Lip}(\cO\times\cI)$ we consider the subalgebra
${\rm F}(\cO\times\cI)$
 of  functions which depend only on a finite number of $\omega_j$'s.
The subalgebra ${\rm F}(\cO\times\cI)$ can be described as follows.
Given  $f\in {\rm F}(\cO\times\cI)$, which depends only on the variables
$(\o_{-k},\ldots,\o_k)$, there exists a function 
$\hat f\,:\,P_k(\cO)\times\cI\,\to\, \C$,
 where $P_k$ is the projection 
 $P_k:\R^\Z\to\R^{2k+1}$ defined as $P_k(\o):=(\o_{-k},\ldots,\o_k),$
 such that  
 $f(\o,I)=\hat f(\o_{-k},\ldots,\o_k,I)$ for every $(\o,I)\in\cO\times\cI.$
\\ 
Finally denote the closure of ${\rm F}(\cO\times\cI)$
in $C_{\rm Lip}(\cO\times\cI)$
 by $\cF(\cO\times\cI)$.
 \end{defn}
 
 \noindent
The spaces $C_{\rm Lip}(\cO\times\cI)$
 and $\cF(\cO\times\cI)$ are Banach algebras (i.e. multiplicative algebras with constant equal to 1)
 w.r.t. the norm $|\cdot|^{\g,\cO\times\cI}$.

\smallskip

\noindent
The following extension result will be proved in Appendix \ref{appendice tecnica}.
\begin{lemma}[Lipschitz extension]\label{godiva}
Given $\cO\subset \cQ$ and a ball $B_\rho$ in some complex Banach  space $E$
 and $f\in\cF(\cO\times B_\rho)$
there exists an extension $\tilde f : \cQ\times B_{\rho/2}\to\C$ such that
$|\tilde f|^{\g,\cQ\times B_{\rho/2}}\leq 2 |f|^{\g,\cO\times B_\rho},$
$\tilde f$ is continuous w.r.t  the product topology in $\cQ$,
$\tilde f$ is Lipschitz on  $B_{\rho/2}$ with estimate
$$
|\tilde f(\o,I)-\tilde f(\o,I')|\leq 
4\rho^{-1}|f|^{\g,\cO\times B_\rho}|I-I'|_E\,,
\qquad
\forall\, \o\in\pan\,,\ \   I,I'\in B_{\rho/2}\,.
$$
\end{lemma}

\medskip

In the following  we mainly consider $\cI=\mathcal I(p,r)$ (see \eqref{pifferello}), however in order to use Lemma 
\ref{godiva} we need to pass to the complex.
Then we define
\begin{equation}\label{pifferone}
\cI(\C)=\cI(p,r,\C):=\{\ 
 		I\in { B}_{r^2}(\tw_{2p})\quad
 		\mbox{with}\quad I_j=0\quad \mbox{for}\quad j\in\cS^c\		\}\,,
 	\end{equation} 
	which we systematically identify 	
	with the open ball of radius $r^2$
	centered at the origin of the Banach space
	$\{w=(w_j)_{ j\in \cS}\,\, : \,  \sup_{j\in\cS}\abs{w_j} \jjap{j}^{2p}<\infty\}.$

\begin{defn}[Real and complex Hamiltonians]\label{ala}
Let $0<r_0\leq r$, $p_0\geq p>1.$
Let $\cH_{r,p}=\cH_{r,p}^{\cO\times\cI}$ with $\cI=\cI(p_0,r_0)$  be the space of parameter depending  {\it real} regular Hamiltonians 
$H\,:\, \cO\times \cI\ni (\omega,I)  \mapsto H(\omega,I)\in \scH_{r,p}$ such that
\begin{equation}\label{entusiasmo}
H_{\al,\bt}\in \cF(\cO\times\cI)\,,\ \ \forall\,  \alpha,\beta\in \cM\quad \mbox{and}\quad
\und H_\g := \sum_{\al,\bt\in \cM} |H_{\al,\bt}|^\g u^\al\bar u^\bt\,\in\,\scH_{r,p}\,.
\end{equation}
We set
\begin{equation}\label{gorma}
\| H\|_{r,p}=\| H\|_{r,p}^{\cO\times\cI}:= |\und H_\g|_{r,p}
\stackrel{\eqref{norma1}}{=}	
\frac12  \sup_j  
\sum_{(\bal,\bbt)\in \cM} \abs{H_{\bal,\bbt}}^\gamma\pa{ \bal_j + \bbt_j}u_p^{\bal + \bbt - 2e_j}
\,.
\end{equation}
Respectively for $\cI(\C)=\cI(p_0,r_0,\C)$ we define 
the space $\cH_{r,p}(\C)=\cH_{r,p}^{\cO\times\cI(\C)}$ of parameter depending  {\it complex} Hamiltonians 
$H\,:\, \cO\times \cI(\C)\ni (\omega,I)  \mapsto H(\omega,I)\in \scH_{r,p}(\C)$
satisfying \eqref{entusiasmo} with $\cI(\C)$ instead of $\cI$
and verifying  the reality condition $H(\o,I)\in \scH_{r,p}$ when $I\in\cI.$
\\
Finally we denote by
$\cH_{r,p}^0$, resp. $\cH_{r,p}^0(\C)$
the subspace of $\cH_{r,p}$, resp. of $\cH_{r,p}(\C)$,
such that $H_{|u=0}=0.$
\end{defn}

The following result is proved in Appendix \ref{appendice tecnica}.
\begin{lemma}\label{mente}
$(\cH^{0}_{r,p}, \norma{\cdot}_{r,p})$ is a Banach-Poisson algebra in the following sense:\\
1. $(\cH^{0}_{r,p}, \norma{\cdot}_{r,p})$ is a Banach space\\
2. for any $F, G \in\cH_{r + \rho, p}$ the following bound holds
\begin{equation}\label{commXHK}
	\norma{\{F,G\}}_{r,p}
	\le 
	8\max\set{1, \frac{r}{\rho} }
	\norma{F}_{r+\rho,p}\norma{G}_{r+\rho,p}\,.
\end{equation}

\end{lemma}

\begin{prop}[Monotonicity]\label{che testicoli}
The norm $\|\cdot\|_{r,p}$  is monotone decreasing  in $p$ and monotone increasing in $r$: 
\begin{equation}\label{che palle}
 \|\cdot\|_{r,p+\delta}\leq \|\cdot\|_{r + \rho,p} \quad \forall\, \rho, \delta\geq 0.
 \end{equation}
 \end{prop}
The fact that this norm is increasing in $r$ follows directly from mass conservation and the fact that $H(0)=0$. Concerning the monotonicity in $p$, we refer the reader to \cite[Proposition 6.3]{BMP:2019}, where the proof is contained, written in the case of $| \cdot |_{r,p}$. The fact that it holds also in the Lipschitz frame, follows trivially.

\begin{prop}[Hamiltonian flow]\label{ham flow}
	Let $S\in\cH_{r+\rho, p}=\cH_{r+\rho, p}^{\cO\times\cI}$ with 
	\begin{equation}\label{stima generatrice}
	\norma{S}_{r+\rho,p} \leq\delta:= \frac{\rho}{16 e\pa{r+\rho}}. 
	\end{equation} 
	Then, for all $(\o,I)\in\cO\times\cI$ the time $1$-Hamiltonian flow
	of $S=S(\cdot,\o,I)$   is well defined, analytic, symplectic; more precisely 
	$\Phi^1_{S}: B_r(\tw_p)\to
	B_{r + \rho}(\tw_p)$ with
	\begin{equation}
	\label{pollon}
	\sup_{u\in  B_r(\tw_p)} 	\norm{\Phi^1_{S}(u)-u}_{r,p}
	\le
	(r+\rho)  \norma{S}_{r+\rho, p}
	\leq
	\frac{\rho}{16 e}.
	\end{equation}
	For any $H\in \cH_{r+\rho, p}$
	we have that\footnote{$e^{\set{S,\cdot}} H:=\sum_{k\in\N}\ad^k_S\pa{H}/k!$,
	where $\ad_S=\{S,\cdot\}$.}
	$H\circ\Phi^1_S= e^{\set{S,\cdot}} H\in\cH_{r, p}$,
	$e^{\set{S,\cdot}} H-H \in\cH_{r, p}^{0}$
	 and
	\begin{align}
	\label{tizio}
	\norma{\es H}_{r, p}& \le 2 \norma{H}_{r+\rho, p}\,,
	\\
	\label{caio}
	\norma{\pa{\es - \id}H}_{r, p}
	&\le  \delta^{-1}
	\norma{S}_{r+\rho, p}
	\norma{H}_{r+\rho, p}\,,
	\\
	\label{sempronio}
	\norma{\pa{\es - \id - \set{S,\cdot}}H}_{r, p} &\le 
	\frac12 \delta^{-2}
	\pa{\norma{S}_{r+\rho, p}}^2
	\norma{H}_{r+\rho, p}.
		\end{align}
	More generally for any $h\in\N$ and any sequence  $(c_k)_{k\in\N}$ with $| c_k|\leq 1/k!$, we have 
	\begin{equation}\label{brubeck}
	\norma{\sum_{k\geq h} c_k \ad^k_S\pa{H}}_{r, p} \le 
	2 \norma{H}_{r+\rho, p} \big(\norma{S}_{r+\rho, p}/2\delta\big)^h
	\,,
	\end{equation}
	where  $\ad_S\pa{\cdot}:= \set{S,\cdot}$.
\end{prop}
The proof  is based on \eqref{commXHK} and on the Lie series expansion for 
$\es$,
see \cite{BMP:2019} Proposition 2.1 and Lemma 2.1 for details.

\begin{rmk}\label{cotechino}
If we are working in $\cH_{r, p}(\C)$
all the estimates in
Lemma \ref{mente},
Proposition \ref{che testicoli} and Proposition \ref{ham flow}
still hold except for \eqref{pollon} which holds only for real $I$.
Indeed without the reality condition \eqref{real}
the generating function $S$ does not define anymore a Hamiltonian flow satisfying
the reality condition $\overline{u(t)}=\bar u (t).$
\end{rmk}

\section{Degree decompositions, projections and normal forms}\label{proiezioni}

{We want to prove that, in suitable variables,  $\cT_I$ introduced in \eqref{pippero} is an
 invariant torus on which the flow is linear. To this purpose  we introduce 
 a suitable degree decomposition, whose
main idea  is to make a {``power series expansion" } in the variables $|v|^2 - I$ and $z$, which control the distance to the torus $\cT_I$,  in order to highlight the terms which prevent $\cT_I$ from being a KAM torus. In the case of finite dimensional tori, one typically introduces action-angle variables but it is well known that this produces a singularity at $I = 0$. In the infinite dimensional case this should be avoided and requires some care.  An effective strategy is the expansion described below, first introduced in \cite{Bourgain:2005} and 
further formalized in \cite[Section 4]{BMP:almost}.
For convenience of the reader, we sketch here the main features of this decomposition.
The main novelty here is to control  the regularity with respect to the parameters.

\noindent
Fix $\cS$ as in \eqref{orey}. Consider a Hamiltonian $H(u)$ expanded
in Taylor series at $u=0$ and  tautologically rewrite $H$ as 
	\begin{equation}\label{vitello}
H= \sum_{\substack{m,\al,\bt\in \N^\cS\\ \al\cap \bt= \emptyset\\ a,b\in \N^{\cS^c}}} H_{m,\al,\bt,a,b}|\ub|^{2m}\ub^\al\bar \ub^\bt \zb^a\bar \zb^b
\end{equation}
	where, by slight abuse of notation\footnote{Consisting in a reordering of the indexes $j$.},  $u=(v,z)$ with $\ub= \pa{v_j}_{j\in \cS}:=\pa{u_j}_{j\in \cS}$
and 
$\zb=\pa{\zb_j}_{j\in \cS^c}:=\pa{u_j}_{j\in \cS^c}$.
Moreover by $\al\cap \bt= \emptyset$, here and in the following,
we mean that $\al_j\neq 0$ implies $\bt_j=0$.
\\
Then introduce the auxiliary ``action'' variables
 $w=(w_j)_{ j\in \cS}$ 
 substituting 
$|v|^{2m} v^\al \bar v^\bt z^a \bar z^b \rightsquigarrow w^m v^\al \bar v^\bt z^a\bar z^b$ in \eqref{vitello}.
Now we  Taylor expand the Hamiltonian
 with respect to $w$ and $z$ at the point
$w_j=I_j$ for $j\in \cS$ and $z=0$ respectively.

\begin{defn}[Degree decomposition] Let $I\in \mathcal I(p,r)$ (recall \eqref{pifferello}). For every integer $d\ge -2$ and any regular Hamiltonian $H\in\scH_{r,p}$ we define the  projection $(\Pi^d H)(u)=(\Pi^d_I H)(u)=H^{(d)} (u)$ defined as

	\begin{equation}\label{grado vero}
	  H^{(d)} (u):= \!\!\!\!\!\sum_{\substack{m,\al,\bt,\delta\in \N^\cS, a,b\in\N^{\cS^c}\\\al\cap \bt= \emptyset\,,\; \delta\preceq m \\ 2|\delta|+|a|+|b|= d+2} }  \!\!\!\!\!{H}_{m,\al,\bt,a,b}
	\binom{m}{\delta} I^{m-\delta} (|\ub|^2- I)^\delta \ub^\al \bar \ub^\bt \zb^a \bar \zb^b\,,
	\end{equation}
	where $\delta \preceq m$ means that $\delta_s \leq m_s$ for any $s\in\cS$, $|v|^2 =\pa{|v_s|^2}_{s\in\cS}$, while the multi-index notations are introduced in Definition \ref{baccala}.
	We also set $\Pi^{\leq d}:=\sum_{d'=-2}^{d}\Pi^{d'}$ and $\Pi^{\geq d}:={\rm Id} - \Pi^{< d}$, 
	analogously for $\Pi^{< d}$
	and
	$\Pi^{> d}$.
	Moreover we define, e.g.,
	$H^{(\le d)}:=\Pi^{(\le d)}H$
	and also
	 $\scH_{r,p}^{d}:=\Pi^{d}\scH_{r,p}$ and, e.g., $\scH_{r,p}^{\leq d}:=\Pi^{\leq d}\scH_{r,p}$. 
\end{defn}
Note that, if $\cS=\Z$, projections coincide with the ones of Section 4 of \cite{BMP:almost}, while if $\cS= \emptyset$, $H^{(d)}$ represents the usual homogeneous degree at $\zb=0$.
\\
In this way, given $H\in\scH_{r,p}$, then 
\begin{equation}\label{semicroma}
H = H^{(\le 0)} + H^{(\ge 1)} \equiv H^{(-2)} + H^{(-1)} + H^{(0)} + H^{(\ge 1)}  
\end{equation}
 where $H^{(-2)}$ consists of terms which are constant w.r.t. both $z$ and and the "auxiliary action" $w = \abs{v}^2$, $H^{(-1)}$ is independent of the action but linear in the $z_j$, while  $H^{(0)}$ contributes with two terms: the one linear in the action and independent of $z$, the second one quadratic in $z$ and independent of the action. Finally, $H^{(\ge 1)} $ is what is left and $X_{H^{(\ge 1)}}$ vanishes on $\cT_I$. 
 \smallskip
 
 The operators $\Pi^d$ define continuous projections (see
 Section 4 of \cite{BMP:almost} and also
  Proposition \ref{proiettotutto}) satisfying $\Pi^d \Pi^d = \Pi^d $ and
	$\Pi^{d'} \Pi^d=\Pi^d \Pi^{d'} =0$
	for every $d'\neq d$,  $d'\geq -2.$ Moreover, this decomposition enjoys all the crucial properties required for a KAM scheme to converge, in particular they behave well with respect to Poisson brackets, that is: 
	$$
	\forall F, G \in\scH_{r,p}\quad  \set{F, G^{\ge 1}}^{(-2)} = 0 
	$$
and 	
\begin{equation*}
\label{woodstok}
 \begin{aligned}
  \quad F^{(-2)}=0 \, \Longrightarrow \, \set{F, G^{\ge 1}}^{(\leq -1)} = 0, \\
   \quad F^{( -1)} = 0 = F^{ (-2)} \, \Longrightarrow  \set{F, G^{\ge 1}}^{(\le 0)} = 0.
 \end{aligned}
\end{equation*}
Note also that
if $\Pi^{< d_1}F = \Pi^{< d_2} G=0$, then 
	$
	 \Pi^{< d_1+d_2}\set{F,G}=0\,.
$
For all the properties of the projections see \cite{BMP:almost} Proposition 4.1 and 4.2.

 As is standard, on $\scH_{r,p}$ we define the projections
\begin{equation}
\label{ragno}
\Pi^\cK H := \sum_{\bal\in\N^\Z}H_{\bal,\bal}|u|^{2\bal}\,,\qquad \Pi^\cR H := H- \Pi^\cK H\,,
\end{equation}
which  are continuous on $\scH_{r,p}$. \\
Correspondingly, we define the following subspaces of $\scH_{r,p}$:
\begin{equation}\label{dolcenera}
\scH_{r,p}^\cK:=\{ H\in \scH_{r,p}\,:\quad \Pi^\cK H= H\}
\,,\qquad
\scH_{r,p}^\cR:=\{ H\in \scH_{r,p}\,:\quad \Pi^\cR H= H\}
 \,.
\end{equation}
Moreover, e.g. $\scH_{r,p}^{\leq 0, \cR}:=
\scH_{r,p}^{\leq 0}\cap \scH_{r,p}^{\cR}.$
Note that $\scH_{r,p}^\cK\subseteq \ker L_\omega$ and  if $\omega\in \dg$ then  the two spaces coincide.
\\
Note that by \eqref{grado vero} and \eqref{ragno}
we have
\begin{equation}\label{scalogno}
d \ \ \mbox{odd}\quad\Longrightarrow\quad
\scH_{r,p}^{d,\cK}=\{0\}\,.
\end{equation}
In  Lemma 4.3 of \cite{BMP:almost}  we proved that 
the map $\lambda\to \Lambda$
defined by
\begin{equation}\label{silvestre}
	\Lambda = \sum_{j\in \cS} \lambda_j( |v_j|^2 -I_j)	 + \sum_{j\in\cS^c} \lambda_j |z_j|^2
		\end{equation}
	is a linear isometry from $\ell^\infty$ to
$\scH_{r,p}^{0,\cK}$
for every $r,p$.
Note that, taking $\Lambda$ as above 
\begin{equation}\label{creep}
H=\set{\Lambda,G^{(d)}}\qquad
\Longrightarrow
\qquad
H=H^{( d)}\,.
\end{equation}

\medskip
The projections defined in \eqref{grado vero}
naturally extend to
$H\in\cH_{r,p}$ or $H\in\cH_{r,p}(\C)$
setting
\begin{equation}\label{grado falso}
(\Pi^d H) (u,\o,I):=(\Pi^d_I H(\o,I))(u)\,.
\end{equation}
Similarly for \eqref{ragno}.

\noindent
The following result  is proved in Appendix \ref{appendice tecnica}.

\begin{prop}\label{proiettotutto}
	For every  $d\in\N\cup\{-2,-1\}$ 
	 and $H\in \cH_{r',p}=\cH_{r',p}^{\cO\times\cI}$
	where $\cI=\mathcal I(p,r)$
	with
$r'\geq \sqrt 2 r$
 the following holds.
\\
(i)	 The projection operators
	$\Pi^d: \cH_{r',p}\to \cH_{r',p}$ are continuous with bound  
	\begin{equation}\label{cacioepepe}
\|\Pi^d H \|_{ r',p}
\leq 
3^{\frac{d}{2}+1}
\|H\|_{r',p}\,.
\end{equation}
(ii) Moreover
 the following representation formula holds
\begin{equation}\label{rappre}
\Pi^{\ge d} H(u) = \sum_{\substack{\delta\in \N^\cS, a,b\in\N^{\cS^c}\\ 2|\delta|+|a|+|b|= d+2} } 
 (|\ub|^2-I)^\delta \zb^a \bar \zb^b\check{H}_\delta(u)
\qquad\mbox{for}\ \ |u|_p< r'\,,
\end{equation}
where $\check{H}_\delta(u)$ are analytic  in ${B}_{r'}(\tw_p)$ and can be written in totally convergent power series in every ball
$|u|_p\leq \kappa_* r'$ with $\kappa_*< 1$.
\\
Analogous statements hold for the complex case  $H\in \cH_{r',p}(\C)=\cH_{r',p}^{\cO\times\cI(\C)}$
	where $\cI(\C)=\mathcal I(p,r,\C)$ (recall formula \eqref{pifferone}).
\end{prop}

As above for $\scH_{r,p}$ we define the corresponding subspaces 
$\cH_{r,p}^d,$ $\cH_{r,p}^\cK$, $\cH_{r,p}^\cR$, etc.
of $\cH_{r,p}.$ Analogously for the complex case $\cH_{r,p}(\C)$.
In  particular we discuss the subspace $\cH_{r,p}^{0,\cK}$.

\begin{defn}
We denote by $\cH^{0,\cK}=\cH^{0,\cK}(\cO\times \cI)$ the space of maps 
$$
\cO\times \cI \ni (\o,I )\to \lambda(\o,I)\in \ell^\infty(\R)\quad\qquad \mbox{with}\quad\lambda_j\in \cF\pa{\cO\times \cI} \quad \forall j\in\Z,
$$
endowed with the norm (recall \eqref{solenoide})
\begin{equation}
	\label{grevity}
	\|\lambda\|_\infty:=\sup_{j\in\Z}|\lambda_j|^\gamma\,.
\end{equation} 
Then by \eqref{silvestre}
$\cH^{0,\cK}$ is isometrically equivalent to $\cH_{r,p}^{0,\cK}$ for all $r,p.$
Similarly in the complex case $\cH^{0,\cK}(\C)=\cH^{0,\cK}(\cO\times \cI(\C)).$
\end{defn}

\noindent
 Since are we mainly interested in the decomposition \eqref{semicroma}, we remark  (see Proposition \ref{proiettotutto})  that for  $H\in \cH_{r',p'}^{\cO\times\cI}$ with
 $\cI=\mathcal I(p,r)$ and
$r'\geq \sqrt 2 r\,,\  p'\leq p$, one has 
\begin{equation}\label{ritornello}
\begin{aligned}
&\|\Pi^0 H\|_{r',p'}\,,\ \| \Pi^{0,\cK} H\|_\infty \leq  3  \|H\|_{r',p'}
\,, \\&
\|\Pi^{-1} H\|_{r',p'},\|\Pi^{-2} H\|_{r',p'}  \leq   \|H\|_{r',p'}\,,\quad
\|\Pi^{\geq 1} H\|_{r',p'}  \leq  6  \|H\|_{r',p'}\,.
\end{aligned}
\end{equation}
Analogous estimates hold in the complex case
$H\in \cH_{r',p'}^{\cO\times\cI(\C)}$ with
 $\cI(\C)=\mathcal I(p,r,\C)$ (see \eqref{pifferone}).

\section{Proof of Theorems  \ref{toro Sobolevp} and \ref{torobolev}}\label{cinque}

\noindent
\begin{defn}[Normal Forms] {For $\omega\in\cQ$ let\footnote{Note that $D$ is a linear map into the space of formal quadratic polynomials. } 
 \begin{equation*}
 	\label{uffa}
D: \omega \quad \mapsto \quad D(\o):= \sum_{j\in\Z} \omega_j |u_j|^2\,.
 \end{equation*}}
 Let 
 $0<r<r_0$, $1<p_0\leq p$ and $\cI=\mathcal I(p,r)$,
  as in \eqref{pifferello}.
 We will say that    $N$ is an analytic family of normal forms   if 
$N-D\in  \cH_{r_0,p_0}^{{\ge 1},\cO\times \cI}$. We denote such affine subspace by $\cN_{r_0,p_0}= \cN_{r_0,p_0}^{\cO\times\cI}$. The  same definition holds in the corresponding complex spaces.

 \end{defn}

\begin{rmk}
Given $N\in \cN_{r_0, p_0}^{\cO\times\cI}$  for every $\o\in\cO$ and $I\in \cI(p,r)$,  the flat torus $\cT_I$ is invariant for the dynamics of
$N(\omega,I)$.
\end{rmk}

\noindent

Let us state our counter-term KAM result.
\begin{them}
	\label{allaMoserbis} 
Fix  $0<\g<1$, $\cS$   as in \eqref{orey}, $r>0$ and $p>1$. Set  $\cO=\dg$ and $\cI= \mathcal I(p,r)$. Consider 
$r_0, \rho, \delta>0$ and $p_0>1$ with \begin{equation}
\label{newton}
\rho\le \frac{r_0-r}{4}\,,\quad r \le \frac{r_0}{2}\,,\quad
\delta\leq\frac{p-p_0}{4}\,.
\end{equation}
There exists $\bar{\epsilon},1/\bar C>0$, decreasing functions of
$
\rho/r_0$ 
and
$\delta$ 
such that the following holds. 
Consider 
  $N_0\in\cN_{r_0, p_0 }^{\cO\times\cI}$,
    $H-D\in \cH_{r_0,p_0}^{\cO\times\cI}$ and assume that
\begin{align}\label{maipendanti}
& (1+\Theta)^5 \epsilon \le \bar{\epsilon} \quad \mbox{where}
\\
& \epsilon:=\g^{-1}\sup_{I\in\mathcal I(p,r)}\norma{H- N_0}_{r_0 ,p_0 }\,,
\quad \Theta := 
\g^{-1}\sup_{I\in\mathcal I(p,r)}\norma{D- N_0}_{r_0 ,p_0}\,.\nonumber
\end{align}
Then there exist
a $\rho/4$-close to the identity  symplectic diffeomorphism in $u$,
Lipschitz in $\omega$ and analytic in $I$
\begin{equation}\label{pocafede}
	\Psi:{B}_{r_0-\rho}\pa{\tw_{p}}\times \cO\times \cI
	\to
	{B}_{r_0}\pa{\tw_{p}} \,, \quad (u;\omega,I)\mapsto \Psi(u;\omega,I)
	\end{equation}
\noindent and a counter-term $\Lambda\in \cH^{0,\cK}(\cO\times\cI)$, where
\begin{equation}\label{loredana}
\Lambda = \sum_{j}\lambda_j \pa{\modi{j} - I_j}\,, \quad  \|\lambda\|_\infty
 \le \bar C \g   (1+\Theta)^2\epsilon\,,
\end{equation}
such that   $\Psi(\cT_I;\o,I)$ is an invariant torus for the dynamics of
$\Lambda(\o,I)+H(\o)$ for every $\o\in\dg$ and $I\in\mathcal I(p,r).$ 
\\
More precisely there exists 
an analytic family of normal forms $N \in\cN_{r_0-\rho,p_0 + \delta}^{\cO\times\cI}$, such that
\begin{equation}\label{coniugio}
\pa{\Lambda + H}\circ \Psi - N=0\,.
\end{equation}
Finally if $N_0$ and $H$ admit a complex extension on $\cI(p,r,\C)$ satisfying \eqref{maipendanti}, then also the counter-term $\lambda$ extends to an holomorphic map in $\cH^{0,\cK}(\cO\times \cI(p,r,\C))$ satisfying \eqref{loredana}.

\end{them}

\begin{rmk}\label{venezia}
i)
The map $\Psi$ is $\rho/4$-close to the identity
namely
$$
|\Psi(u;\omega,I)-u|_p\leq \rho/4\,,\qquad
\forall\, (u;\omega,I)\in
{B}_{r_0-\rho}\pa{\tw_{p}}\times \cO\times \cI\,.
$$
ii)
 By \eqref{loredana} and Cauchy estimates we get that $\lambda$ is uniformly Lipschitz in $\mathcal I(p,r/2)$
 \begin{equation}\label{berte} 
 \frac{\abs{\lambda(\omega,I) - \lambda(\omega,I')}_\infty}{\abs{I - I'}_{2p}}
 \le 
 4\bar C    (1+\Theta)^2   \g r^{-2}\epsilon\,,
 \end{equation}
 $\forall\, \omega\in\dg,\ \  
 I,I'\in \mathcal I(p,r/2)\,,\ \ I\neq I'\,.$
\end{rmk}

\begin{proof}[Proof of Theorem \ref{toro Sobolevp} and \ref{torobolev}]\label{auricchio}
Theorem \ref{toro Sobolevp} follows from Theorem \ref{allaMoserbis} in a straightforward way.
Fix $p_0= p-(p-1)/2$, $r_0= 2r=4\rr$, $\rho = r/8$ and $\delta= (p_*-1)/16$.
Note that by these choices 
\begin{equation}\label{adu}
\mbox{the constants}\ \bar{\epsilon}\ \mbox{and}\ \bar C\   \mbox{depend only on}\  p_*\,.
\end{equation}
One first rewrites $H_V$ in \eqref{Hamilto0} as $D +\Lambda+P'$ where $P' = P + \sum_{{j}} \lambda_j I_j$,  
 by setting  $\lambda_j = j^2-\omega_j + V_j$.
Set $H=D+P'$ and $N_0=D$ so that by definition $\Theta=0$. Since the Hamiltonian $P\in \scH_{r _0,p_0}(\C)$ does not depend on $\omega,I$,
trivially  one has  $P'\in \cH_{r _0,p_0}(\C)$ and $\|{P'}\|_{r _0,p_0}=\|{P}\|_{r _0,p_0}=\norm{P}_{r _0,p_0}$.
\\
Since by hypothesis
$\g\leq \abs{f}_{\mathtt R}$ we have by \eqref{cornettonep} that
$\rr^2/ \mathtt R\leq  \e_*$; then the hypothesis  
$c_1 \rr^2\leq  \mathtt R$ of Proposition \ref{neminchia}
holds for all $p$ satisfying \eqref{diavola} taking $\e_*$ 
small enough.
Then , by Proposition \ref{neminchia}, 
\begin{equation}\label{sade}
\epsilon=\gamma^{-1}\|P'\|_{r_0,p_0}= \gamma^{-1}|P|_{r_0,p_0}\leq \gamma^{-1}
|P|_{2r,(p_*+3)/4}
\leq {\rm c}(p_*)\varepsilon
\end{equation}
 for 
some\footnote{Note that the constants $c_1,c_2$ in Proposition \ref{neminchia}
continuously depend on $p$, which belongs to the compact
$[(p_*+1)/2,p_*]$.} 
${\rm c}(p_*)>1$
by \eqref{cornettonep}.
   Again taking $\e_*=\e_*(p_*)$ in  \eqref{cornettonep} small enough, condition \eqref{maipendanti} is satisfied and 
   Theorem \ref{allaMoserbis} gives us the desired change of variables provided that $\Lambda\in \cH^{0,\cK}(\C)$ is fixed accordingly.

   Now we denote $\omega_j =\nu_j$ if $j\in \cS$ and $\omega_j=\Omega_j$ otherwise.
We get the equations
\begin{equation}
\label{equaOme}
\begin{cases}
&\bO_j+ \lambda_j(\nu, \bO,I) = j^2+V_j \,,\quad \mbox{if }\; j\in \cS^c\\
&\nu_j +\lambda_j(\nu, \bO,I)= j^2+V_j \,,\quad \mbox{if }\; j\in \cS\,,
\end{cases}
\end{equation}
for $\o\in \dg$ and $I\in \cI(p,r,\C)$.
By Lemma \ref{godiva} we   Lipschitz extend the map 
$\lambda$ 
to the whole  $\pan\times \cI(p,\rr,\C)$,
(recall that $\rr=r/2$)
in  such a way that
 \eqref{loredana} and \eqref{berte}  hold
  for $\omega,\omega'\in \pan$ (with $\bar C\rightsquigarrow 2\bar C$).
  From now on we can work only on the 
  real\footnote{Obviously the same statements hold also in the complex case} case $I\in\cI(p,\rr)$.
By \eqref{loredana}
and taking $\e_*(p_*)$ small enough such that
$2\bar C\epsilon\leq 2 \bar C c(p_*)\e_*(p_*)\leq 1/2$
(recall \eqref{sade})
 we use the 
 Contraction Lemma (recall Lemma\footnote{
 We use it with 
$
F_j(\nu,\VSc,w):=-\lambda_j(\nu,q+\VSc+w)\,,$ 
$j\in\cS^c,
$
where $q=(j^2)_{j\in\cS^c}$
and $r:=\bar C \g \epsilon.$
 } \ref{zagana}), 
 we solve the first equation finding 
 $\bO: \cQ_\cS \times [-\nicefrac14,\nicefrac14]^{\cS^c}\times \mathcal I(p,\rr) \to \cQ_{\cS^c}$
 which is continuous in the product topology of $\cQ_\cS \times [-\nicefrac14,\nicefrac14]^{\cS^c}$ and
 satisfies
 \begin{eqnarray}\label{martini}
	 |{\bO_j(\nu,\VSc,I)-j^2-\VSc_{\!\!,j}}|
	 &\leq& 2\bar C \g \epsilon\,,\\
	\sup_{\nu'\neq\nu}\frac{|\bO_j(\nu,\VSc,I)-\bO_j(\nu',\VSc,I)|}{|\nu-\nu'|_\infty}
	&\le& 4 \bar C \epsilon\,,\nonumber
	\\
	\sup_{\VSc'\neq\VSc}\frac{|\bO_j(\nu,\VSc',I)-\bO_j(\nu,\VSc,I)|}{|\VSc'-\VSc|_\infty}
	&\le& 4 \bar C\nonumber \\	\quad
	\sup_{I'\neq I}
	\frac{\abs{\bO_j(\nu,\VSc,I) - \bO_j(\nu,\VSc,I')}}
	{ \abs{I - I'}_{2p}}
&\leq&
 2\bar C       \g \rr^{-2}\epsilon
\,,
	\nonumber
	\end{eqnarray}
	for every $ j\in \cS^c$ and where the first three estimates follows from \eqref{loredana}
	(recalling \eqref{grevity} and \eqref{solenoide}) and the
	 last one from \eqref{berte}.
	By \eqref{martini} and \eqref{sade} we 
	prove \eqref{omegone} and  \eqref{batacchio}
	 taking $C\geq 4\bar C c(p_*)$. Note that this condition depends only on $p_*.$
 \\
Finally  the second equation in \eqref{equaOme} uniquely defines
\begin{equation}\label{visciola}
\VSf_{\!,j}(\nu,\VSc,I)=\nu_j +\lambda_j(\nu, \bO(\nu,\VSc,I),I)- j^2
\end{equation}
for $j\in \cS$.
\\
By Lemma \ref{pinzellone} iii)  
we can Lipschitz extend
the map $\Psi:{B}_{r_0-\rho}\pa{\tw_{p}}\times\dg\times \mathcal I(p,r) \to {B}_{r_0}\pa{\tw_{p}} $ 
on 
${B}_{r_0-\rho}\pa{\tw_{p}}\times\pan\times \mathcal I(p,r)$ and set  $\Phi(u;\nu,\VSc,I):=\Psi(u;\omega(\nu,\VSc,I),I)$, where $\omega(\nu,\VSc,I)$ was defined in \eqref{arancini}.
The map $\Phi$ conjugates the NLS Hamiltonian $H_V$ to  normal form
for all $\nu$ such that 
$\o(\nu,\VSc,I)\in \dg$, namely
for $\nu\in \cC$
(recalling \eqref{arancini}).
Then the proofs of Theorems \ref{toro Sobolevp} and \ref{torobolev} 
are concluded  provided that one shows that the set $\cC$
defined in \eqref{arancini} satisfies \eqref{mirto}.  This is the content of Lemma \ref{misurino} proved in Section Section \ref{bellogrosso}, where all the measure estimates are discussed.

\end{proof}



\section{Small divisors and Homological equation}
%
%
%
%
%

\vskip10pt

The proof of Theorem
\ref{allaMoserbis} 
is based on an iterative scheme that 
kills out the obstructing terms, namely  terms belonging to 
$ \cH_{r,p}^{-2},\cH_{r,p}^{-1}$ and $ \cH_{r,p}^{0}$,
by solving Homological equations of the form 
\begin{equation}\label{esperidi}
L F^{(d)} = G^{(d)},\qquad G^{(d)}\in \cH_{r,p}^{d},\quad d= -2,-1,0.
\end{equation}
where, for all $\omega\in\cQ$,  $L_\omega$ is an operator acting on formal Hamiltonians as
\[
L_\omega[\cdot]:= \{D(\omega),\cdot \}\,,\qquad L_\omega H(\omega,I):= \im\sum_{\bal,\bbt\in\N^\Z }\omega \cdot \pa{\bal-\bbt}
H_{\bal,\bbt}(\omega,I)u^\bal \bar u^\bbt\,.
\]
The convergence of the iterative KAM scheme comes from a good control of  $L^{-1} G^{(d)}$ over the set $\dg$.

\bigskip
For $\omega\in \dg$ 
the Lie derivative operator $L_\omega$ is formally invertible on the subspace $\cH_{r,p}^\cR$ with inverse
\begin{equation}\label{vesuvio}
(L^{-1}_\omega H(\omega,I))_{\bal,\bbt}:=\frac{-\im H_{\bal,\bbt}(\omega,I)}{\omega\cdot(\bal-\bbt)}\,.
\end{equation}
We now show that the  inverse is well defined also at a non formal level.
\begin{prop}\label{mah}
Let  $\cO= \dg$ and $\cI= \cI(p,r)$ and set  $\cH_{r,p}^{\leq 0, \cR}= \cH_{r,p}^{\leq 0, \cR,\cO\times\cI}$.  For every $0<\delta<1$, \eqref{vesuvio} defines a bounded linear  operator $ 
 L^{-1}:\cH_{r,p}^{\leq 0, \cR}\ \to\ \cH_{r,p+\delta}^{\leq 0, \cR}$ with 
 estimate\footnote{The constant $1<\eta\le 2$ was introduced in \eqref{cicoria}.}
\begin{equation}\label{ss148}
\|L^{-1} H\|_{r,p+\delta}
\le 
\frac{1}{\gamma}
\exp\left( \exp\left(\frac{c}{ \delta^{1/\eta}}\right)\right)
\|H\|_{r,p} \,,
\end{equation} 
with $i_*,\eta$  introduced in Definition in \ref{orey} and $c=c(i_*).$ 
The same estimate  holds in the corresponding complex spaces.
\end{prop}

\begin{rmk}
 In this section by $c$ we denote possibly different
 constants depending only on $i_*$
 introduced in Definition in \ref{orey}.
\end{rmk}
\begin{proof}
	Let us first show that if $H_{\bal,\bbt}\in \cF(\cO\times\cI)$ then the same holds for   $(L^{-1}H)_{\bal,\bbt}$.
	 Indeed, since $|\bal|=|\bbt|<\infty$, the expression $\o\cdot(\bal-\bbt)$ depends only on a finite number of frequencies $\omega_i$ and hence is continuous w.r.t. the product topology.
	 As for the Lipschitz dependence, $\o\in \dg$ implies that $1/\omega\cdot(\bal-\bbt)$ is $C^\infty$ and the result follows, since the product of Lipschitz functions is Lipschitz.
	\\
			By \eqref{vesuvio} and  \eqref{gorma} we get
	\begin{equation}\label{pontina}
		\|L^{-1} H\|_{r,p+\delta} 
		\le 
		\frac{1}{\gamma}
		K \|H\|_{r,p}\,,
	\end{equation}
	where
	\begin{equation}\label{tacchinobistot}
		K = 
		\g \sup_{j\in\Z} 
		\sup_{(\bal,\bbt)\in \cM_j}
		\pa{{\frac{\jjap{j}^2}{\Pi_s \jjap{s}^{\bal_s + \bbt_s}}} }^\delta 
		\abs{\frac{1}{\omega\cdot\pa{\bal-\bbt})}}^\g\,,
	\end{equation}
	and (recall that $H\in \cH_{r,p}^{\le 2,\cR}$)
	\begin{equation}\label{caciocavallo}
		\cM_j:=\left\{(\bal,\bbt)\in \cM\ \ {\rm s.t.}\ \   \bal\neq\bbt\,,\ \ 
		\sum_{s\in \cS^c}\bal_s+\bbt_s\le 2\,,\ \ 
		\bal_j+\bbt_j\neq 0
		\right\}\,,
	\end{equation}
	where $\cM$ was defined in \eqref{zorro}.
We define
\begin{equation}\label{divisor}
	\cM_j':=\left\{(\bal,\bbt)\in \cM_j,\quad \mbox{such that}\quad 	\abs{\sum_{s\in \Z}{\pa{\bal_s-\bbt_s}s^2}}< 2 \sum_{s\in\Z}\abs{\bal_s-\bbt_s}\right\}\,,
\end{equation}
and consider  
\begin{equation}\label{tacchinobistot2}
	K_1 = 
	\g \sup_{j\in\Z} 
	\sup_{(\bal,\bbt)\in \cM_j'}
	\pa{{\frac{\jjap{j}^2}{\Pi_s \jjap{s}^{\bal_s + \bbt_s}}} }^\delta \abs{\frac{1}{\omega\cdot\pa{\bal-\bbt})}}^\g
\end{equation}
and 
\begin{equation}\label{tacchinobistot3}
	K_2= 
	\g \sup_{j\in\Z} 
	\sup_{(\bal,\bbt)\in\cM_j\setminus \cM_j'}
	\pa{{\frac{\jjap{j}^2}{\Pi_s \jjap{s}^{\bal_s + \bbt_s}}} }^\delta \abs{\frac{1}{\omega\cdot\pa{\bal-\bbt})}}^\g.
\end{equation}
Then
\begin{equation}\label{bominaco}
K=\max\{K_1,K_2\}\,.
\end{equation}
We now have to give an upper bound on $K_1$ and $K_2$.
First note that, by \eqref{pan},
$$
\abs{\omega\cdot\pa{\bal-\bbt}} \ge |\sum_{s\in\Z} s^2\pa{\bal_s - \bbt_s}| -
\frac{1}{2}\sum_{s\in\Z}\abs{\bal_s - \bbt_s}\,
$$ so 
\begin{equation}\label{fiandre}
	\abs{\sum_{s\in \Z}{\pa{\bal_s-\bbt_s}s^2}}\ge 2 \sum_{s\in \Z}\abs{\bal_s-\bbt_s}
	\qquad
	\Longrightarrow
	\qquad
	\abs{\omega\cdot\pa{\bal-\bbt}}\ge |\bal-\bbt| \geq1. 
\end{equation}

\begin{rmk}
 In the following we will strongly use the fact that the involved functions
 depend only on a {\sl finite} number of $\o_j$.
 In this case the Lipschitz semi-norm 
 is bounded by the $\ell^1$-norm of the gradient. 
\end{rmk}

The estimate of  $K_2$ is trivial. Indeed,  since $(\bal,\bbt)\in\cM_j\setminus \cM_j'$  then 
$\abs{\omega\cdot\pa{\bal-\bbt}}\ge |\bal-\bbt| \geq1$. Therefore
\[
\abs{\frac{1}{\omega\cdot\pa{\bal-\bbt})}}^\g\le \sup_{ \omega\in \dg}\frac{1}{|\omega\cdot\pa{\bal-\bbt}|}+ \g  \sup_{ \omega\in \dg}\frac{|\bal-\bbt|}{(\omega\cdot\pa{\bal-\bbt})^2}\le 2\,,
\]
so that in conclusion 
\begin{equation}\label{K2}
K_2 \le 2\,.
\end{equation}
\smallskip
\noindent
Let us now study $K_1$.
Since $(\bal,\bbt)\in \cM_j'$, then (recall Definition \ref{diomichela})
\[
\sup_{ \omega\in \dg}\abs{\frac{1}{\omega\cdot\pa{\bal-\bbt})}} \le \frac{2}{\gamma \gatta(\bal-\bbt)}, \quad \mbox{with}\quad \gatta(\ell):=\prod_{s\in\cS} \frac{1}{(1+|\ell_s|^2 \langle i(s)\rangle^2)^{3/2}}\,;
\] 
as for the Lipschitz variation we estimate it as 
\[
\sup_{ \omega\in \dg} \sum_{j}\abs{\partial_{\omega_j} \frac{1}{\omega\cdot\pa{\bal-\bbt}} } \le \sup_{ \omega\in \dg}\frac{|\bal-\bbt|}{(\omega\cdot\pa{\bal-\bbt})^2} \le
\frac{4}{\gamma^2\big(\gatta(\bal-\bbt)\big)^3}\,,
\]
where the last inequality comes form $|\bal-\bbt|\le \big(\gatta(\bal-\bbt)\big)^{-1}$.
Note that the sum in the left hand side above is over a {\sl finite} number of 
indexes $j$'s.
In conclusion we have proved that for  $(\bal,\bbt)\in \cM_j'$
\[
\abs{\frac{1}{\omega\cdot\pa{\bal-\bbt})}}^\g\le \frac{14}{\gamma\big(\gatta(\bal-\bbt)\big)^3}\,,
\]
and hence, recalling the definition of $\eta$ in \eqref{cicoria}, we have
\begin{eqnarray}\label{ostiabeach}
&&K_1\le  14 K_1'\,,\qquad \mbox{where}
\\
&&
K_1':=
\sup_{j\in\Z} 
\sup_{(\bal,\bbt)\in \cM_j'}
\pa{{\frac{\jjap{j}^2}{\Pi_s \jjap{s}^{\bal_s + \bbt_s}}} }^\delta \prod_{s\in\cS} \pa{1 + \jap{i(s)}^2 \abs{\bal_s - \bbt_s}^2}^{9/2}\,.\nonumber
\end{eqnarray}
 By \eqref{pontina},
\eqref{bominaco}, \eqref{K2}, \eqref{ostiabeach}
and the following crucial estimate
 \begin{equation}\label{barisciano}
\log	K_1'\leq c\exp(135\delta^{-1/\eta})
\,,
\end{equation}
 Proposition \ref{mah} follows. 
 \end{proof}
\noindent
The rest of this section is devoted to the proof of 
estimate \eqref{barisciano}.
We first need some preliminaries discussed
in the following subsection.

\subsection{Preliminaries}

We first state an elementary result that will be useful
below:
\begin{lemma}\label{carciofo}
	If $\ell\in\Z^\Z$ with $|\ell|<\infty$, satisfies $\fm(\ell)=\pi(\ell)=0,$ then
	$|\ell|$ is even and $|\ell|\neq 2.$
\end{lemma}
\begin{proof}
	As usual we write in a unique way $\ell=\ell^+ -\ell^-$ where
	$\ell^\pm_s\geq 0$  and $\ell^+_s\ell^-_s= 0$ for every $s\in\Z.$
	Since  $\fm(\ell)=0$ we get $|\ell^+|= |\ell^-|$, therefore
	$|\ell|=|\ell^+|+ |\ell^-|$ is even.\\
	Now assume by contradiction that $|\ell^+|= |\ell^-|=1.$ Then
	$\ell^+=e_i,\ell^-=e_j$ for some $i\neq j$ and $\pi(\ell)=i-j,$
	which contradicts $\pi(\ell)=0.$
\end{proof}
 
 \medskip
\noindent
We now need some  basic results and notations coming from \cite{BMP:2019}.

\begin{defn}\label{n star}
	Given a vector $v\in \N^\Z$, $0<|v|<\infty,$
	 we denote by $\na=\na(v)$ the vector $\pa{\na_l}_{l\in I}$ 
	 (where $I\subset \N$ is finite)  which is the 
	 decreasing rearrangement 
	of
	$$
	\{\N\ni h> 1\;\; \mbox{ repeated}\; v_h + v_{-h}\; \mbox{times} \} \cup \set{ 1\;\; \mbox{ repeated}\; v_1 + v_{-1} + v_0\; \mbox{times}  }.
	$$
\end{defn}
\begin{rmk}
	A good way of envisioning this list is as follows. 
	Given the set of (commutative) variables $\pa{x_j}_{j\in \Z}$, we consider a monomial 
	(recall \eqref{mergellina})
	\[
x^v= \prod_i x_i^{v_i} = x_{j_1} x_{j_2}\cdots x_{x_{|v|}}\,,
	\] 
	then $\na(v)$ is the decreasing rearrangement of the list $\pa{\jap{j_1},\dots,\jap{j_{|v|}}}$.
As an example consider $v=\pa{v_j}_{j\in \Z}$
\[
 \mbox{with}\quad v_{-6}=1,v_{-3}=4, v_{-1}=v_0=1, v_1=v_{6}=2\,,\quad \mbox{and}\quad v_j=0 \; \mbox{otherwise}\,,
\]
  then 
$
\na(v)= (6,6,6,3,3,3,3,1,1,1,1)
$.
\end{rmk}

\noindent
Given $(\bal,\bbt)\in \cM$ 
with $|\bal|=|\bbt|>1,$
from now on we define
$$
\na:=\na(\bal+\bbt)\,.
$$ 
We  observe that
there exists a choice of $\s_i = \pm1, 0$ such that 
by momentum conservation
\begin{equation}\label{pi e cappucci}
 \sum_l \sigma_l\na_l=0\,.
\end{equation}
with $\sigma_l \neq 0$  if $\na_l \neq 1$. Indeed we recall that 
\[
\sum_{j\in \Z} j(\bal_j-\bbt_j) =  \sum_{h\in \N} h(\bal_h+\bbt_{-h} - (\bal_{-h} + \bbt_{h}))
\]
and that each $ h>1$ appears exactly $\bal_{h}+ \bal_{-h}+\bbt_{h}+ \bbt_{-h}$ times in the sequence $\na$. Hence  we assign the value $\s=1$, $\bal_{h}+ \bbt_{-h}$ times  to the $\na_l= h$ (and $\s=-1$ the remaining times). The value $h=1$ instead appears  in the sequence $\na$, $\bal_{1}+ \bal_{-1}+\bbt_{1}+ \bbt_{-1}+\bal_0+\bbt_0$ times. Hence  we assign the value $\s=1$, $\bal_{1}+ \bbt_{-1}$ times, the value $\s=-1$, $\bal_{-1}+ \bbt_{1}$ times and $\s=0$ all the remaining times.

\smallskip
\noindent
From \eqref{pi e cappucci} we deduce 
\begin{equation}\label{eleganza}
\na_1\le  \sum_{l\ge 2}\na_l.
\end{equation}
Indeed, if $\sigma_1 = \pm 1$, the inequality follows directly from \eqref{pi e cappucci}; if $\sigma_1 = 0$, then $\na_1=1$ and consequently $\na_l = 1\, \forall l$. Since the mass is conserved, the list $\na$ has at least two elements, and the inequality is achieved.



\medskip
\noindent
We finally need the following elementary result 
proved at the end of Appendix \ref{appendice tecnica}.

\begin{lemma}\label{luchino}
	Let $x_1\geq x_2\geq \ldots\geq x_N\geq 2.$ Then
	$$
	\frac{\sum_{1\leq\ell\leq N} x_\ell}{\prod_{1\leq\ell\leq N} \sqrt{x_\ell}}
	\leq 
	\sqrt{x_1}+\frac{4}{ \sqrt{x_1}}\,.
	$$
\end{lemma}

\subsection{Proof of estimate \eqref{barisciano}}

For $j\in\Z$ and $(\bal,\bbt)\in \cM_j$
by Lemma \ref{luchino}  we can write 
\begin{equation}\label{cacioricotta}
\frac{\jjap{j}^2}{\Pi_s \jjap{s}^{\bal_s + \bbt_s}} 
\le 
\frac{\jjap{\na_1}}{\prod_{l\ge 2} \jjap{\na_l}} 
\stackrel{\eqref{eleganza}}\le 
\frac{\sum_{l\ge 2}\jjap{\na_l}}{\prod_{l\ge 2} \jjap{\na_l}} 
\le 
\frac{4 + \jjap{\na_2}^{\frac12}}{\Pi_{l\ge 2} \jjap{\na_l}^{\frac12}} 
\le 
\frac{3}{\Pi_{l\ge 3}  \jjap{\na_l}^{\frac12}}\,.
\end{equation}

Now, by \eqref{ostiabeach}, \eqref{cacioricotta} and \eqref{diofantino nuc}
\begin{equation}\label{semper}
K_1'\leq 
K_1'':=
\sup_{j\in\Z} 
\sup_{(\bal,\bbt)\in \cM_j'} 
 \pa{\frac{3}{\Pi_{l\ge 3}  \jjap{\na_l}^{\frac12}}}^\delta 
\prod_{s\in\cS} \pa{1 + \jap{i(s)}^2 \abs{\bal_s - \bbt_s}^2}^{9/2}\,.
\end{equation}
Then \eqref{barisciano}
follows by 
 \begin{equation}\label{barisciano2}
\log	K_1''\leq c\exp(135\delta^{-1/\eta})\,.
\end{equation}
It remains to prove \eqref{barisciano2}.

In the following we call $j_1,j_2\in\cS^c$ with $|j_1|\geq |j_2|$
the possible normal sites.
Set
$$
\al_i=\bal_{s(i)}\,,\qquad
\bt_i=\bbt_{s(i)}\,,
$$
the quantity in the right hand side of 
\eqref{semper} reads
\begin{equation}\label{fidelis}
 \pa{\frac{3}{\Pi_{l\ge 3}  \jjap{\na_l}^{\frac12}}}^\delta 
\prod_{i\in\N } \pa{1 + \jap{i}^2 \abs{\al_i - \bt_i}^2}^{9/2}\,.
\end{equation}

%
%
%
%

By Lemma \ref{carciofo}, the constraint $\sum_{s\in \cS^c}\bal_s+\bbt_s\le 2$ implies that  there exists at least one $s\in \cS$ such that $\bal_s+\bbt_s\ne 0$. We denote the largest in absolute value $s\in \cS$ with this property as $s_\tM$
and $i_\tM:=i(s_\tM)$.   

\smallskip

In proving \eqref{barisciano2} we often meet 
the quantity 
\begin{equation}\label{cioli}
A_k(\delta):=\sum_{i\leq i_\tM:\, k_i\geq 1} - \frac{\delta}9  k_i\log\jjap{s(i)} + \log\pa{1 + \jap{i}^2 k_i^2}\,,
\end{equation}
defined for\footnote{Recall Definition 	\ref{baccala}.
Note that 
the sum  over $i$ is restricted to the indexes  such that $k_i\geq 1.$} $k\in\N^\Z$,
which is estimated by the following result,
whose proof is postponed in Appendix
\ref{appendice tecnica}.
\begin{lemma}\label{granita}
For every $k\in\N^\Z$ 
 $$
A_k(\delta)\leq 
c\exp(45\delta^{-1/\eta})\,.
$$
\end{lemma}

\smallskip

\noindent
We divide the proof of \eqref{barisciano2} in six cases
according to how many and in which position the normal sites 
appear in the list $\na$. 

\subsection*{Case  1.   $\na_2> s_\tM$} Here, since there are at most two normal sites we get 
\[
\Pi_{l\ge 3}  \jjap{\na_l} = \prod_{i\leq i_\tM} \jjap{s(i)}^{\al_i + \bt_i}\,.
\]
Recalling that 
 \[
 \jap{i(s)}\abs{\bal_s - \bbt_s}= { \jap{i} \abs{\al_i - \bt_i}}
\le  \jap{i}\pa{\al_i + \bt_i}  \,,\quad \mbox{if}\quad s=s(i)\,,\ \ i\in\N \,,
\]
we get
\begin{equation}\label{vvv}
\begin{aligned}
\log K_1''&\le
 \log (3^\delta )\\
&+\sup_{\al,\bt}\left(
 -\frac{\delta }{2} \sum_{i\leq i_\tM}(\al_i + \bt_i)\log\jjap{s(i)} + 
 {9/2} \sum_{i\leq i_\tM} \log\pa{1 + \jap{i}^2(\al_i + \bt_i)^2}
 \right)\\ &=
\log (3^\delta )
+{9/2}\sup_{k\in\N^\Z} A_k(\delta)
\leq 
c\exp(45\delta^{-1/\eta})
\,,
\end{aligned}
\end{equation}
where $A_k$ was defined in \eqref{cioli}, estimated in Lemma \ref{granita}.

\medskip

\subsection*{Case 2.  $\na_1 > s_\tM= \na_2$ and
 only one normal site}
 
 We have $\sum_{i\in \cS^c} \bal_i + \bbt_i =1$ 
 and the normal site must be $\na_1$. Moreover
\begin{equation}\label{checco}
\Pi_{l\ge 3}  \jjap{\na_l} =  \jjap{s(i_\tM)}^{\al_{i_\tM} + \bt_{i_\tM}-1} 
\prod_{i< i_\tM} \jjap{s(i)}^{\al_i + \bt_i}\,,
\end{equation}
so 
\begin{equation}\label{qulo}
\begin{aligned}
K_1''
&\le  3^\delta
\sup_{ \al,\bt } 
\jjap{s(i_\tM)}^{-(\al_{i_\tM} + \bt_{i_\tM}-1)\delta/2}
\pa{1 + \jap{i_\tM}^2 \abs{\al_{i_{\tM}} - \bt_{i_{\tM}}}^2}^{9/2}\\&\times
\prod_{i< i_{\tM}}  \jjap{s(i)}^{-(\al_i + \bt_i)\delta/2}\pa{1 + \jap{i}^2 \pa{\al_i + \bt_i}^2}^{9/2}.
\end{aligned}
\end{equation}
If $\al_{i_{\tM}}=\bt_{i_{\tM}}$ and hence $i_\tM$ does not appear in the small divisors, then we proceed as in case 1.
\\
\gr{(a)} 
If $\al_{i_{\tM}}+\bt_{i_{\tM}}\ge 2$ then we claim that
\begin{equation}\label{caffe}
\log\left(
\jjap{s(i_\tM)}^{-(\al_{i_\tM} + \bt_{i_\tM}-1)\delta/2}
\pa{1 + \jap{i_\tM}^2 \abs{\al_{i_{\tM}} - \bt_{i_{\tM}}}^2}^{9/2}\right)
\le 
\frac{c}{\delta}.
\end{equation}
In order to prove our claim we consider two cases $i_\tM\leq i_*$ and $i_\tM> i_*$.
In the first case, letting $x= \al_{i_{\tM}} + \bt_{i_{\tM}} -1 \ge 1$, the left hand side of \eqref{caffe} is bounded by
$$ 
\frac92\pa{
-\bar\delta  x \log 2 + \log\pa{1 + i_*^2(x + 1)^2}
}\qquad \mbox{with}\quad
\bar\delta:= \frac{\delta}{9}
$$
which is negative for $x\geq c/\bar\delta^2$ and, therefore, bounded
by  $c\log \bar\delta.$
\\
Consider now the case $i_\tM> i_*$.
By \eqref{cicoria} 
the left hand side of \eqref{caffe} is bounded 
by\footnote{Recalling that $i_\tM> i_*\geq 21$ and $\log 21\geq 3.$}
$$ 
\frac92\pa{
-\bar\delta  x \log^{1+\eta}i_\tM + \log\pa{1 + i_\tM^2(x + 1)^2}
}
\leq
\frac92 f(i_\tM,x)
\,,
$$
with $f$ defined in \eqref{vita}.
Reasoning as above we can estimate
$f(i_\tM,x)$ by $c/\bar\delta$ obtaining \eqref{caffe}.
By \eqref{qulo} and  \eqref{caffe}
we have
$$
\begin{aligned}
\log K_1'' &\le
 \log (3^\delta )
+\frac{c}{\delta}\\
&+\sup_{\al,\bt}\left(
 -\frac{\delta }{2} \sum_{i< i_\tM}(\al_i + \bt_i)\log\jjap{s(i)} + 
\frac92\sum_{i< i_\tM} \log\pa{1 + \jap{i}^2(\al_i + \bt_i)^2}
 \right)
 \\
 &
 \leq
 \log (3^\delta )
+\frac{c}{\delta}+
 \sup_{k\in\N^\Z} A_k(\delta)
\leq 
c\exp(45\delta^{-1/\eta})
 \,,
\end{aligned}
$$
by \eqref{cioli} and Lemma \ref{granita}.

\medskip
\noindent
\gr{(b)} If $\al_{i_{\tM}}+\bt_{i_{\tM}}=1$
 (then $|\al_{i_{\tM}}-\bt_{i_{\tM}}|=1$), 
 here the second factor in \eqref{qulo} is equal to one. Thus in order to bound the third factor (i.e.
$(1+ \langle i_\tM\rangle^2)^{9/2}\leq 2^{9/2} i_\tM^{9}$)
we distinguish two cases: $i_{\tM}\leq i_*$ and $i_{\tM}>i_*.$
In the first case $2^{9/2} i_\tM^{9}\leq 2^{9/2} i_*^{9}$
and the estimate of $K_1''$ in \eqref{qulo} proceeds as in case 1 above.
On the other hand when $i_{\tM}>i_*$ 
  we need a different argument.

  Given $u\in \Z^\Z$, with $|u|<\infty,$  consider the set
		\[
		\set{j\neq 0 \,,\quad \mbox{repeated}\quad  \abs{u_j} \;\mbox{times}}\,,
		\]
where $D<\infty$ 		is its cardinality.
Define the vector $m=m(u)$ as the reordering of the elements of the set above 
such that
 $|m_1|\ge |m_2|\ge \dots\geq |m_D|\ge 1.$ 	
 Given $\bal\neq\bbt\in\N^\Z,$ with $|\bal|=|\bbt|<\infty$
 we consider $m=m(\bal-\bbt)$ 
 and\footnote{ 
 The relation between $m$ and $\na$ is the following.	
If we denote by $D$ the cardinality of $m$ and by $N$ the one of $\na$,
respectively, we have 
$
D+\bal_0+\bbt_0\le N
$
 and 
$
(|m_1|,\dots,|m_D|,\underbrace{1,\;\dots \;,1}_{N-D\;\rm{times}} )\, \preceq\,
	 \pa{\na_1,\dots \na_N}
$, recalling Definition \ref{baccala}.
}
 $\na=\na(\bal+\bbt).$ 

\begin{lemma}[Lemma C.4 of \cite{BMP:2019}]
Given $\bal\neq\bbt\in\N^\Z,$ with $1\leq|\bal|=|\bbt|<\infty$
and satisfying  \eqref{divisor}, we 
have
\begin{equation}
\label{chiappechiare}
 \abs{m_1}\le  31\sum_{l\ge 3}\na_l^2\,. 
\end{equation}
\end{lemma}
\noindent
Note that
 \begin{equation}\label{abbiocco}
\al_{i_{\tM}}\neq\bt_{i_{\tM}}
\qquad
\implies
\qquad
s_\tM\leq |m_1|\,.
\end{equation}
In the present case
$|\al_{i_{\tM}}-\bt_{i_{\tM}}|=1$, by \eqref{abbiocco}
\begin{equation}\label{mars}
\begin{aligned}
&s(i_\tM) 
=
s_\tM \leq |m_1|\le 31 \sum_{l\ge 3} \na_l^2 \\
&=
 31 \sum_{i< i_\tM} \jjap{s(i)}^2 (\al_i+\beta_i)
 \leq 
  31 \sum_{i\leq i_*} \jjap{s(i)}^2 (\al_i+\beta_i)
\\ & +
    31 \sum_{i_*<i< i_\tM} s^2(i) (\al_i+\beta_i)
  \leq
31 \sum_{i\leq i_*} s_*^2 (\al_i+\beta_i)
  +
    31 \sum_{i_*<i< i_\tM} s^2(i) (\al_i+\beta_i)
\end{aligned}
\end{equation}
using that $s(i)$ is increasing.
By \eqref{cicoria}-\eqref{mustgoon}, for the inverse function $i(s)$ we have
for integer $j\geq 1$ and $s, s'\geq s_*:=s(i_*)$
\begin{equation}\label{innuendo}
i(s+s')\leq i(s)+i(s')\,,\qquad
i(js)\leq ji(s)\,,\qquad
i(s^2)\leq i^2(s)\,.
\end{equation}
Applying the inverse function $i(s)$ to the inequalities in \eqref{mars}
and using \eqref{innuendo} we get
\begin{equation}\label{king}
\begin{aligned}
i_\tM &\leq
31 i_*^2\sum_{i\leq i_*}^*  (\al_i+\beta_i)
  +
    31 \sum_{i_*<i< i_\tM}^* i^2 (\al_i+\beta_i)
    \leq 
    c \sum_{i< i_\tM}^* i^2 (\al_i+\beta_i)
    \\
  & \leq 
    c \prod_{i< i_{\tM}} \pa{1 + \jap{i}^2 \pa{\al_i + \bt_i}^2}
\end{aligned}
\end{equation}
where by $\displaystyle\sum^*$ we mean that the sum is only over
the indexes $i$ such that $\al_i+\beta_i\geq 1.$
By \eqref{qulo} we get
\begin{eqnarray}
\label{squlo}
\log K_1''
&\le& 
c+
\log( 3^\delta)
+
\nonumber
\\
&&
\sup_{ \al,\bt } 
\sum_{i< i_\tM}\left( 
-\frac{\delta}{2}
(\al_i + \bt_i)\log
\jjap{s(i)}
+
\frac{27}{2} \log \pa{1 + \jap{i}^2 \pa{\al_i + \bt_i}^2}
\right)
\nonumber
\\
&\leq&
c+
\log( 3^\delta)
+\frac{27}{2}
\sup_{k\in\N^\Z} A_k(\delta/3)
\leq
c\exp(135\delta^{-1/\eta})
\,,
\end{eqnarray}
again by \eqref{cioli} and Lemma \ref{granita}.
 
\noindent
\gr{Case  3. $\na_1 > s_\tM = \na_2$ and two normal sites.}
\\
Now
\begin{equation}\label{tre cappucci}
\Pi_{l\ge 3}  \jjap{\na_l} =  \jjap{j_2}\jjap{s(i_\tM)}^{\al_{h_\tM} + \bt_{h_\tM}-1} 
\prod_{i< i_\tM} \jjap{s(i)}^{\al_i + \bt_i}\,,
\end{equation}
where $j_2$ is the smallest\footnote{In absolute value.} normal site.
We have
\begin{equation}
\label{marocchino}
\begin{aligned}
K_1''
&\le  3^\delta
\jjap{j_2}^{-\delta/2}
\sup_{ \al,\bt } 
\jjap{s(i_\tM)}^{-(\al_{i_\tM} + \bt_{i_\tM}-1)\delta/2}
\pa{1 + \jap{i_\tM}^2 \abs{\al_{i_{\tM}} - \bt_{i_{\tM}}}^2}^{9/2}
\\ &\times
\prod_{i< i_{\tM}} \jjap{s(i)}^{-(\al_i + \bt_i)\delta/2}\pa{1 + \jap{i}^2 \pa{\al_i + \bt_i}^2}^{9/2}.
\end{aligned}
\end{equation}
\gr{(a)} If $\al_{i_{\tM}}=\bt_{i_{\tM}}$, or if $\al_{i_{\tM}}+\bt_{i_{\tM}}\ge 2$ then we proceed as in case 2-(a), since $\jjap{j_2}^{-\frac{\delta}{2}}\le 1$ and does not bother.\\
\gr{(b)} Let now $ \al_{i_{\tM}}+\bt_{i_{\tM}}=1$
(so that \eqref{abbiocco} holds). 
The analogous of \eqref{king} is 
\begin{equation}\label{crimson}
i_\tM\leq 
    c\Big(i(\jjap{j_2})\Big)^2 \prod_{i< i_{\tM}} \pa{1 + \jap{i}^2 \pa{\al_i + \bt_i}^2}\,.
\end{equation}
Then by \eqref{marocchino} we get
\begin{equation}\label{genesis}
\begin{aligned}
\log K_1''
&\le 
c+
\log( 3^\delta)
+
18 \log i(\jjap{j_2})
-
\frac{\delta}{2}\log \jjap{j_2}
\\
&+
\sup_{ \al,\bt } 
\sum_{i< i_\tM}\left( 
-\frac{\delta}{2}
\log
\jjap{s(i)}
+
\frac{27}{2} \log \pa{1 + \jap{i}^2 \pa{\al_i + \bt_i}^2}
\right)\,.
\end{aligned}
\end{equation}
Since
$$
\begin{aligned}
\sup_{x\geq 2}\left(18 \log i(x)
-
\frac{\delta}{2}\log x\right)
=
\sup_{y\geq i(2)}\left(18\log y
-
\frac{\delta}{2}\log s(y)\right)\\
\stackrel{\eqref{cicoria}}
\leq
c+
\sup_{y\geq i_*}\left(18 \log y
-
\frac{\delta}{2}\log^2 y \right)
\leq c+ \frac{162}{\delta}\,,
\end{aligned}
$$
we obtain
$$
\log K_1''
\le 
c+
\log( 3^\delta)
+
\frac{162}{\delta}
+\frac{27}{2}
\sup_{k\in\N^\Z} A_k(\delta/3)
\leq
c\exp(135\delta^{-1/\eta})
\,,
$$
again by \eqref{cioli} and Lemma \ref{granita}.

\subsection*{Case 4 $\na_1 = s_\tM$ and the (eventual) normal sites are $< \na_2$}
Recall that we have to estimate from above the quantity in \eqref{fidelis}.
\\
 Let us start with the case of two normal sites $j_1,j_2$.
 
 We start considering the case $\na_1=\na_2$,
 which implies $\al_{i_\tM} + \bt_{i_\tM}\geq 2$.
We get
\begin{equation}\label{tre cappucci normali}
\Pi_{l\ge 3}  \jjap{\na_l} =  \jjap{j_1} \jjap{j_2}\jjap{s(i_\tM)}^{\al_{i_\tM} + \bt_{i_\tM}-2} 
\prod_{i< i_\tM} \jjap{s(i)}^{\al_i + \bt_i}\,.
\end{equation}
\gr{(a)} If $\al_{i_{\tM}}=\bt_{i_{\tM}}$, or if $\al_{i_{\tM}}+\bt_{i_{\tM}}\ge 2$ then we are reduced to case 3-(a).\\
\gr{(b)} If $\al_{i_{\tM}}=2,\bt_{i_{\tM}} = 0$
(or viceversa), we proceed as in case 3-(b) and apply Lemma C.4\footnote{in the notation of \cite{BMP:2019} $\na_1 = |m_1|$}. 
Since, again, \eqref{abbiocco} holds, the analogous of \eqref{crimson} is 
\begin{equation}\label{kingcrimson}
\begin{aligned}
i_\tM &\leq 
    c\Big(i(\jjap{j_1})\Big)^2\Big(i(\jjap{j_2})\Big)^2 \prod_{i< i_{\tM}} \pa{1 + \jap{i}^2 \pa{\al_i + \bt_i}^2}\\
    &\leq
    c\Big(i(\jjap{j_1})\Big)^4 \prod_{i< i_{\tM}} \pa{1 + \jap{i}^2 \pa{\al_i + \bt_i}^2}
    \,.
\end{aligned}
\end{equation}
Then the analogous of \eqref{genesis} is
$$
\begin{aligned}
\log K_1''
&\le 
c+
\log( 3^\delta)
+
36 \log i(\jjap{j_1})
-
\frac{\delta}{2}\log \jjap{j_1}
\\
&+
\sup_{ \al,\bt } 
\sum_{i< i_\tM}\left( 
-\frac{\delta}{2}
\log
\jjap{s(i)}
+
\frac{27}{2} \log \pa{1 + \jap{i}^2 \pa{\al_i + \bt_i}^2}
\right)\,.
\end{aligned}
$$
We conclude as in case 3-(b).

We now consider the case $\na_1>\na_2$.
\\
Note that $\al_{i_\tM}+\bt_{i_\tM}=1$
(so that \eqref{abbiocco} holds).
Let us denote by $s_\tM'$ the second 
largest\footnote{Recall that the tangential sites
we are considering are positive.} 
tangential site and set $i_\tM':=i(s_\tM')$.
By construction $s_\tM'=\na_2$
and $\al_{i_\tM'}+\bt_{i_\tM'}\geq 1$.
Then
$$
\Pi_{l\ge 3}  \jjap{\na_l} =  \jjap{j_1} \jjap{j_2}\jjap{s(i_\tM')}^{\al_{i_\tM'} + \bt_{i_\tM'}-1} 
\prod_{i< i_\tM'} \jjap{s(i)}^{\al_i + \bt_i}\,.
$$
In this case the analogous of \eqref{kingcrimson}
is
$$
i_{\tM}\leq
c\Big(i(\jjap{j_1})\Big)^4 \prod_{i< i_{\tM'}} \pa{1 + \jap{i}^2 \pa{\al_i + \bt_i}^2}\,,
$$
which follows by \eqref{mars}
where  the second line holds with $i_{\tM'}$
instead of $i_{\tM}$.
Then we proceed as in the case $\na_1=\na_2$.
\\
Note 
If there is only one normal site or if there is none, then the same arguments apply word by word with the only "advantage" that there is only one $\jjap{j_1}$ or none in \eqref{tre cappucci normali}.

\subsection*{Case 5 $\na_1= s_\tM$ and only one normal site $j_1=\na_2$} Here \eqref{checco} holds and we proceed as in case 2.

\subsection*{Case 6 $\na_1= s_\tM$, $j_1=\na_2$ and two normal sites }
The proof follows word by word the one of case 3.
\\
This concludes the proof of \eqref{barisciano2}, which, by \eqref{semper}, implies \eqref{barisciano}.

\section{Iterative Lemma and Proof of Theorem \ref{allaMoserbis}}\label{AAA}

 Let $r, r_0,p,p_0, \rho, \delta$ be  as in \eqref{newton}
 and $1<\eta\leq 2$
 as in Definition \ref{orey}.
  Let $\{\rho_n\}_{n\in\N}, \{\delta_n\}_{n\in\N}$ be  the  
  summable\footnote{Note that $\frac{1+\eta}{2}>1.$}  sequences:
 	\begin{equation}\label{amaroni}
 	\rho_n= \frac{\rho}{{6}} 2^{-n}\,, \quad \delta_n = c_\eta\delta n^{-\frac{1+\eta}{2}}\quad \forall n\ge 1\,,
	\quad
	c_\eta^{-1}:=\frac{24}5\sum_{n\geq 1}n^{-\frac{1+\eta}{2}}>12\,,\; \delta_0 = \frac{\delta}{8},
 \end{equation} 
   Let us define recursively
   \begin{eqnarray}\label{pestob}
 &&r_{n+1} = r_n - {3}\rho_n\ \to \ r_\infty:=r_0-\rho \geq 7 r/4
 \qquad   {\rm (decr
 easing)}, \nonumber\\
 &&p_{n+1} = p_n + {3}\delta_n\ \to \ p_\infty:=
 p_0+\delta < p\qquad  {\rm (increasing)} \,,
\end{eqnarray}
recalling \eqref{newton}.

Set $\cO=\dg$, $\cI=\mathcal I(p,r)$  and $\cI(\C)= \cI(p,r,\C)$(recall  \eqref{diofantino nuc}, \eqref{pifferello} and \eqref{pifferone}). Since these sets are fixed, we omit to write them explicitly in the notations of this section. For instance we denote $ \cH_{r_n,p_n}=\cH^{\cO\times\cI}_{r_n,p_n}$  and $\cH_{r_n,p_n}(\C)=\cH^{\cO\times\cI(\C)}_{r_n,p_n}$. \\
By \eqref{ritornello} and \eqref{pestob}, we can use projections
$\Pi^0, \Pi^{-1}, \Pi^{\geq 1},$ on these spaces for all $n$. 
\begin{rmk}
A crucial point for the convergence of the algorithm is that, thanks to \eqref{ritornello}, no small divisor appears
in the estimate of the counter-terms, see \eqref{orte} and \eqref{lambdozzo} below.
\end{rmk}

\medskip

Let  
\begin{equation}\label{bisanzio}
H_0 := D + G_0 + \Lambda_0\,,
\qquad
G_0\in \cH_{r_0,p_0}(\C)\,,\qquad
\Lambda_0\in \scH_{r,p}^{0,\cK}\,,
\end{equation}
where $D$ is defined in \eqref{uffa}
and the counter-terms $\Lambda_0 = \sum_{j\in\Z}\lambda_{0,j}(|u_j|^2 - I_j)$ 
 with $\lambda_0 = (\lambda_{0,j})_{j\in\Z}$ free parameters in $\ell^\infty$. 
We define 
\begin{equation}\label{vigili}
\begin{aligned}
\e_0&:=\gamma^{-1}\pa{\norma{\zeroK{G_0}}_\infty + \norma{\zeroR{G_0}}_{r_0, p_0} + \norma{\due{G_0}}_{r_0, p_0} + {\norma{G_0^{(-1)}}_{r_0, p_0}}} \\ \Theta_0 &:= \gamma^{-1}\norma{\buon{G_0}}_{r_0, p_0} +\e_0. 
\end{aligned}
\end{equation}

\begin{lemma}[Iterative step]\label{iterativo}
 There exists  a constant $\mathfrak C>1$ large enough\footnote{Depending only on $i_*,\eta$ defined in \eqref{cicoria}.}
 such that
if 
\begin{eqnarray}\label{gianna}
&\eps_0 	\leq \pa{1 + \Theta_0}^{-5} \tK^{-{3}}\,,
\quad 
\tK
:=   \pa{\frac{r_0}{\rho}}^6\sup_{n\ge 1} 2^{6n} 
\exp\left(\cachi^{n^\xi}
  -\chi^n {(1-\chi/2) }\right)\,,
\nonumber
\\
& \mbox{where}\;
\cachi:= \exp\left(\frac{\mathfrak C}{ \delta^{1/\eta}}\right)\,,\quad
  \xi:=\frac{1+\eta}{2\eta}<1\,,
  \qquad
  \chi:=3/2\,,
\end{eqnarray}
with $\eta$ as in \eqref{cicoria}, then we can iteratively construct a sequence of generating functions 
$S_i = \due{S_i} + S_i^{(-1)}+ \zero{S_i}\in\cH_{r_{i}-\rho_i, p_{i+1}}(\C)$ 
and a sequence of  counter-terms  $\bar\Lambda_i\in\cH^{0,\cK}(\C)$ such that the following holds, for $n\ge 0$.

$(1)$ For all $\omega\in \dg$, $I\in \cI(p,r)$, for all $ i = 0,\ldots,  n -1 $ and all $ p'\ge p_{i+1}$ the 
time-1 Hamiltonian flow
 $\Phi_{S_i}$ generated by $S_i=S_i(\omega,I)$   satisfies
	\begin{equation}
\sup_{u\in  {\bar B}_{r_{i+1}}(\tw_{p'})} \norm{\Phi_{S_i}(u)- u}_{p'} \le \rho 2^{-2i-7} \,.\label{ln}
 \end{equation}
 Moreover
	\begin{equation}\label{ucazzo}
	\Psi_n := \Phi_{S_0}\circ\cdots \circ \Phi_{S_{n-1}} 
	\end{equation}
	is a well defined, analytic map ${\bar B}_{r_n }(\tw_{p'}) \to {\bar B}_{r_0}(\tw_{p'})$ for all $p'\ge p_n $ with the bound
	\begin{equation}
	\label{cosi}
	 \sup_{u\in {\bar B}_{r_n}(\tw_{p'})}\abs{\Psi_{n}(u) - \Psi_{n-1}(u)}_{p'}  \le  \rho 2^{-2n -3}.
	\end{equation}

	$(2)$ We set $\cL_0:=0$  and for $i=1,\dots,n$ 
	\begin{equation}\label{emorroidi}
	 \mathcal{L}_{i} + \id := e^{\set{S_{i-1},\cdot}}\pa{\mathcal{L}_{i-1} + \id},\quad \Lambda_{i} := \Lambda_{i-1} - \bar{\Lambda}_{i-1}\,,\quad H_i= e^{\{S_{i-1},\cdot\}}H_{i-1}
	 \end{equation}
	  where $\Lambda_{i}\in \scH_{r_{i},p_i}^{0,\cK}$ are free parameters and $\mathcal{L}_{i} : \cH^{0,\cK}(\C)\to\cH_{r_{i}, p_{i}}(\C)$ 
are bounded linear operators.  Note that $\Lambda_i$ are free parameters, while $\bar\Lambda_i$ are given functions of $(\omega,I)$.
	 We have
	\begin{equation}
\label{cioccolato}
H_{i} = D + G_{i} + \pa{\id + \mathcal{L}_{i}}\Lambda_{i},\qquad G_{i}, \in\cH_{r_{i}, p_{i}}(\C).
\end{equation}
 Setting for $ i = 0,\ldots, n  $
	\begin{equation}\label{xhx-i}
	\begin{aligned}
	 \eps_i&:=\gamma^{-1}\pa{\norma{\zeroK{G_i}}_\infty + \norma{\zeroR{G_i}}_{r_i, p_i} + \norma{\due{G_i}}_{r_i, p_i} + {\norma{G_i^{(-1)}}_{r_i, p_i}}},  \\ \Theta_i&:=\gamma^{-1} {\norma{G_i^{\ge 1}}_{r_i,p_i}} +\e_i \,,
	\end{aligned}
	\end{equation}
we have
\begin{eqnarray}
& \e_i \leq   \e_0  e^{- \chi^{i}+1} \,, 
\qquad
\qquad  \Theta_i \leq   \Theta_0 \sum_{j=0}^i 2^{-j}\,, \label{en} \\
& 
 \label{fringe}
\norma{\pa{\call_{i} -\call_{i-1}} h}_{r_i,p_i}
\le 
\tK \eps_0 \pa{1 + \Theta_0}^2 2^{-i} \norma{h}_\infty,
\\ &\nonumber
\norma{\call_i h}_{r_i,p_i} \le \tK (1+  \Theta_0)^2\e_0\sum_{j=1}^i 2^{-j}\norma{h}_\infty,
\end{eqnarray}
for all $h\in\cH^{0,\cK}(\C).$
Finally the counter-terms satisfy the bound
 \begin{equation}
 \norma{\bar{\Lambda}_{i-1}}_\infty \le  \g \tK \e_{i-1}(1+\Theta_0)^{{2}}\,,\quad i=1,\dots,n \,. \label{intestino}
	\end{equation}
\end{lemma}

\begin{proof}[Proof of Theorem \ref{allaMoserbis}] Starting from the Hamiltonian $H$ satisfying \eqref{maipendanti}, we set $G_0 = H - D$ in \eqref{bisanzio}. The smallness conditions \eqref{gianna} are met, provided that we choose $\bar\epsilon$ and $\bar C$ appropriately. \\
Using \eqref{cosi} we define $\Psi$ as the limit of the $\Psi_n$ (which define a Cauchy sequence) and $\Lambda = \Lambda_0 = \sum_j \bar\Lambda_j < \infty$. Note that the series is summable by \eqref{intestino}. For more details see \cite[Section $6$]{BMP:almost}.
\end{proof}

\begin{proof}[Proof of the iterative Lemma]
We start with a Hamiltonian
	$H_0= D+ \Lambda_0 +G_0$
with $\Lambda_0\in \cH^{0,\cK}$ and $D$.

At the $n$'th step we have an expression of the form
\[
H_n=D+\pa{\id+\call_n}\Lambda_n+G_n,
\]
with $G_n\in \cH_{r_n,p_n}$.
\\
To proceed to the step $n+1$ we apply the change of variables $e^{\{S_n,\cdot\}}$. The generating function $S_n$ and the counter-term $\bar\Lambda_n$ are fixed as the unique solutions of the Homological equation
\begin{equation}\label{homo sapiens}
\Pi^{\le 0}  \pa{\set{S_n,D+G_n^{\ge 1}} +\pa{\id+\call_n}\bar\Lambda_n+G_n }= G_{n}^{(-2, \cK)}\,,
\end{equation}
recalling that $G_{n}^{(-1, \cK)}=0$ by \eqref{scalogno}.
This equation can be written component-wise as a triangular system and solved consequently.  Indeed
we have
\begin{align}
	& {\set{S_n^{(-2)},D} +\Pi^{-2,\cR}\call_n\bar\Lambda_n+G^{(-2,\cR)}_n }=0,\\
	&\set{S_n^{(-1)},D}+\Pi^{-1}\set{S_n^{(-2)},G_n^{\ge 1}} +\Pi^{-1}\call_n\bar\Lambda_n+G^{(-1)}_n =0,\\
	\label{counter}
	&\Pi^{0,\cK}\set{S_n^{(-2)}+ S_n^{(-1)},G_n^{\ge 1}} +\bar{\Lambda}_n+\Pi^{0,\cK}\call_n\bar\Lambda_n+G^{(0,\cK)}_n=0,\\
	&\set{S_n^{(0,\cR)},D}+\Pi^{0,\cR}\set{S_n^{(-2)}+ S_n^{(-1)},G_n^{\ge 1}} +\Pi^{0,\cR}\call_n\bar\Lambda_n+G^{(0,\cR)}_n = 0\,.
\end{align}
We start by solving the equations for $S_n$ it "modulo $\bar\Lambda_n$", then we determine the counter-term by inversion of an appropriate linear operator resulting from inserting the equations for $S_n$ into equation \eqref{counter}. \\
We hence have by Proposition \ref{mah}
\begin{align}
	& S_n^{(-2)}= L_{\betta}^{-1} \pa{\Pi^{-2}\call_n\bar\Lambda_n+G^{(-2)}_n },\label{ostrica}\\
	&S_n^{(-1)}= L_{\betta}^{-1}\pa{\Pi^{-1}\set{
	S_n^{(-2)},
	G_n^{\ge 1}} +\Pi^{-1}\call_n\bar\Lambda_n+G^{(-1)}_n},\nonumber \\
	& S_n^{(0,\cR)} = L_{\betta}^{-1} \pa{\Pi^{0,\cR}\set{S_n^{(-2)}+ S_n^{(-1)},G_n^{\ge 1}} +\Pi^{0,\cR}\call_n\bar\Lambda_n+G^{(0,\cR)}_n }\,.\nonumber
\end{align}
Plugging them into \eqref{counter} we thus get
\begin{align*}
	&\Pi^{0,\cK}\set{L_{\betta}^{-1} \pa{\Pi^{\leq -1}\call_n\bar\Lambda_n  +\Pi^{-1}\set{L_{\betta}^{-1} {\Pi^{-2}\call_n\bar\Lambda_n },G_n^{\ge 1}}},G_n^{\ge 1}} +\bar{\Lambda}_n+\Pi^{0,\cK}\call_n\bar\Lambda_n=\\
	& -\Pi^{0,\cK}\set{L_{\betta}^{-1} \pa{G^{\leq -1}_n  +\Pi^{-1}\set{L_{\betta}^{-1} {G^{(-2)}_n },G_n^{\ge 1}}},G_n^{\ge 1}}-G^{(0,\cK)}_n\,.
\end{align*}
Note that the left hand side of the equation above can be written as $(\id+ M_n)\bar{\Lambda}_n$, where  $M_n: \cH^{0,\cK}\to \cH^{0,\cK}$  is 
\begin{equation}\label{lingua}
\begin{aligned}
M_n h:= \Pi^{0,\cK}\set{L_{\betta}^{-1} \pa{\Pi^{\leq -1}\call_n h  +\Pi^{-1}\set{L_{\betta}^{-1} {\Pi^{-2}\call_n h },G_n^{\ge 1}}},G_n^{\ge 1}} +\Pi^{0,\cK}\call_n h\,.
\end{aligned}
\end{equation}
Similarly to Lemma 6.2 of \cite{BMP:almost} one has 
\begin{equation}\label{orte}
\norma{M_n h}_{\infty} \le \norma{ h}_{\infty}/2\,.
\end{equation}

In order to prove \eqref{orte}, we treat the three  summands of $M_n$ separately, we recall
that by \eqref{ritornello}, \eqref{fringe} and  by the smallness condition in \eqref{gianna}
\[
\|{\pon \mathcal{L}_n} h\|_{\infty}\le 3 \|{ \mathcal{L}_n} h\|_{r_n,p_n}\le 3 \tK (1+\Theta)^2\e_0\sum_{j=1}^n 2^{-j} \|h\|_{\infty}< \frac16 \|h\|_{\infty}.
\]
In  the other  summands  we use the identification $  \|\Pi^{0,\cK} F\|_\infty  \le 3\|F\|_{r',p'} $ for any $r',p'$ (see  formulas \eqref{silvestre} and \eqref{ritornello}   such that
$r'\ge r \sqrt{2}$ and $p'\le p$.
\\
We have
 by \eqref{che palle},  \eqref{ritornello},  Propositions \ref{fan} and \ref{mah} 
\begin{equation} \label{otricoli}
\begin{aligned}
	&	\| \pon \set{ \Pi^{\leq -1} L^{-1}_\omega \mathcal{L}_n{h},G_n^{\geq 1}}\|_{\infty} \le 3
	\|\set{ \Pi^{\leq -1} L^{-1}_\omega \mathcal{L}_n{h},G_n^{\geq 1}}\|_{\sqrt 2 r,p} 
	\\
	&
	\stackrel{\eqref{commXHK},\eqref{pestob}}\le 
	 120 \norma{\Pi^{\leq -1} L^{-1}_\omega \mathcal{L}_n{h}}_{r_n, p}
	\norma{G_n^{\geq 1}}_{r_n, p}
	\stackrel{\eqref{xhx-i}}\leq
	 240 \g \norma{ L^{-1}_\omega \mathcal{L}_n{h}}_{r_n, p }\Theta_n
	 \\
	&
	\stackrel{\eqref{newton},\eqref{pestob},\eqref{ss148}}
	\le 240 
\exp\left(  \exp\left(\frac{c}{(3\delta)^{1/\eta}}\right)\right)\norma{  \mathcal{L}_n{h}}_{r_n,p_n} \Theta_n \\
&\stackrel{\eqref{gianna},\eqref{fringe}}{\leq }
\frac16
 \tK^2 \e_0(1+\Theta_0)^3 \|h\|_{\infty}
 \stackrel{\eqref{gianna}}\leq \frac16 \|h\|_{\infty}
 \,,
\end{aligned}
\end{equation}
where, in order to control the exponential term, we used the definition of
$\tK$ given in \eqref{gianna}.
\\
Now we estimate the remaining term in \eqref{lingua}.
We have,
again by \eqref{che palle}, \eqref{ritornello}, \eqref{xhx-i}, Proposition \ref{fan} and Lemma \eqref{mah} 
\begin{eqnarray*}
&&\| \Pi^{0,\cK}\set{L_{\betta}^{-1} \Pi^{-1}\set{L_{\betta}^{-1} {\Pi^{-2}\call_n h },G_n^{\ge 1}},G_n^{\ge 1}}\|_\infty\\
&&\leq 3\|\set{L_{\betta}^{-1} \Pi^{-1}\set{L_{\betta}^{-1} {\Pi^{-2}\call_n h },G_n^{\ge 1}},G_n^{\ge 1}}\|_{\sqrt 2 r,p}
\\
&&
\stackrel{\eqref{commXHK},\eqref{pestob}}\le 
400\|L_{\betta}^{-1} \Pi^{-1}\set{L_{\betta}^{-1} {\Pi^{-2}\call_n h },G_n^{\ge 1}}\|_{3r/2,p}\|G_n^{\ge 1}\|_{3r/2,p}
\\
&&
\stackrel{\eqref{xhx-i}}\leq
	 400 \g 
	\|L_{\betta}^{-1} \Pi^{-1}\set{L_{\betta}^{-1} {\Pi^{-2}\call_n h },G_n^{\ge 1}}\|_{3r/2,p} 
	\ \Theta_n
	 \\
&& \stackrel{\eqref{ritornello},\eqref{ss148}}
\le \!\!\!\!\!\!\!\!
400  \exp\left( c \exp\left(\left(\frac{4}{3 \delta }\right)^{1/\eta}\right)\right)\|\set{L_{\betta}^{-1} {\Pi^{-2}\call_n h },G_n^{\ge 1}}\|_{3r/2,(p+ p_n)/2}\ \Theta_n
	 \\
&& \stackrel{\eqref{commXHK},\eqref{xhx-i}}
\le 
\ 2^{14}\g\ \exp\left( c \exp\left(\left(\frac{4}{3 \delta }\right)^{1/\eta}\right)\right)
\|L_{\betta}^{-1} {\Pi^{-2}\call_n h }\|_{r_n,(p+ p_n)/2}\ \Theta_n^2
	 \\
&& \stackrel{\eqref{ritornello},\eqref{ss148}}
\le 
2^{14}\  \exp\left(2 c \exp\left(\left(\frac{4}{3 \delta }\right)^{1/\eta}\right)\right)
\| \call_n h \|_{r_n,p_n}\ \Theta_n^2\\ &&
\stackrel{\eqref{gianna},\eqref{fringe}}{\leq }
\frac16
 \tK^2 \e_0(1+\Theta_0)^4 \|h\|_{\infty}
 \stackrel{\eqref{gianna}}\leq \frac16 \|h\|_{\infty},
\end{eqnarray*}
noting that
by \eqref{newton} and \eqref{pestob}
\[
r_n-\frac32 r\geq r_\infty-\frac32 r\geq \frac{r}{4}\,,
\qquad
\frac{ p - p_n}2\geq \frac{p-p_\infty}2  \ge  \frac32 \delta, 
\]
and estimating the exponential term as in \eqref{otricoli}.
This concludes the proof of
\eqref{orte}.

Then we have that:
 \begin{equation}
 \label{lambdozzo}
\begin{aligned}
\bar{\Lambda}_n 
=&-(\id +M_n)^{-1}\
\Pi^{0,\cK}\big\{L_{\betta}^{-1} \pa{G^{\leq -1}_n \Pi^{-1}\set{L_{\betta}^{-1} {G^{(-2)}_n },G_n^{\ge 1}}},G_n^{\ge 1}\big\}\\ &
-(\id +M_n)^{-1} G^{(0,\cK)}_n
\end{aligned}
 \end{equation}
 is well defined.
 In order to prove \eqref{intestino} we split \eqref{lambdozzo} in three terms and note that by \eqref{orte}
 the operator norm of $(\id +M_n)^{-1}$ is bounded by 2.
 Regarding the first one we obtain 
 \begin{align*}
	&	\|(\id +M_n)^{-1} \pon \set{  L^{-1}_\omega G_n^{\le -1},G_n^{\geq 1}}\|_{\infty} \le 6
	\|\set{  L^{-1}_\omega G_n^{\le -1},G_n^{\geq 1}}\|_{\sqrt 2 r,p} 
	\\
	&
	\stackrel{\eqref{commXHK},\eqref{pestob}}\le 
	 240 \norma{ L^{-1}_\omega G_n^{\le -1}}_{r_n, p}
	\norma{G_n^{\geq 1}}_{r_n, p}
	\stackrel{\eqref{xhx-i}}\leq
	 240 \g \norma{ L^{-1}_\omega G_n^{\le -1}}_{r_n, p }\Theta_n
	 \\
	&
	\stackrel{\eqref{newton},\eqref{pestob},\eqref{ss148}}\le 240
\exp\left( c \exp\left(\left(\frac{2}{3\delta}\right)^{1/\eta}\right)\right)\norma{  G_n^{\le -1}}_{r_n,p_n} \Theta_n \\ &
\stackrel{\eqref{gianna},\eqref{xhx-i},\eqref{fringe}}
\leq 
\frac14
 \tK \g \e_n   (1+\Theta_0)
 \,,
\end{align*}
taking $\frak C$ large enough here and below.
 Regarding the second term we get
 \begin{eqnarray*}
&&\| (\id +M_n)^{-1}\Pi^{0,\cK}\set{L_{\betta}^{-1} \Pi^{-1}\set{L_{\betta}^{-1} {G^{(-2)}_n },G_n^{\ge 1}},G_n^{\ge 1}}\|_\infty\\
&&
\leq 6\|\set{L_{\betta}^{-1} \Pi^{-1}\set{L_{\betta}^{-1} {G^{(-2)}_n },G_n^{\ge 1}},G_n^{\ge 1}}\|_{\sqrt 2 r,p}
\\
&&
\stackrel{\eqref{commXHK},\eqref{pestob}}\le 
800\|L_{\betta}^{-1} \Pi^{-1}\set{L_{\betta}^{-1} {G^{(-2)}_n },G_n^{\ge 1}}\|_{3r/2,p}\|G_n^{\ge 1}\|_{3r/2,p}
\\
&&
\stackrel{\eqref{xhx-i}}\leq
	 800 \g 
	\|L_{\betta}^{-1} \Pi^{-1}\set{L_{\betta}^{-1} {G^{(-2)}_n },G_n^{\ge 1}}\|_{3r/2,p} 
	\ \Theta_n
	 \\
&& \stackrel{\eqref{ritornello},\eqref{ss148}}
\le \!\!\!\!\!\!
800  \exp\left(  \exp\left(c\left(\frac{2}{3 \delta }\right)^{1/\eta}\right)\right)\|\set{L_{\betta}^{-1} {G^{(-2)}_n },G_n^{\ge 1}}\|_{3r/2,(p+ p_n)/2}\ \Theta_n
	 \\
&& \stackrel{\eqref{commXHK},\eqref{xhx-i}}
\le 
\ 2^{15}\g\ \exp\left( \exp\left(c\left(\frac{2}{3 \delta }\right)^{1/\eta}\right)\right)
\|L_{\betta}^{-1} {G^{(-2)}_n }\|_{r_n,(p+ p_n)/2}\ \Theta_n^2
	 \\
&& \stackrel{\eqref{ritornello},\eqref{ss148}}
\le 
2^{15}\  \exp\left(2 \exp\left(c\left(\frac{2}{3 \delta }\right)^{1/\eta}\right)\right)
\| G^{(-2)}_n \|_{r_n,p_n}\ \Theta_n^2\\ &&
\stackrel{\eqref{gianna},\eqref{xhx-i},\eqref{fringe}}{\leq }
\frac14
 \tK\g \e_n (1+\Theta_0)^2\,.
\end{eqnarray*}
Finally for the third term we get  
$$
\|(\id+M_n)^{-1}G_n^{0,\cK}\|_\infty
\stackrel{\eqref{xhx-i}}{\leq }
  2\g \e_n,
$$
 so \eqref{intestino} follows.
%

By substituting  in the equations \eqref{ostrica} 
 we get the final expressions for $S^{(-2)}_{n}$ and $S^{(-1)}_n$ and finally $S^{(0,\cR)}_n$ which by
\eqref{commXHK}, \eqref{ss148}, \eqref{xhx-i}, \eqref{fringe}, \eqref{gianna} and \eqref{intestino}
 yield  the estimates 
\begin{align}
\|\due{S_n}\|_{r_n, p_n + \delta_n} &\le  (1+\tK^2(1+\Theta_0)^4 \e_0) \tD_n \eps_n
\leq 2 \tD_n \eps_n
\nonumber
\\ 
\|S_n^{(-1)}\|_{r_n - \rho_n, p_n + 2\delta_n}  
&\le 
(1+16 r_0 \rho_n^{-1}\tD_n \Theta_0)2 \tD_n\eps_n 
\nonumber
\\
\|S_n^{(0)}\|_{r_n - 2\rho_n, p_n + 3\delta_n} 
& \le 
(1+16 r_0 \rho_n^{-1}\tD_n \Theta_0)^2 3 \tD_n\eps_n 
\label{coratella}
\end{align}
where 
 $$
 \tD_n :=  
 \exp\left(  \exp\left(\frac{c}{\delta_n^{1/\eta}}\right)\right)
 \leq
  \exp\left( \frac13 \cachi^{n^\xi}\right)
 $$ 
 ($ \cachi$ and $\xi$ were defined in \eqref{gianna})
 systematically using the inductive hypothesis and the first bound  in \eqref{gianna}.
 The final bound thus reads (recall \eqref{amaroni} and \eqref{pestob}) 
 
 \begin{equation}
 \label{pasta frolla}
 \|S_n\|_{r_n - 2\rho_n, p_{n+1}} 
 \leq   
 \frac{r_0^2}{\rho^2}4^{n+8}
 \exp\left( \cachi^{n^\xi}\right)
 \eps_n (1 + \Theta_0)^2
 \stackrel{\eqref{gianna},\eqref{en}}\leq
\frac{\rho}{2^{2n+10}r_0}\,.
 \end{equation}
 Then we can apply Proposition \ref{ham flow} since \eqref{stima generatrice} is satisfied by $S_n$
 with $\rho\to\rho_n$ and $r\to r_{n+1}.$
 Then  item (1) of Lemma \ref{iterativo} is easily  proved. In particular
 \eqref{ln} follows by \eqref{pollon} and \eqref{pasta frolla}
 (for complete details see the analogous proof of \cite[Lemma 6.1]{BMP:almost}).

  Regarding item (2), by construction we have
 $$
 \call_{n+1} - \call_{n} = \pa{e^\set{S_n,\cdot} - \id} \circ (\call_n + \id),
 $$
 hence by \eqref{caio}, \eqref{stima generatrice}, \eqref{amaroni} and \eqref{silvestre}
 \begin{equation}\label{dattero}
\begin{aligned}
 \norma{ (\call_{n+1} - \call_{n})h }_{r_{n+1}, p_{n+1}} 
\leq &
\frac{r_0 2^{n+9}}{\rho}
\|S_n\|_{r_n-2\rho_n,p_{n+1}} 
\|(\call_n + \id)h\|_{r_n-2\rho_n,p_{n+1}} 
\\
\stackrel{\eqref{pasta frolla},\eqref{gianna},\eqref{fringe}}\leq
 &\frac{r_0^3}{\rho^3}8^{n+9} 
 \exp\left(\cachi^{n^\xi}\right)
 \eps_n (1 + \Theta_0)^2 \norma{h}_\infty,
\end{aligned}
 \end{equation}
 which by \eqref{gianna} proves \eqref{fringe}.

As for the expression of $G_{n+1}$, by \eqref{cioccolato} and \eqref{emorroidi} we have 
\[G_{n+1} = e^{\set{S_n,\cdot}} H_n - \sq{D_\betta + \pa{\id + \call_{n+1}}\Lambda_{n+1}}\,.
\]
Since $S_n$ solves the Homological equation \eqref{homo sapiens}, 
we have that by \eqref{emorroidi}
\begin{align}\label{schifo al cazzo}
G_{n+1} &= \dueK{G_n} + G_n^{\ge 1}+\Pi^{\ge 1}\pa{\call_{n+1} \bar{\Lambda}_n+\set{S_n,G_n^{\ge 1}}}  + G_{n+1,*},
	\\ G_{n+1,*} &=  \{S_n,G_n^{\le 0}\} 
	+ \Pi^{\le 0}\pa{\call_{n+1}-\call_n}\bar{\Lambda}_n+
	\pa{e^{\{S_n,\cdot\}} - \id - \set{S_n,\cdot}}
	G_n  \nonumber \\
	&  - \sum_{h=2}^\infty \frac{ (\ad S_n)^{h-1}}{h!} \pa{\Pi^{\le 0}\pa{\id + \call_n}\bar\Lambda_n + G_n^{\le 0} +\Pi^{\le 0} \{S_n^{(-1)} + \due{S_n},G_n^{\ge 1}\} }, \nonumber
\end{align}
using that 
$\Pi^{\le 0}  \pa{\set{S_n,D}}=\pa{\set{S_n,D}}$ by \eqref{creep}
and that $\set{S_n,  G_n^{(-2,\cK)}}=0$.
Note  that $G_{n+1,*}$ is  quadratic in $S_n \sim G_n^{\le 0}$.
\\
Recalling \eqref{amaroni}, \eqref{pestob}, \eqref{commXHK}, \eqref{ritornello}, \eqref{gianna}, 
\eqref{xhx-i}, \eqref{fringe}, \eqref{intestino} and 
Proposition \ref{ham flow}, which can be applied by \eqref{pasta frolla}, we have
\begin{align*}
& \norma{G_{n+1,*}}_{r_{n+1}, p_{n+1}} 
 \stackrel{\eqref{dattero}}\leq 
\frac{2^{n+6}r_0}{\rho}
\|
S_n
\|_{r_n-2\rho_n,p_{n+1}}
\g\eps_n 
\\
& + \frac{r_0^3}{\rho^3}8^{n+9} 
 \exp\left(\cachi^{n^\xi}\right)
 \eps_n (1 + \Theta_0)^2
 \norma{\bar{\Lambda}_n}_\infty
+\frac{2^{2n+21}r_0^2}{\rho^2}
\|
S_n
\|_{r_n-2\rho_n,p_{n+1}}^2
\g\Theta_0 \\ 
& +
\frac{2^{n+10}r_0}{\rho}
\|
\Pi^{\le 0}\pa{\id + \call_n}\bar\Lambda_n + G_n^{\le 0} \\
& +\Pi^{\le 0} \{S_n^{(-1)} + \due{S_n},G_n^{\ge 1}\}
\|_{r_n-2\rho_n,p_{n+1}}
\|
S_n
\|_{r_n-2\rho_n,p_{n+1}}
\\
&\stackrel{\eqref{pasta frolla},\eqref{coratella}}\leq   
2^{6n+55}
 \frac{r_0^6}{\rho^6}
 \exp\left( 2\cachi^{n^\xi}\right)
 \eps_n^2 (1 + \Theta_0)^5\g
 \\
&+
\frac{2^{n+24}r_0}{\rho}
\left(\|
\bar\Lambda_n\|_\infty + \g \eps_n +2^n\frac{r_0^2}{\rho^2} \Theta_0^2 \eps_n
 \exp\left( \cachi^{n^\xi}\right)
 \g
\right)
\|
S_n
\|_{r_n-2\rho_n,p_{n+1}}
 \\
&\stackrel{\eqref{pasta frolla},\eqref{intestino}}\leq 
2^{6n+55}
 \frac{r_0^6}{\rho^6}
 \left(\exp\left( \cachi^{n^\xi}\right)+\tK\right)
 \exp\left( \cachi^{n^\xi}\right)
 \eps_n^2 (1 + \Theta_0)^5\g
  \\
&\stackrel{\eqref{gianna}}\leq 
2^{57}
\frac{\g}{\tK^2}
 \left(\exp\left( \cachi^{n^\xi}\right)+1\right)
 e^{-\chi^{n+1}/2}
 \eps_n 
 \leq 
 2^{60}
\frac{\g}{\tK}e^{-\chi^{n+1}}\eps_0
\leq \g e^{-\chi^{n+1}}\eps_0\,,
\end{align*}

taking ${\frak C}$ large enough.
Recalling \eqref{xhx-i} and \eqref{schifo al cazzo} 
this implies\footnote{Recall that the semi-norm of the constant term $\dueK{G_n} $ is zero.} the first estimates in \eqref{en}.
\\
By the above estimate and
recalling  \eqref{xhx-i} and \eqref{schifo al cazzo} we get
\begin{align*}
& \Theta_{n+1}
\leq 
 \e_{n+1}+\Theta_n+e^{-\chi^{n+1}}\eps_0
+\g^{-1}\|\Pi^{\ge 1}\pa{\call_{n+1} \bar{\Lambda}_n+\set{S_n,G_n^{\ge 1}}}\|_{r_{n+1},p_{n+1}}
\\
&\leq
\Theta_n+4 e^{-\chi^{n+1}}\eps_0
+
2^4\e_n
+2^{n+7}r_0\rho^{-1}
\|
S_n
\|_{r_n-2\rho_n,p_{n+1}}\Theta_0
\leq \Theta_n+\Theta_0 2^{-n-1},
\end{align*}
again by \eqref{ritornello}, \eqref{gianna}, \eqref{fringe}, \eqref{intestino}, \eqref{commXHK}, \eqref{pasta frolla}.
This finally implies the second estimates in \eqref{en}.
\\
The analyticity of $\Psi_n$ and $\Lambda_n$ follows by Proposition \ref{mah}, point (ii)
 of Proposition
\ref{proiettotutto} and 
recalling Remark \ref{cotechino}.
\end{proof}

\section{Measure estimates}\label{bellogrosso}

\begin{lemma}[Measure estimates]\label{misurino} The set
$\cC$
defined in \eqref{arancini} satisfies \eqref{mirto}, namely
\begin{equation*}
\begin{aligned}
	&{\rm meas}_{\cQ_\cS}(\cQ_\cS \setminus \cC(\VSc,I,\g))
	 \le C_0\g\,,\\
	&{\rm meas}_{[-\nicefrac14,\nicefrac14]^\cS} \Big(\VSf\big(\cQ_\cS \setminus \cC(\VSc,I,\g),\VSc,I\big)\Big) \le C_0\g\,.
	\end{aligned}
	\end{equation*} 
\end{lemma} 

\begin{proof}
We start proving the first estimate in \eqref{mirto}.
\\
	Take $\gamma\leq 1/2.$
For $\ell\ne 0$ with $|\ell|<\infty$, $\sum_{s\in \cS^c} |\ell_s|\le 2$, $\pi(\ell)=0$ and $\fm(\ell)=0$, set for simplicity $\cR_\ell=\cR_\ell(\VSc,I,\gamma)$ and 
 define the 
resonant set
\begin{equation}\label{unbraccio}
\cR_\ell:=\set{\nu\in \cQ_{\cS}:	|\o(\nu,\VSc,I)\cdot \ell|\le 
\gamma \prod_{s\in \cS}\frac{1}{(1+|\ell_s|^2 \langle i(s)\rangle^2)^{3/2}}=:\gamma\gatta(\ell)
}.
\end{equation}
Recalling  the definition of $\o(\nu,\VSc,I)$ in \eqref{arancini} 
we note that by the continuity with respect to the product topology
of the functions  
	$\nu\mapsto\bO_j(\nu,\VSc,I)$
	with $ j\in \cS^c\,,\, I\in  \mathcal I(p,\rr)$
	(and since $|\ell|<\infty$)
	we have the crucial fact that
	 the sets  $\cR_\ell$ are closed (w.r.t. the product topology) and, therefore,  {\it measurable with respect to the  
	product probability measure on
	$\cQ_{\cS}$}. Moreover, since $\cQ_{\cS}$ is compact 
	all the $\cR_\ell$ are compact too.
	\\
	In the following we will often omit to write the immaterial dependence
	on $\VSc$ and $I$.
	\\
	Set $\tq(\ell):= \sum_{s\in \N} s^2 \ell_s.$
	If $|\tq(\ell)|\geq |\ell|$ 
	then 
	\[
	|\omega\cdot \ell|\ge | \sum_{s\in \N} s^2 \ell_s|  - \frac12|\ell|\ge\frac12 |\ell| \ge \frac12
	\]
	and 	 $\cR_{\ell}=\emptyset.$ 
	Recall the definition of the set $\cC$
	given in \eqref{arancini} and \eqref{diofantino nuc}.
	Denoting by $\meas$ the product probability measure on $\cQ_\cS$, we have
	\begin{equation}\label{labirinto}
	\cQ_\cS\setminus\cC = \bigcup_{\ell\in A} \mathcal R_\ell\,,\qquad 
	\mbox{which implies}
	\qquad 
	\meas(\cQ_\cS\setminus\cC) \le \sum_{\ell\in A} \meas(\mathcal R_\ell)\,,
	\end{equation}
	where
		\begin{equation}\label{dedalo}
	A:=\left\{ \ell\in\Z^\Z\ :\   0<|\ell|<\infty,\   \sum_{s\in \cS^c}  |\ell_s|\le 2,\ 
	  \fm(\ell)=\pi(\ell)= 0, \  |\tq(\ell)| <|\ell|
	\right\}\,.
	\end{equation}
	Fix $\ell\in A.$
	We note that 
	\begin{equation}\label{vacca}
\exists\, \bar s=\bar s(\ell)\in\cS\quad
\mbox{(depending only on $\ell$) such that }
\ell_{\bar s}\neq 0,
\end{equation}
	otherwise $0<|\ell|=\sum_{s\in \cS^c}|\ell_s|\le 2$, which contradicts Lemma
	\ref{carciofo}.
Since $\omega(\nu)= (\nu,\Omega(\nu))$, we get
	\[
	 |t|^{-1}| (\omega(\nu + t e_{\bar s})\cdot\ell)
	- (\omega(\nu)\cdot\ell)
	  |\ge |\ell_{\bar s}|-2\sup_{s\in \cS^c}|\Omega_s|^{\rm lip}\stackrel{\eqref{batacchio}}\ge 
	 1-2C\e\geq 1/2\,,
	\]
	taking $\e_*(p_*)$ small enough in \eqref{cornettonep}.
	Set $\hat\nu=(\nu_s)_{s\neq \bar s}$. Then for every $\hat\nu$
	there exist $a_\ell(\hat\nu)=a_\ell(\hat\nu,\VSc,I,\gamma)<
	b_\ell(\hat\nu)=b_\ell(\hat\nu,\VSc,I,\gamma)$
	satisfying
	\begin{equation}\label{cannoli}
	\big\{ \nu_{\bar s}\ \ \mbox{s.t.}\ \ (\nu_{\bar s},\hat\nu)\in \mathcal R_\ell\big\}
	\subseteq \big(a_\ell(\hat\nu),b_\ell(\hat\nu)\big)\,,\qquad
	  \mbox{with}\ \  
	b_\ell(\hat\nu)-a_\ell(\hat\nu)\leq 4\gamma \gatta(\ell)\,,
	\end{equation}
which implies
	$$
	\meas\{ \nu_{\bar s}\ \ \mbox{s.t}\ \ (\nu_{\bar s},\hat\nu)\in \mathcal R_\ell\}
	\leq 4\gamma \gatta(\ell)\,.
	$$ 
	Since $\mathcal R_\ell$ is measurable, by Fubini's Theorem
	we have that 
	$$
	\meas(\mathcal R_\ell) \le 
	4\gamma \gatta(\ell)
	=
	4\gamma \prod_{s\in \cS}\frac{1}{(1+|\ell_s|^2 \langle i(s)\rangle^2)^{3/2}}
	=
	4\gamma \prod_{i\in\N}\frac{1}{(1+|\ell_{s(i)}|^2 \langle i\rangle^2)^{3/2}}
	\,.
	$$
	Therefore 
	\begin{equation}\label{istanbul}
	\meas(\cQ_\cS\setminus\cC) \le
	4\gamma \sum_{\ell\in A}
	 \prod_{i\in\N}\frac{1}{(1+|\ell_{s(i)}|^2 \langle i\rangle^2)^{3/2}}
	\,.
\end{equation}
	We claim that
	\begin{equation}\label{coperta}
	\sum_{\ell\in A} \gatta(\ell)=
\sum_{\ell\in A}
	 \prod_{i\in\N}\frac{1}{(1+|\ell_{s(i)}|^2 \langle i\rangle^2)^{3/2}}
	 \leq \frac{C_0}{17}\,,
\end{equation}
taking $C_0$ large enough.
	
	Given $k\in\Z^\N$ with $|k|<\infty$ we define $\ell^k\in\Z^\Z$
	 supported on $\cS$ setting
	$$
	\ell^k_{s}:=k_{i(s)}\,,\ \ {\rm for}\ \ s\in\cS\,,\ \  \ell_s=0\ \ {\rm for}\ \ s\notin\cS\,.
	$$
	Now for each $\ell\in A$ there exist  unique $k\in\Z^\N$ with $|k|<\infty$
	and
	$s_1,s_2\in \cS^c$ and $\s_1,\s_2=\pm 1,0$ such that
	\begin{equation}\label{apapaia}
\ell = \ell^k+ \s_1 \be_{s_1} + \s_2 \be_{s_2}\,.
\end{equation}
	On the other hand, given $k\in\Z^\N$ with $|k|<\infty$, 
	there exist at most $36(|k|+2)$ vectors $\ell\in A$ satisfying \eqref{apapaia}.
	Indeed we prove that, given  $\s_1,\s_2=\pm 1,0$, there exist at 
	most\footnote{Note that there are 9 possible choices of $\s_1,\s_2=\pm 1,0$.}
	 $4(|k|+2)$ couples
	$(s_1,s_2)\in\cS^c\times\cS^c$ such that $\ell$ in \eqref{apapaia} 
	belongs to $A$.
Indeed they have to satisfy
	\[
\begin{cases}
\s_1s_1+\s_2s_2=-\pi(\ell^k) \\ 
\s_1s_1^2+\s_2s^2_2 = - \tq(\ell^k) + h\,, \qquad {\rm for \ some\ \ }|h|<|\ell| \leq |k|+2\,.
\end{cases}	
\]
Then by \eqref{istanbul} we get
$$
\sum_{\ell\in A}
	 \prod_{i\in\N}\frac{1}{(1+|\ell_{s(i)}|^2 \langle i\rangle^2)^{3/2}}
	\le 36\g	 \sum_{k\in \Z^\N:  0<|k|<\infty} 
	(|k|+2) \prod_{i\in \N}\frac{1}{(1+|k_i|^2 \langle i\rangle^2)^{3/2}} \,.
	$$
Since
\[
|k|+2 \le  2\prod_{i\in \N}(1+|k_i|^2 \langle i\rangle^2)^{1/2}
\]
we get
$$
\sum_{\ell\in A}
	 \prod_{i\in\N}\frac{1}{(1+|\ell_{s(i)}|^2 \langle i\rangle^2)^{3/2}}
 \le 72\g	 \sum_{k\in \Z^\N:  0<|k|<\infty} 
	 \prod_{i\in \N}\frac{1}{1+|k_i|^2 \langle i\rangle^2 }\,.
	$$
where the last sum converges (see \cite{Bourgain:2005} or Lemma 4.1 of \cite{BMP:2019}).
This concludes the proof of \eqref{coperta} taking $C_0$ large enough and, by \eqref{istanbul}, the proof of the first estimate in \eqref{mirto}.

\medskip

Let us now prove the second estimate in \eqref{mirto}.
Let us denote for brevity 
\begin{equation}\label{provolazzo}
{\rm meas}'(E):=
{\rm meas}_{[-\nicefrac14,\nicefrac14]^\cS}(E\cap [-\nicefrac14,\nicefrac14]^\cS)
\end{equation}
the  product probability measure on $[-\nicefrac14,\nicefrac14]^\cS$.
Since $\VSf$ (defined in \eqref{labello}) is 
   Lipschitz  $C\e\g$-close to the map
  $\VSs$
 (defined in \eqref{piramide})
 we have that
$$
 \VSf(\cQ_\cS,\VSc,I')\supset [-\nicefrac14,\nicefrac14]^\cS
$$
for every
$\VSc\in [-\nicefrac14,\nicefrac14]^{\cS^c}$
and\footnote{In the present proof we actually take $I'=I$.
We have introduced $I'$ for estimate \eqref{nastrorosa}
that will be used in the proof of Lemma \ref{culonembo}.} 
$I'\in \mathcal I(p,\rr)$.
 The function 
 $\VSf(\nu,\VSc,I)$ is invertible (w.r.t. $\nu$).
  In particular there exists
 a function 
 \begin{equation}\label{tanfo}
 \begin{aligned}
 g\, :\,  [-\nicefrac38,\nicefrac38]^{\cS}\times [-\nicefrac14,\nicefrac14]^{\cS^c}\times \mathcal I(p,\rr) & \longrightarrow 
 [-4 \bar C \g \epsilon,4 \bar C \g \epsilon]^\cS \\
  (\VS,\VSc,I)  &\mapsto\  g(\VS,\VSc,I)
	\end{aligned}
\end{equation}
 which is continuous w.r.t. the product topology
 in $ [-\nicefrac38,\nicefrac38]^{\cS}\times [-\nicefrac14,\nicefrac14]^{\cS^c}$ and 
 satisfies
\begin{equation}\label{appesta}
\VSf\big(q+\VS+g(\VS,\VSc,I),\VSc,I\big)=\VS\,,
\qquad \mbox{where}\qquad
q=(q_j)_{j\in\cS}\,,\ \ \ q_j:=j^2\,,
\end{equation}
 with Lipschitz estimates
  \begin{equation}\label{fetore}
|g(V_\cS',\VSc,I)-g(\VS,\VSc,I)|_{\ell^\infty_{\cS}}
\leq 
6 \bar C  \epsilon
|V_\cS'-\VS|_{\ell^\infty_{\cS}}\,.
\end{equation} 
 Indeed recalling \eqref{visciola}
  $g$  satisfies the fixed point equation
  (e.g. by a slightly modified version of Lemma 
\ref{zagana})
 $$
 g=-\lambda\big(q+\VS+g, \bO(q+\VS+g,\VSc,I),I\big)\,;
 $$
then \eqref{tanfo} and \eqref{fetore} follow from \eqref{loredana}) (recalling \eqref{grevity}
 and \eqref{solenoide}).
  \\ 
  For the remaining of the proof we drop the dependence on $\VSc$ and $\gamma.$
  Fix $\ell\in A, $
first we note that  for all $I,I'\in \cI(p,\rr)$ the set $\VSf(\mathcal R_\ell(I);I')$
is measurable since 
$\mathcal R_\ell(I)$ is closed and $\VSf$ is continuously invertible.
 Recalling \eqref{unbraccio} we have
\begin{align*}
	& \VSf\big(\mathcal R_\ell(I),I'\big) 
	\cap [-\nicefrac14,\nicefrac14]^\cS
	=\big\{\VS\in[-\nicefrac14,\nicefrac14]^\cS: \,
	\exists\, \nu \in \mathcal R_\ell(I),   
	\VS=\VSf(\nu,I')
	\big\}
	\nonumber
\\
&	\stackrel{\eqref{appesta}}=
	\big\{\VS\in[-\nicefrac14,\nicefrac14]^\cS\ \mbox{s.t.}\ 	
	|\o\big(q+\VS+g(\VS,I'),I\big)\cdot \ell|\le 
\gamma\gatta(\ell)
\big\}=:\mathtt E
	\,.
	\label{foggydew}
	\end{align*}
As usual we set $\ell=:(\ell_\cS,\ell_{\cS^c})$.
Let $\bar \cS:=\{j\in\cS \ :\ \ell_j\neq 0\}.$
We split $\ell_\cS=:(\ell_{\bar \cS},0)$.
\\
Decomposing $\VS=(V_{\bar \cS},V_{\cS\setminus\bar\cS})$
we set
$$
\mathtt E(V_{\cS\setminus\bar\cS}):=\{
V_{\bar \cS} \in 
[-\nicefrac14,\nicefrac14]^{\bar\cS}\ \ :\ \ 
(V_{\bar \cS},V_{\cS\setminus\bar\cS})\in\mathtt E\}\,.
$$
We claim that for every $V_{\cS\setminus\bar\cS}\in
[-\nicefrac14,\nicefrac14]^{\cS\setminus\bar\cS}
$
\begin{equation}\label{eire}
\meas_{[-\nicefrac14,\nicefrac14]^{\bar\cS}}
\Big(\mathtt E(V_{\cS\setminus\bar\cS})\Big)
\leq
16\g \gatta(\ell)
\,.
\end{equation}
Then by Fubini's theorem we get
\begin{equation}\label{nastrorosa}
	   {\rm meas}_{[-\nicefrac14,\nicefrac14]^\cS}(\mathtt E)
	   =
   {\rm meas}' \big(
   \VSf\big(\mathcal R_\ell(I),I'\big)
   \big)\leq
16\g \gatta(\ell)\,.
   \end{equation}
Then
$$
{\rm meas}_{[-\nicefrac14,\nicefrac14]^\cS} \Big(\VSf\big(\cQ_\cS \setminus \cC(I),I'\big)\Big)
\leq
	\sum_{\ell\in A} 16\g \gatta(\ell)
$$
and the second estimate in \eqref{mirto} follows by
\eqref{coperta} and taking $I'=I$.

Let us finally prove the claim in \eqref{eire}.
Fix $V_{\cS\setminus\bar\cS}\in
[-\nicefrac14,\nicefrac14]^{\cS\setminus\bar\cS}
$.
Set for brevity\footnote{Recall \eqref{arancini} and \eqref{appesta}.}
\begin{eqnarray*}
f(V_{\bar \cS})
&:=&
\o\big(q+\VS+g(\VS,I'),I\big)\cdot \ell
\\
&\stackrel{\eqref{arancini}}=&
V_{\bar\cS}\cdot \ell_{\bar \cS}+
(q+g(\VS,I'))\cdot \ell_{\cS}+
\Omega\big(q+\VS+g(\VS,I'),I\big)\cdot \ell_{\cS^c}\,.
\end{eqnarray*}	
The aim now is to prove that the increment of $f$
{\sl in a suitable direction} is bounded from below.
Since $\ell\in A$ by \eqref{vacca} we get $|\ell_\cS|\geq 1$ and
$|\ell_{\cS^c}|\leq 2$.
Define  $\sigma_{\bar \cS}$ by
$\sigma_{\bar\cS,j}={\rm sign}\, \ell_j$ for $j\in\bar\cS;$
then  $\sigma_{\bar\cS}\cdot \ell_{\bar\cS}=
|\ell_{\bar\cS}|=
|\ell_\cS|\geq 1$ and $|\sigma_{\bar\cS}|_\infty=1$.
Then we get
$$
|t|^{-1}|f(V_{\bar \cS}+t \sigma_{\bar\cS})-f(V_{\bar \cS})|
\geq |\ell_\cS|-16 \bar C  \epsilon |\ell_\cS|
\geq |\ell_\cS|/2\,,
$$
by \eqref{fetore} and \eqref{martini} and  taking $\e_*(p_*)$ in  \eqref{cornettonep}
small enough such that
	$
	16\bar C\epsilon
	\leq 16 \bar C c(p_*)\e\leq 1/2
	$
	(recall \eqref{sade},\eqref{adu}).
Then we get that
$$
\meas\{ t\in\R \ :\  V_{\bar \cS}+t \sigma_{\bar\cS}\in \mathtt E(V_{\cS\setminus\bar\cS})\}\leq
4\g \gatta(\ell)/|\ell_\cS|\,.
$$
We need the following  
(finite dimensional) result proved in the Appendix.

\begin{lemma}\label{zufolo}
 Let $E\subset [-\nicefrac14,\nicefrac14]^n$ be a measurable set.
 Let $\xi=(\hat\xi,\pm 1)$ with $\hat\xi\in\R^{n-1}$.
 Assume that for every $x\in [-\nicefrac14,\nicefrac14]^n$ 
 we have\footnote{$\meas_\R$, resp.
$\meas_{\R^n}$,  being the standard   measure in $\R$,
resp $\R^n$.}
 $\meas_\R\{ t\in\R \ :\  x+t \xi\in E\}\leq \delta.$
 Then
 $\meas_{\R^n}(E) \leq 2^{1-n}\delta |\xi|_2^2, $ where $|\cdot|_2$
 denotes the Euclidean norm.
\end{lemma}
\noindent
Then \eqref{eire} follows by 
 the above lemma  with\footnote{Here we consider
the case when $\bar\ell_\cS=-1$ the $+1$ case is analogous.}
$$
E= \mathtt E(V_{\cS\setminus\bar\cS})\,,\ \ \ n=\sharp \bar\cS\,,\ \ \  
\xi=\sigma_{\bar\cS}\,,\ \ \ 
x=V_{\bar \cS}\,,\ \ \ 
\delta=4\g \gatta(\ell)
/|\ell_\cS|
$$
and noting that $|\sigma_\cS|_2^2=n\leq |\ell_\cS|$
and $\meas_{[-\nicefrac14,\nicefrac14]^{n}}(A)=2^n \meas_{\R^n}(A)$
for every $A\subseteq [-\nicefrac14,\nicefrac14]^{n}.$
\end{proof}

\begin{proof}[\bf{Proof of Corollary \ref{ga}}]
It is equivalent to prove that 
\begin{equation}\label{ovetto2}
		\cB(I,\g):= \{ V\equiv(\VS, \VSc)\in [-\nicefrac14,\nicefrac14]^\Z\;: \,\VS\in \VSf(\pan_\cS\setminus\cC(\VSc,I,\g),\VSc,I) \}\,
	\end{equation}
  is a Borel set in $[-\nicefrac14,\nicefrac14]^\Z$ with measure bounded by $C_0\g$ ($C_0$ is the constant in \eqref{mirto}).


\noindent
The function $h:[-\nicefrac14,\nicefrac14]^\Z\to \cQ_{\cS}\times
[-\nicefrac14,\nicefrac14]^{\cS^c}$ defined by\footnote{$g$ was defined in \eqref{tanfo}.
By \eqref{appesta} $h$ is the inverse of the function 
$(\nu,\VSc)\mapsto \big(\VSf(\nu,\VSc),\VSc\big)$.}
$$
h(V):=\big(q+\VS+g(\VS,\VSc,I),\VSc\big)
$$
is continuous with respect to the product topology.
Then, for every $\ell\in A$ (recall \eqref{dedalo}), the set
\begin{align*}
\cB_\ell(I,\g) 
& :=  \{ V\equiv(\VS, \VSc)\in [-\nicefrac14,\nicefrac14]^\Z\;: \, 
|\o(h(V))\cdot\ell|\leq \g \gatta(\ell)
 \}\\
 \stackrel{\eqref{unbraccio},\eqref{appesta}} = & 
 \{ V\equiv(\VS, \VSc)\in [-\nicefrac14,\nicefrac14]^\Z\;: \quad 
\VS\in  \VSf\big(\mathcal R_\ell(\VSc,I),\VSc,I\big)
 \}
\end{align*}
is closed and, therefore, measurable.
Then by Fubini's theorem and \eqref{nastrorosa} (with $I=I'$) we get
$
	   {\rm meas}_{[-\nicefrac14,\nicefrac14]^\Z}(\cB_\ell)
	   $ $
	   \leq
16\g \gatta(\ell).
  $
  Since
 $\cB$ defined in \eqref{ovetto2}  can be written as
$\cB=\cup_{\ell\in A}\cB_\ell$ (recall \eqref{labirinto}), by \eqref{coperta} we get
${\rm meas}_{[-\nicefrac14,\nicefrac14]^\Z}(\cB)$
$	   \leq
C_0\g.$

\end{proof}

\begin{proof}[\bf{Proof of Lemma \ref{culonembo}}]
The first estimate in \eqref{giuncata} is equivalent to
\begin{equation}\label{concallato}
{\rm meas}_{[-\nicefrac14,\nicefrac14]^\cS} \Big(\VSf\big(\cQ_\cS \setminus \cC'(\VSc,I,\g),\VSc,I\big)\Big) 
\le  C_0\g\,,
\end{equation}
where $C_0$ was defined in 
	Theorem \ref{toro Sobolevp}.

In the following we drop the dependence on $\VSc$. We first note that recalling 
 \eqref{unbraccio}, \eqref{dedalo} and setting $L:=C  \rr^{-2}\epsilon$, we get
\begin{equation}\label{caponata}
|I-I'|_{2p_*}\leq \gatta(\ell)/4L
\qquad\Longrightarrow \qquad
\cR_\ell(I',\gamma/2)
\subset
\cR_\ell(I,\gamma)\,,\qquad 
\forall\, \ell\in A\,.
\end{equation}
Indeed, since $\ell\in A$,
$$
|\omega(\nu,I)\cdot\ell|
\stackrel{\eqref{batacchio}}\leq 
|\omega(\nu,I')\cdot\ell|+2 \gamma L |I-I'|_{2p_*}
\leq
\gamma \gatta(\ell)/2+2 \gamma L |I-I'|_{2p_*}
\leq \gamma \gatta(\ell)\,.
$$
Therefore
\begin{eqnarray*}
\cQ_\cS \setminus \Big(
\cC(I,\gamma)\cap\cC(I',\gamma/2)
\Big)
&=&
\Big( \cQ_\cS \setminus 
\cC(I,\gamma)\Big)
\bigcup
\Big( \cQ_\cS \setminus 
\cC(I',\gamma/2)
\Big)
\\
&\stackrel{\eqref{labirinto},\eqref{caponata}}=&
\bigcup_{\ell \in A} \cR_\ell(I,\gamma) \!\!\!\!
\bigcup_{\substack{\ell \in A\,, \\ \gatta(\ell)<4L|I-I'|_{2p_*}}}\!\!\!\! \cR_\ell(I',\gamma/2)\,.
\end{eqnarray*}
Then
$$
\cQ_\cS \setminus\cC'(I,\gamma)
:=
\bigcup_{\ell \in A} \cR_\ell(I,\gamma)
\bigcup_{k\in\N} 
\bigcup_{\substack{\ell \in A\,, \\ \gatta(\ell)<4L|I-I^{(n_k)}|_{2p_*}}} \cR_\ell(I^{(n_k)},\gamma/2)
$$
and (recalling the notation in \eqref{provolazzo}) 
\begin{align*}
& {\rm meas}' \Big(\VSf\big(\cQ_\cS \setminus \cC'(I,\g);I\big)\Big)
\leq 
\sum_{\ell \in A}
{\rm meas}' \big(\VSf(\mathcal R_\ell(I,\gamma);I)\big) \\
& +
\sum_{k\in\N} 
\sum_{\substack{\ell \in A\,, \\ \gatta(\ell)<4L|I-I^{(n_k)}|_{2p_*}}}
{\rm meas}' \big(\VSf(\mathcal R_\ell(I^{(n_k)},\gamma);I)\big)
\\
\stackrel{\eqref{nastrorosa}} \leq & 
16\gamma 
\sum_{\ell \in A}
\gatta(\ell)
+
16\gamma 
\sum_{k\in\N} 
\sum_{\substack{\ell \in A\,, \\ \gatta(\ell)<4L|I-I^{(n_k)}|_{2p_*}}}
\gatta(\ell)
\leq
17\gamma 
\sum_{\ell \in A}
\gatta(\ell)
\stackrel{\eqref{coperta}}\leq
C\gamma\,,
\end{align*}
taking the sub-sequence $(n_k)_{k\in\N}$ growing fast enough\footnote{Since the positive term series in  \eqref{coperta} converges. }. This proves \eqref{concallato}.
The second estimate in \eqref{giuncata} follows from the first one and Fubini Theorem provided that $\cG(I,\g)$ is a Borel set.
This holds true since
\begin{align*}
		\mathcal G(I,\gamma)
&= \cap_k\{ V=(\VS,\VSc)\in[\nicefrac14,\nicefrac14]^\Z:\; \VS\in \VSf(\cC(I^{(n_k)},\g/2),\VSc,I^{(n_k)})\}
\\ &\cap \{ V=(\VS,\VSc)\in[\nicefrac14,\nicefrac14]^\Z\ \mbox{s.t.} \VS\in \VSf(\cC(I,\g),\VSc,I)\}
\\
&= [\nicefrac14,\nicefrac14]^\Z \setminus \pa{\cup_k  \cB(I^{(n_k)},\g/2)\cup \cB(I,\g)}\,.
\end{align*}
The Borel sets $\cB(I^{(n_k)},\g/2), \cB(I,\g)$ are defined in \eqref{ovetto2}.
\end{proof}

\appendix
\section{Technicalities}\label{appendice tecnica}

\begin{proof}[{\bf Proof of Lemma \ref{godiva}}]
Let us consider first the case $f\in {\rm F}(\cO\times B_\rho)$, 
namely when $f$ depends only on a finite number of $\o_j$'s. 
By definition this means that, for some  $k\in\N$
there exists a function 
$\hat f\,:\,P_k(\cO)\times B_\rho\,\to\, \C$,
 where $P_k$ is the projection 
 $P_k:\R^\Z\to\R^{2k+1}$ defined as $P_k(\o):=(\o_{-k},\ldots,\o_k),$
 such that  
 $f(\o,I)=\hat f(\o_{-k},\ldots,\o_k,I)$ for every $(\o,I)\in\cO\times B_\rho.$
By Cauchy estimates
$$
| \hat f(\o,I)-\hat f(\o,I')|\leq 
2 \rho^{-1}|f|^{\g,\cO\times B_\rho}|I-I'|_E\,,
\qquad
\forall\, \o\in\cO\,,\ \   I,I'\in B_{\rho/2}\,.
$$
We need the following 

\begin{lemma}[Lipschitz extension]\label{pinzellone}
 Let $X$ be a  metric space endowed with the metric $d(\cdot,\cdot)$ and $\emptyset\neq U\subseteq X$.
 \\
i) Let $f:U\to\R$  be a $L$-Lipschitz function,
namely $|f(u)-f(v)|\leq L d(u,v)$ for every $u,v\in U.$
Then 
$\bar f(x):=\inf_{u\in U} f(u)+L\, d(x,u)$, $x\in X,$
is a $L$-Lipschitz extension of $f$.
\\
ii) Moreover if $\sup_U|f|=:M<\infty$, 
then $\tilde f:=\max\{-M,\min\{\bar f,M\}\}$
is a $L$-Lipschitz extension of $f$ satisfying
$\sup_X|\tilde f|=M$.
\\
iii) Let $g:U\to\ell^\infty(\R)$ a $L$-Lipschitz function,
namely $|g(u)-g(v)|_{\ell^\infty}\leq L d(u,v)$ for every $u,v\in U.$
Then there exists a $L$-Lipschitz extension $\bar g:X\to \ell^\infty(\R)$ of $g$.
An analogous statement  holds for every $\ell^\infty$-weighted space, in particular for $\tw_p$.
\end{lemma}
\noindent
Before proving  Lemma  \ref{pinzellone}
we conclude the proof of Lemma \ref{godiva}.
Then Lemma  \ref{pinzellone} with
$X=P_k(\pan)\times B_{\rho/2}$,  distance
$d\big((\o,I),(\o',I')\big)=|\o-\o'|_\infty+2\g\rho^{-1}|I-I'|_E,$
 $U=P_k(\cO)\times B_{\rho/2}$,
 $f=\hat f$,
  $L=\gamma^{-1}|f|^{\g,\cO\times B_\rho}$
 implies Lemma \ref{godiva} in the case $f\in {\rm F}(\cO\times B_\rho)$. 
 Note in particular that, since $P_k(\pan)$ is finite dimensional
 the product topology on it coincides with one induced by the norm $|\cdot|_\infty$.
 Then  the extended function $\tilde f\in {\rm F}(	\pan\times B_{\rho/2})$
and its continuity in $\o$ follows by its Lipschitz-continuity in $\o$.
Moreover, in this case, 
$|\tilde f|^{\g,\cQ\times B_{\rho/2}}= |f|^{\g,\cO\times B_\rho}$
and
$$
|\tilde f(\o,I)-\tilde f(\o,I')|\leq 
2 \rho^{-1}|f|^{\g,\cO\times B_\rho}|I-I'|_E\,,
\qquad
\forall\, \o\in\pan\,,\ \   I,I'\in B_{\rho/2}\,.
$$
\\
Consider now the case of a general $f\in\cF(\cO\times B_\rho)$.
By definition there exists a sequence $f_n\in {\rm F}(\cO\times B_\rho)$ such that
$|f_n-f|^{\g,\cO\times B_\rho}\leq 4^{-n-1}|f|^{\g,\cO\times B_\rho}$ and
$g_0:=f_0$, $g_n:=f_n-f_{n-1}$ for $n\geq 1$.
Then
$f=\sum_{n\geq 0} g_n$. 
Moreover 
$|g_0|^{\g,\cO\times B_\rho}\leq 5/4|f|^{\g,\cO\times B_\rho},$
$|g_n|^{\g,\cO\times B_\rho}=|f_n-f|^{\g,\cO\times B_\rho}+|f-f_{n-1}|^{\g,\cO\times B_\rho}\leq 5\cdot 
4^{-n-1}|f|^{\g,\cO\times B_\rho}$
and $g_n\in {\rm F}(\cO\times B_\rho)$.
Then there exist extensions $\tilde g_n \in {\rm F}(\pan\times B_{\rho/2})$, $n\geq 0$ such that
$\tilde g_n=g_n$ on $\cO\times B_{\rho/2}$,
$|\tilde g_n|^{\g,\cQ\times B_{\rho/2}}\leq 5\cdot 
4^{-n-1}|f|^{\g,\cO\times B_\rho},$
$\tilde g_n$ is continuous w.r.t  the product topology in $\cO$, finally
$\tilde g_n$ is Lipschitz on  $B_{\rho/2}$ with estimate
$$
|\tilde g_n(\o,I)-\tilde g_n(\o,I')|\leq 
 10\rho^{-1}4^{-n-1}|f|^{\g,\cO\times B_\rho}|I-I'|_E\,,
\quad
\forall\, \o\in\pan\,,\ \   I,I'\in B_{\rho/2}\,.
$$
Finally set $\tilde f:=\sum_{n\geq 0}\tilde g_n.$
\end{proof}

\begin{proof}[{\bf Proof of Lemma \ref{pinzellone}}]

\noindent
i) It is McShane's Theorem. Incidentally we note that $\bar f$ is the greatest possible
extension while the smaller one is
$\underline f(x):=\sup_{u\in U} f(u)-L\, d(x,u)$.
\\
ii) The fact  that $\tilde f$ is $L$-Lipschitz follows applying twice the following result:
\\
{\it
 Given a $L$-Lipschitz $g:X\to\R$ and $c	\in\R$
 the functions $\bar g:=\max\{ g,c\}$ and $\underline g:=\min\{ g,c\}$
 are $L$-Lipschitz.
}
We prove it only for $\bar g$, the other case being analogous.
We first note that
$\bar g= (g+c+|g-c|)/2$. Then
$$
|\bar g(x)-\bar g(y)|\leq \frac12 \big(|g(x)-g(y)|+\big| |g(x)-c|-|g(y)-c|\big|\big)
\leq |g(x)-g(y)|\leq L d(x,y)\,.
$$
\\
iii)
By the McShane's Theorem in point i), every component $g_j:A\to \R$ 
of
$g$ can be extended to $\tilde g_j: X\to \R$ with the same 
Lipschitz constant $L$. Then, setting $\tilde g:=(\tilde g_j)_{j\in\Z}$,
we have that $\tilde g\in\ell^\infty$;
in fact, for every fixed $x\in X$, $x_0\in A$ and $j\in\Z$,
$$
|\tilde g_j(x)|\leq |\tilde g_j(x)-\tilde g_j(x_0)|+|g_j(x_0)|
\leq L d(x,x_0)+\|g(x_0)\|_{\ell^\infty}\,.
$$
Finally for every
$x_1,x_2\in X$ and $j\in\Z$, we have
that $|g_j(x_1)-g_j(x_2)|\leq L d(x,x_0)$ and, finally,
$|g(x_1)-g(x_2)|_{\ell^\infty} \leq L d(x,x_0)$.
The extension to $\ell^\infty(\C)$ and to
 every $\ell^\infty$-weighted space, in particular for $\tw_p$, is straightforward. 
\end{proof}

\noindent
{\bf Proof of Proposition \ref{fan}}. \\
\noindent
	Writing $F = \sum \bcoeffu{F}\uuu$ and $G = \sum\bcoeffd{G}\uud$ we have
	\begin{eqnarray*}
	\set{F,G} &=& \im\sum_{\baluno,\bbtuno,\baldue,\bbtdue} \bcoeffu{F}\bcoeffd{G} \sum_j \pa{\baluno_j \bbtdue_j  - \bbtuno_j \baldue_j }u^{\baluno + \baldue -e_j}\bar{u}^{\bbtuno + \bbtdue - e_j}\\
	&=:&H=\sum_{\bal,\bbt}H_{\bal,\bbt} u^\bal \bar{u}^\bbt\,,
	\end{eqnarray*}
where 
\begin{equation}\label{catta}
H_{\bal,\bbt}:= \im\sum_j \sum_{\substack{\baluno+\baldue -e_j=\bal \\\bbtuno+\bbtdue -e_j=\bbt}} \bcoeffu{F}\bcoeffd{G} \pa{\baluno_j \bbtdue_j  - \bbtuno_j \baldue_j }.
\end{equation}
Note that $\set{F,G}(0) = 0$,  indeed $H_{\bf{0},\bf{0}} = 0$ since $\bal'+ \bal'' = \bbt' + \bbt'' = e_j$ (namely $ \bal'_j + \bal''_j = \bbt'_j + \bbt''_j = 1$ and $\bal'_k + \bal''_k = \bbt'_k + \bbt''_k = 0 $ for $k\neq j$) implies that $\baluno_j \bbtdue_j  - \bbtuno_j \baldue_j = 0$ by mass conservation $|\bal| = |\bbt|$. \\  
Recalling \eqref{norma1} we have	
	\begin{equation}\label{frisa}
	\begin{aligned}
& \abs{\set{F,G}}_{r,p} \leq \\
	 & \sup_\ell  
	\sum_j \sum_{\substack{\baluno,\bbtuno,\\\baldue,\bbtdue}} \!\!\!\!
	\abs{\bcoeffu{F}}\abs{\bcoeffd{G}}
	(\baluno_j \bbtdue_j+\baldue_j \bbtuno_j)
	(\bbtuno_\ell + \bbtdue_\ell )
 u_p^{ \baluno + \bbtuno - 2 e_{\ell} + \baldue + \bbtdue - 2e_j} 
	  	 \,,
	  	 \end{aligned}
	\end{equation}
	where $u_p=u_p(r)$ was defined in \eqref{giancarlo}.
We split in four terms the right hand side of \eqref{frisa}
according to the splitting
$$
(\baluno_j \bbtdue_j+\baldue_j \bbtuno_j)
	( \bbtuno_\ell + \bbtdue_\ell )
	=\baluno_j \bbtdue_j \bbtuno_\ell 
	+
	\baluno_j \bbtdue_j \bbtdue_\ell +
	\baldue_j \bbtuno_j \bbtuno_\ell 
	+\baldue_j \bbtuno_j \bbtdue_\ell \,;
$$
	we will consider only the term
		\begin{eqnarray}\label{frisa2}
	&&
	\nonumber 
	A:=  \sup_\ell  
	\sum_j \sum_{\baluno,\bbtuno,\baldue,\bbtdue} 
	\abs{\bcoeffu{F}\bcoeffd{G}}
	\baluno_j \bbtdue_j
	 \bbtuno_\ell 
\times	  u_p^{ \baluno + \bbtuno - 2 e_{\ell} + \baldue + \bbtdue - 2e_j} 
	  	 \,,
	\end{eqnarray}
	the others being analogous.
	Noting that 
	 $\forall j\in\Z$, 
	\[
	\sum_{\baldue,\bbtdue}\abs{\bcoeffd{G}}\bbtdue_ju_p^{ \baldue + \bbtdue - 2e_j}\leq \abs{G}_{r,p},
	\] 
	we have
	\begin{align*}
	A &\leq \abs{G}_{r,p} 
	\sup_\ell   
	\sum_{\baluno,\bbtuno}
	\abs{\bcoeffu{F}}\sum_j{\baluno_j }\bbtuno_\ell
	 u_p^{ \baluno + \bbtuno - 2e_\ell}
	 \\
	&=
	\abs{G}_{r,p} 
	\sup_\ell  
	\sum_{\baluno,\bbtuno} 
	\abs{\bcoeffu{F}}\abs{\baluno }
	\bbtuno_\ell
	u_p^{ \baluno + \bbtuno - 2e_\ell}
	\\
	&\leq \abs{G}_{r,p} \sup_\ell  
	\sum_{\baluno,\bbtuno} 
	\abs{\bcoeffu{F}}
	\bbtuno_\ell
	\tilde u_p^{ \baluno + \bbtuno - 2e_\ell}
	\abs{\baluno + \bbtuno}
	\pa{\frac{r}{r+\rho}}^{\abs{\baluno + \bbtuno} - 2}	\,,
	\end{align*}
	where $\tilde u_p$ is short for $u_p(r + \rho)$ (recall \eqref{giancarlo}).
	Since\footnote{Indeed, setting $y:=\rho/r,$ we have that
	$\sup_{x\ge 2} x\pa{\frac{r}{r+\rho}}^{x - 2}=\sup_{x\ge 2} x(1+y)^{2-x}=2$
	if $y\geq \sqrt e-1$. On the other hand, when $0< y <\sqrt e-1$ we have 
 		$$
	 \sup_{x\ge 2} x(1+y)^{2-x}=
	 \frac{(1+y)^2}{e\ln (1+y)}
	 \leq\frac{1}{y} \sup_{0< y <\sqrt e-1}	 \frac{(1+y)^2 y}{e\ln (1+y)}
	 =
	 \frac{2}{y}\,.
	$$
	} 
	$$
	\sup_{|\baluno|=|\bbtuno|}
	\abs{\baluno + \bbtuno}
	\pa{\frac{r}{r+\rho}}^{\abs{\baluno + \bbtuno} - 2}
	\leq
	\sup_{x\ge 2} x\pa{\frac{r}{r+\rho}}^{x - 2}
	\leq 
	2 \max\set{1, \frac{r}{\rho} }
	$$
 we get 
	\begin{align*}
	A &\leq 2 \max\set{1,\frac{r}{\rho}} \abs{G}_{r,p} \sup_\ell   \sum_{\baluno,\bbtuno}\abs{\bcoeffu{F}}
	\bbtuno_\ell\tilde u_p^{\baluno + \bbtuno - 2e_\ell},
	\end{align*}
	and, recalling the definition in \eqref{norma1}, this concludes the proof of \eqref{menate}.
\qed

\begin{rmk}\label{ariccia}
Note that Proposition \ref{fan} and its proof hold for any $\ell^\infty$-weighted norm, namely with norm $\sup_{j\in\Z} w_j |u_j|$,
with $w_j\to\infty.$ 
\end{rmk}

\noindent
\begin{proof}[{\bf Proof of Lemma \ref{mente}}]Let us first prove the Banach structure. Consider a Cauchy sequence of Hamiltonians $H^{(n)}\in \cH^{\cO,0}_{r,p}$. For all $\omega\in \cO$, $I\in \cI$, $H^{(n)}(\omega,I)$ is a Cauchy sequence in $\scH_{r,p}^{0}$ then we define 
\[
H(\omega,I)= \sum_{\bal,\bbt\in \cM} H_{\bal,\bbt}(\omega,I) u^\bal\bar u^\bbt  \in \cH^{0}_{r,p}
\]
for all $\omega\in \cO$ $I\in \cI$, one has point-wise convergence $H^{(n)}(\omega,I)\to H(\omega,I)\in \scH_{r,p}$.
Moreover,  for all $\bal,\bbt$ the sequence $H^{(n)}_{\bal,\bbt}$ is a Cauchy sequence w.r.t the norm $|\cdot|^\g$. 
Since $\cF(\cO\times \cI)$ is  Banach, for all $\bal,\bbt\in \cM$
$H^{(n)}_{\bal,\bbt}\to H_{\bal,\bbt}\in \cF(\cO\times \cI)$. 
\\
By hypothesis $\forall\e>0$  there exist $N$ such that for all $n,m>N$ one has
\[
\frac12  \sup_j  
\sum_{(\bal,\bbt)\in \cM}|H^{(n)}_{\bal,\bbt}- H^{(m)}_{\bal,\bbt}|^\g \pa{ \bal_j + \bbt_j}u_p^{\bal + \bbt - 2e_j} =
\| H^{(n)}- H^{(m)}\|_{r,p}< \e
\]
so taking the liminf on $m$  we get  for all $j$
\[
\frac12 \liminf_m 
\sum_{(\bal,\bbt)\in \cM}|H^{(n)}_{\bal,\bbt}- H^{(m)}_{\bal,\bbt}|^\g \pa{ \bal_j + \bbt_j}u_p^{\bal + \bbt - 2e_j} < \e.
\]
Then, for all $\bal,\bbt$ one has
\begin{equation}
	\label{accidia}
|H^{(n)}_{\bal,\bbt}- H_{\bal,\bbt}|^\g  \le \liminf_m |H^{(n)}_{\bal,\bbt}- H^{(m)}_{\bal,\bbt}|^\g 
\end{equation}
so
\[
\begin{aligned}
& \sum_{(\bal,\bbt)\in \cM}|H^{(n)}_{\bal,\bbt}- H_{\bal,\bbt}|^\g \pa{ \bal_j + \bbt_j}u_p^{\bal + \bbt - 2e_j} \\ 
&\le \sum_{(\bal,\bbt)\in \cM} \liminf_m |H^{(n)}_{\bal,\bbt}- H^{(m)}_{\bal,\bbt}|^\g  \pa{ \bal_j + \bbt_j}u_p^{\bal + \bbt - 2e_j}\\
& \le \liminf_m \sum_{(\bal,\bbt)\in \cM}  |H^{(n)}_{\bal,\bbt}- H^{(m)}_{\bal,\bbt}|^\g  \pa{ \bal_j + \bbt_j}u_p^{\bal + \bbt - 2e_j} < \e\,,
\end{aligned}
\]
by Fatou's lemma. Taking the supremum over $j$ we have proved that $H^{(n)}\to H$ in the $\cH^{\cO,0}_{r,p}$ norm.\\
Concerning the Poisson algebra property, it suffices to use the fact that $\norm{\cdot}^\g$ has the algebra property with respect to standard multiplication and from \eqref{catta} we deduce 
\begin{equation}
\norm{H_{\bal,\bbt}}^\g \le\sum_j \sum_{\substack{\baluno+\baldue -e_j=\bal \\\bbtuno+\bbtdue -e_j=\bbt}} \norm{\bcoeffu{F}}^\g\norm{\bcoeffd{G}}^\g \pa{\baluno_j \bbtdue_j  + \bbtuno_j \baldue_j }.
\end{equation}
Then the proof follows verbatim the one of Proposition \ref{fan}.
\end{proof}
\bigskip

\begin{proof}[{\bf Proof of Proposition \ref{proiettotutto}}]
Items (i) and (ii) directly follow by Propositions 4.1 and 4.2 of \cite{BMP:almost}, respectively.
Here we only discuss 
the analyticity with respect to $I$.
By formula (4.30) in 
and recalling the notations in \eqref{vitello}
we get the representation formula for every $ |u|_{p}\leq r'$
\begin{equation}\label{cippalippa3}
\begin{aligned}
& H^{(d)}(u,\o,I)=
 \sum_{\substack{\al,\bt,\zeta, a,b \\ 2|\zeta|+|a|+|b|\leq d+2}}\ 
\sum_{\substack{\delta\succeq \zeta \\ 2|\delta|+|a|+|b|=d+2}}
\sum_{m\succeq \delta} A\\
&A=\binom{m}{\delta}
 \binom{\delta}{\zeta}
(-1)^{|\delta-\zeta|}
I^{m-\zeta}
 H_{m,\al,\bt,a,b}(\o,I)
 |v|^{2\zeta}
 v^\al \bar v^\bt z^a \bar z^b
\,,
\end{aligned}
\end{equation}
where $m,\al,\bt,\zeta,\delta\in \N^\cS, a,b\in\N^{\cS^c}, \al\cap \bt= \emptyset$.

Set for brevity $X:=X_{\Pi^{d} H}$
the Hamiltonian vector field of $\Pi^{d} H$
(recall \eqref{grado falso}).
Then
define $Y$ component-wise for $j\in\Z$ as
\begin{equation*}
Y_j
:=
\frac12 \!\!\!\!\!\!\!
\sum_{\substack{\al,\bt,\zeta, a,b \\ 2|\zeta|+|a|+|b|\leq d+2}}\ 
\sum_{\substack{\delta\succeq \zeta \\ 2|\delta|+|a|+|b|=d+2}}\ 
\sum_{m\succeq \delta}
\binom{m}{\delta}
 \binom{\delta}{\zeta}
I_p^{m-\zeta}
| H_{m,\al,\bt,a,b}|^\g
 \xi_j
 u_p^{\xi-2e_j}
\,,
 \end{equation*}
 with
 $\xi:=2\zeta+\al+\bt+a+b$
 and
  again $m,\al,\bt,\zeta,\delta\in \N^\cS, a,b\in\N^{\cS^c}, \al\cap \bt= \emptyset$ and
 where $u_p=u_p(r')$ was defined in \eqref{giancarlo}
 and $I_p:= u_p^2/2$.
By the formula after (4.30) in \cite{BMP:almost} we have
$Y\in\tw_p$ with $|Y|_{p}\leq 3^{\frac{d}{2}+1}
|H|_{r,p}$ and  $|X_j(u,\o,I)|\leq Y_j$
for every $j\in\Z,$ 
$u\in B_{r'}(\tw_p)$, $\o\in\cO$,
$I\in \mathcal I(p,r)$ (resp. $I\in \mathcal I(p,r,\C)$ in the complex case).
Therefore $H^{(d)}$
is analytic in   
$\mathcal I(p,r)$ (resp. $I\in \mathcal I(p,r,\C)$ in the complex case)
since can be written in a totally (a fortiori uniformly) convergent series
(see e.g. Theorem 2, Appendix A of  \cite{PT}).
\end{proof}

\medskip

\begin{proof}[{\bf Proof of Lemma  \ref{luchino}}]\footnote{Note that in Lemma C.4 of \cite{BMP:2019}
was stated the following erroneous result: 
{\it	Let $0<a<1$ and $x_1\geq x_2\geq \ldots\geq x_N\geq 2.$ Then
	$
	\frac{\sum_{1\leq\ell\leq N} x_\ell}{\prod_{1\leq\ell\leq N} x_\ell^a}
	\leq 
	x_1^{1-a}+\frac{2}{a x_1^a}\,.
	$
} Anyway in \cite{BMP:2019} we only had used the above result with $a=1$, which is exactly the content of Lemma \ref{luchino}.
}
By induction over $N$.
The result is obvious for $N=1$. 
Let assume it for $N$ and show it for $N+1$.
We have  
	\begin{eqnarray*}
	&&\frac{\sum_{1\leq\ell\leq N+1} x_\ell}{\prod_{1\leq\ell\leq N+1} \sqrt{x_\ell}}
	\leq
	\frac{(\sum_{1\leq\ell\leq N} x_\ell) +x_{N+1}}{(\prod_{1\leq\ell\leq N} \sqrt{x_\ell})\sqrt{x_{N+1}}} 
	\\
	&\leq& 
	\left(\sqrt{x_1}+\frac{4}{ \sqrt{x_1}}\right)
	\frac{1}{\sqrt{x_{N+1}}}+\frac{\sqrt{x_{N+1}}}{\sqrt 2}\,.
	\end{eqnarray*}
	since $x_1\geq x_2\geq \ldots\geq x_N\geq 2$
	implies  $\prod_{1\leq\ell\leq N} \sqrt{x_\ell}\geq\sqrt 2.$
	It remains to prove that
	$$
\left(\sqrt{x_1}+\frac{4}{ \sqrt{x_1}}\right)
	\frac{1}{\sqrt{x_{N+1}}}+\frac{\sqrt{x_{N+1}}}{\sqrt 2}	
	\leq
	\sqrt{x_1}+\frac{4}{ \sqrt{x_1}}\,.
	$$
	Denoting $t:=\sqrt{x_1}$ and $s:=\sqrt{x_{N+1}}$,
	the above inequality is equivalent to 
	$$
	f(t,s):=2 t^2 s-\sqrt 2 t s^2+8s-2 t^2-8\geq 0
	$$
	for $\sqrt 2\leq s\leq t.$ Since $f$ is a concave function of $s$ we have that
	$$f(t,s)\geq \min\{ f(t,\sqrt 2), f(t,t)\}\,.$$
	It is immediate to see that 
	$$
	\min_{t\geq \sqrt 2} f(t,\sqrt 2)=7\sqrt 2 -9>0\,,
	\qquad
	\min_{t\geq \sqrt 2} f(t,t)=12\sqrt 2-16>0\,,
	$$
	showing that $f(t,s)>0$ for $\sqrt 2\leq s\leq t$
	and concluding the proof.
\end{proof}

\medskip

\begin{proof}[{\bf Proof of Lemma \ref{granita}}]
Let $\bar\delta:= \frac{\delta}{9}$.
We split the sum in \eqref{cioli} into two terms $A_k=A_k^*+A_k^>,$ with
$$
\begin{aligned}
&A_k^*:=\sum_{i< i_*: k_i\geq 1} - \bar\delta  k_i\log\jjap{s(i)} + \log\pa{1 + \jap{i}^2 k_i^2}\,,
\\
&A_k^>:=\sum_{i_*\leq i\leq i_\tM: k_i\geq 1} - \bar\delta  k_i\log\jjap{s(i)} + \log\pa{1 + \jap{i}^2 k_i^2}\,.
\end{aligned}
$$

Regarding the first term we get, recalling that $s(i)\geq i,$
\begin{eqnarray*}
A_k^*
&\leq&
\sum_{i< i_*: k_i\geq 1} - \bar\delta  k_i\log\jjap{i} +1+ 2\log \jap{i} +2 \log k_i
\\&\leq &
8 i_*\log i_* - \!\!\!\!\!\!
\sum_{i< i_*: k_i\geq 1} \!\!\!\!((\bar\delta\log 2)  k_i -2 \log k_i)
\leq 
8 i_*\left(
\log i_*+
\log\frac{1}{\bar\delta}
\right)\,,
\end{eqnarray*}
using that $\max_{x\geq 1} - (\bar\delta\log 2)  x +2 \log x\leq 2 \log 1/\bar\delta.$
\\
Consider now the second term. By \eqref{cicoria} we get
\begin{eqnarray*}
A_k^>
&\leq&
\sum_{i_*\leq i\leq i_\tM, k_i\geq 1} - \bar\delta  k_i\log^{1+\eta}|i|
 + \log\pa{1 + \jap{i}^2 k_i^2}
 \leq
 A_{k,1}^>+A_{k,2}^>\,,
\end{eqnarray*}
where
$$
 A_{k,1}^>:=
\sum_{i_*\leq i\le i_\tM, k_i\geq 28/\bar\delta} 2f(i,k_i)\,,
\qquad
A_{k,2}^>:=
\sum_{i_*\leq i\le i_\tM, 1\leq k_i< 28/\bar\delta} 2f(i,k_i)\,,
$$
and\footnote{Recall that $i_*\geq 3$.}
\begin{equation}\label{vita}
f(i,k):=- \bar\delta  k\log^{1+\eta}i +3\log i+2\log (\bar\delta k)+2\log(1/\bar\delta)\,.
\end{equation}
When $k_i\geq 28/\bar\delta$ we have\footnote{Using that
$-\frac14 x+2\log x\leq 0$ for $x\geq 28$ and $\log i\geq \log i_*\geq 1.$}
$$
\begin{aligned}
f(i,k_i)
&\leq
- \frac12\bar\delta  k_i\log^{1+\eta}i 
+2\log (\bar\delta k_i)+2\log(1/\bar\delta)\\
&\leq 
- \frac14\bar\delta  k_i\log^{1+\eta}i 
+2\log(1/\bar\delta)
\leq 
- 7\log^{1+\eta}i 
+2\log(1/\bar\delta)
\,,
\end{aligned}
$$
which is negative for $i\geq 1/\bar\delta.$ Then we get
$
A_{k,1}^>
\leq
\frac{4}{\bar\delta}\log(1/\bar\delta)\,.
$
\\
We finally  consider the term $A_{k,2}^>$, namely 
when $k_i< 28/\bar\delta$. In this case
$$
f(i,k_i)\leq \Big(3- \bar\delta  \log^{\eta}i \Big)\log i+7+2\log(1/\bar\delta)\,.
$$
We claim that the last quantity is negative for $i\geq \exp(4\bar\delta^{-1/\eta}).$
Indeed, since $\eta\leq 2$
$$
f(i,k_i)\leq -\frac{4}{\sqrt{\bar\delta}}+7+2\log(1/\bar\delta)
\leq 0
$$
for $\bar\delta\leq 1/2.$
On the other hand for  $i< \exp(4\bar\delta^{-1/\eta})$ we have
$$
f(i,k_i)\leq \frac{12}{\bar\delta}+7+2\log(1/\bar\delta)
\leq  \frac{20}{\bar\delta}\,.
$$
So we finally get
$$
A_{k,2}^>
\leq
\frac{40}{\bar\delta}\exp(4\bar\delta^{-1/\eta})
\leq 40\exp(5\bar\delta^{-1/\eta})\,.
$$
\end{proof}

\medskip

\begin{proof}[{\bf Proof of Lemma \ref{zufolo}}]
To fix ideas we consider the case 
$\xi_n=-1$.
Set $x=:(\hat x,x_n)$.
Let us introduce the portion of hyperplane  (which is a graph over $\hat x$)
$$
P:=\{(\hat x,\hat\xi\cdot\hat x)\,:\ \ \hat x\in [-\nicefrac14,\nicefrac14]^{n-1}\}\,,
$$
orthogonal to $\xi$.
Note that for every $y\in E$ there exist unique $t\in\R$ and
$\hat x\in [-\nicefrac14,\nicefrac14]^{n-1}$ such that
$
y=(\hat x,\hat\xi\cdot\hat x)+ t \xi
$.
Then by Fubini's theorem we have that
\begin{eqnarray*}
\meas(E) &=&|\xi|_2 \int_P \meas\{t\in\R\ \ :\ \ 
(\hat x,\hat\xi\cdot\hat x)+ t \xi \in \E\}\,d\sigma
\,\\
&\leq\,&
|\xi|_2\delta \int_P d\sigma=|\xi|_2\delta \int_{[-\nicefrac14,\nicefrac14]^{n-1}}
\sqrt{1+|\hat\xi|_2^2} d\hat x
\\
&=&2^{1-n}\delta |\xi|_2^2.
\end{eqnarray*}
\end{proof}

\section{Topology, measure and continuous functions on infinite product spaces}\label{ciavatta}

\subsection*{Product topology}

Fix $\varrho>0$.
Let us consider the set $[-\varrho,\varrho]^\Z$ endowed with the {\sl product topology},
namely the coarsest topology (i.e. the topology with the fewest open sets) for which all the projections 
$\pi_{j}:[-\varrho,\varrho]^\Z\to[-\varrho,\varrho]$, with $\pi_j(\omega):=\omega_j,$ $j\in\Z,$ are continuous.
\\
We call  {\sl cylinder} a subset
$\bigotimes_{n\in \N} A_n$
 of $[-\varrho,\varrho]^\Z$ such that 
$A_n \subseteq[-\varrho,\varrho]$
with $A_n \neq [-\varrho,\varrho]$ only for finitely many $n\in\N$.
Then a basis of the product topology is given by the open cylinders,
namely cylinders $\bigotimes_{n\in \N} A_n$, where $A_n$ are open.
\\
By the Tychonoff's theorem
$[-\varrho,\varrho]^\Z$ with the product topology is a
compact Hausdorff space.


\subsection*{Product measures}

The {\sl product $\sigma$-algebra (of the Borel sets)} of $[-\varrho,\varrho]^\Z$ is generated by 
the set of cylinders $\bigotimes_{n\in \N} A_n$, where $A_n$ are
Borel sets (w.r.t. the standard topology on $[-\varrho,\varrho]$).
\\
The {\sl  probability product measure} $\mu$ on 
product $\sigma$-algebra  of $[-\varrho,\varrho]^\Z$
is defined by
$$
\mu \big(\bigotimes_{n\in \N} A_n \big):=
\prod_{n\in\N}
\frac1{2\varrho} 
 |A_n |\,,
 $$
where $|A_n|$ denotes  the usual Lebesgue measure of the Borel set $A_n$.

Through the bijective  map 
$$
\Vs:\pan\ \to \ [-\nicefrac12,\nicefrac12]^\Z\,,\qquad
\mbox{where}\qquad
\mathscr V^*_{j}(\nu):=\nu_j-j^2\,,\ \ \ j\in\Z\,,
$$
we induce the product topology and the probability product measure 
 on the set $\pan$.
Analogously for $\pan_\cS$ and $\pan_{\cS^c}$.

\subsection*{Product measures}
\medskip

Given a compact Hausdorff  space $X$ and a Banach space $E$ 
we denote by $C(X,E)$ the Banach space of continuous functions
$f:X\to E$ endowed with the uniform norm 
$$
|f|_{C(X,E)}:=\sup_{x\in X}|f(x)|_E=\max_{x\in X}|f(x)|_E\,.
$$

\begin{lemma}[Lipschitz fixed point]\label{zagana}
Let $\mathcal C$ be the closed subset of the Banach space 
$C(\pan_{\cS}\times [-\nicefrac14,\nicefrac14]^{\cS^c},\ell^\infty_{\cS^c})$
(with the product topology on $\pan_{\cS}\times [-\nicefrac14,\nicefrac14]^{\cS^c}$)
defined as 
$$
\mathcal C:=\{ w\in   C(\pan_{\cS}\times [-\nicefrac14,\nicefrac14]^{\cS^c},\ell^\infty_{\cS^c})\quad \mbox{s.t.}\quad
|w|_{C(\pan_{\cS}\times [-\nicefrac14,\nicefrac14]^{\cS^c},\ell^\infty_{\cS^c})}\leq r\}\,, 
$$
for some $0<r<1/4$.
Let $F\in C(\pan_\cS\times [-\nicefrac14,\nicefrac14]^{\cS^c}\times [-r,r]^{\cS^c},\ell^\infty_{\cS^c})$ with
$$
|F(\nu,\VSc,w)|_{\ell^\infty_{\cS^c}}\leq r
\,,\qquad
|F(\nu,\VSc,w)-F(\nu,\VSc,w')|_{\ell^\infty_{\cS^c}}
\leq 
1/2 |w-w'|_{\ell^\infty_{\cS^c}}\,,
$$
for all $ \nu\in\pan_{\cS}\,, 
w,w'\in \pan_{\cS^c}\,.$
Then there exists a unique $w\in \mathcal C$ such that
$$
w(\nu,\VSc)=F(\nu,\VSc,w(\nu, \VSc))\,.
$$
Moreover if $F$ is $L$-Lipschitz for some $L>0$ w.r.t. $\nu\in \pan_{\cS}$ 
(endowed with the $\ell^\infty$-metric)
then  $w$ is $2L$-Lipschitz in $\nu$.
Analogously if $F$ is $L'$-Lipschitz w.r.t. some  other parameter $I$ in some Banach space
then  $w$ is $2L'$-Lipschitz in $I$.
\end{lemma}
\begin{proof}
First note that if $w\in\mathcal C$ then 
$F(\nu,\VSc,w(\nu,\VSc))\in\mathcal C$, since the product topology in 
$[-r,r]^{\cS^c}$ is weaker than the one induced by the $\ell^\infty$-norm.
Set $\Phi:\mathcal C\to \mathcal C$ by
$$
(\Phi(w))(\nu,\VSc):=F(\nu,\VSc,w(\nu,\VSc))\,.
$$
Let us check that $\Phi$ is a contraction on $\mathcal C$.
Indeed
\begin{eqnarray*}
&&|\Phi(w)-\Phi(w')|_{\mathcal C}
= \\
&& \sup_{\nu\in \pan_\cS,\,
\VSc\in [-\nicefrac14,\nicefrac14]^{\cS^c}} |F(\nu,\VSc,w(\nu,\VSc))-F(\nu,\VSc,w'(\nu,\VSc))|_{\ell^\infty_{\cS^c}}
\\
&& \leq
 \frac12 \sup_{\nu\in \pan_\cS,\,
\VSc\in [-\nicefrac14,\nicefrac14]^{\cS^c}} |w(\nu,\VSc)-w'(\nu,\VSc)|_{\ell^\infty_{\cS^c}}
=\frac12 |w-w'|_{\mathcal C}\,.
\end{eqnarray*}
The existence of the fixed point follows from the Contraction Mapping Theorem on Banach spaces.
\\
Assume now that 
$F$ is $L$-Lipschitz for some $L>0$ w.r.t. $\nu\in \pan_{\cS}$ 
(endowed with the $\ell^\infty$-metric $d_\infty$). Then
\begin{eqnarray*}
&&
|w(\nu,\VSc)-w(\nu',\VSc)|_{\ell^\infty_{\cS^c}}
\leq
|F(\nu,\VSc,w(\nu,\VSc))-F(\nu',\VSc,w(\nu,\VSc))|_{\ell^\infty_{\cS^c}}
\\
&& +
|F(\nu',\VSc,w(\nu,\VSc))-F(\nu',\VSc,w(\nu',\VSc))|_{\ell^\infty_{\cS^c}}
\\
&&
\leq L d_\infty(\nu,\nu')+\frac12|w(\nu,\VSc)-w(\nu',\VSc)|_{\ell^\infty_{\cS^c}}
\end{eqnarray*}
implying
$$
|w(\nu,\VSc)-w(\nu',\VSc)|_{\ell^\infty_{\cS^c}}
\leq
2L d_\infty(\nu,\nu')\,.
$$
\end{proof}

\medskip

The next lemma regards the analyticity of the map $ \mathfrak i$ in \eqref{manodedios}.
Without loss of generality we consider here only the case $\cS=\Z$
since we are not assuming $I_j\neq 0.$
Let us first introduce the notation $\C^\infty_s:=\{\varphi=(\varphi_j)_{j\in\Z} \ \mbox{with}\ |\Im \varphi_j|<s,\
\forall\ j\in\Z
 \}$ for some $s>0$ and  the equivalence relation $\varphi\sim\varphi'$ iff
 $\varphi-\varphi'\in 2\pi\Z.$
 We set $\T^\infty_s:=\C^\infty_s/\sim$.

\begin{lemma}\label{lemmadedios}
For $\sqrt I\in {\bar B}_{r}(\sob_{p})$ the map $ \mathfrak i$ in \eqref{manodedios}
can be extended to an analytic map
$$
 \mathfrak i: \T^\infty_s \to  \tw_p,\quad 
 [\varphi]=\big[(\varphi_j)_{j\in\Z}\big]\mapsto (\sqrt{I_j} e^{\im \varphi_j} )_{j\in\Z}\,,
$$
for any $s>0.$
\end{lemma}
\begin{proof}
 For $\varphi\in \C^\infty_s$ we denote by
 $[\varphi]\in\T^\infty_s$   the equivalence class of $\varphi.$
$\T^\infty_s$ is a metric space endowed with the
distance
 $$
 d([\varphi],[\psi]):=\min_{\psi'\in[\psi]}|\varphi-\psi'|_\infty\,,
 $$
where $|\cdot|_\infty$ is the norm on $\ell^\infty=\ell^\infty(\C).$
Moreover it is a Banach manifold. Indeed given every point
$[\varphi]\in \T^\infty_s$, set $\rho:=\min\{1, s-|\Im \varphi|\}/2$ 
and consider the open
 ball $B_\rho([\varphi])\subset \T^\infty_s$ of radius $\rho$
centered in $[\varphi]$  and the open ball
$U_\rho$  of radius $\rho$ centered at the origin of 
$\ell^\infty;$ then a local chart is  $\Phi:B_\rho([\varphi])\to 
U_\rho$ defined so that
$\Phi^{-1}(\psi):=[\varphi+\psi].$
We claim that
$$
g:=\mathfrak i\circ \Phi^{-1} : U_\rho \to \tw_p\,,\qquad
g(\psi):=(\sqrt{I_j} e^{\im (\varphi_j+\psi_j)} )_{j\in\Z}
$$
is analytic since the Frechet derivative $D g: U_\rho \to \mathcal L(\ell^\infty,\tw_p)$ is continuous.
Indeed for every $\psi'\in \ell^\infty$
$$
D g(\psi)[\psi']
=
(\im \sqrt{I_j} e^{\im (\varphi_j+\psi_j)} \psi'_j)_{j\in\Z}
$$
and for $\psi,\tilde\psi\in U_\rho$
the operator norm (recall $\sqrt I\in {\bar B}_{r}(\sob_{p})$)
satisfies
\begin{eqnarray*}
\|D g(\psi)-D g(\tilde\psi)\|_{\rm {op}}
&=&
\sup_{\psi'\in \ell^\infty, |\psi'|=1}
\big|\big(\im \sqrt{I_j} e^{\im \varphi_j} (e^{\im \psi_j}-
e^{\im \tilde\psi_j})\psi'_j\big)_{j\in\Z}\big|_{\tw_p}
\\
&\leq& 
r e^{|\Im \varphi|+\rho}
\big|\big( (1-
e^{\im (\tilde\psi_j-\psi_j)})\big)_{j\in\Z}\big|_{\ell^\infty}
\\
&
\leq&
r e^{|\Im \varphi|+2\rho}
|\tilde\psi-\psi|_{\ell^\infty}\,.
\end{eqnarray*}
\end{proof}

{\sl On the density of linear flow $\nu t$ on the flat torus $\cT_I$ (recall \eqref{pippero}).}

\noindent
We note that if $|\nu|_{\ell^\infty}<\infty$ the  linear flow $t\mapsto \nu t$ is not dense on the torus $\cT_I$,
otherwise\footnote{Recall the 
definition of $\cS_I$ in \eqref{shula}.} $\nu \mathbb Q^{\cS_I}$ would be a countable dense subset but 
$\ell^\infty$-based topologies
are not separable. On the other hand if the sequence $\nu_j$  increases very fast (e.g. super-exponentially)
the flow is dense.
\begin{lemma}\label{uccellagione}
 If for any $n\geq 1$ the vector
$\nu^{(n)}$
with entries  $\nu_j$, $|j|\leq n,$ $j\in \cS_I$ is rationally independent and 
$$
\lim_{n\in\cS_I, |n|\to \infty} |\nu_n| \sum_{j\in\cS_I,\,|j|>|n|} |\nu_j|^{-1}=0
$$
then the flow is dense.
\end{lemma}
\begin{proof}
Fix $\delta>0$, $\bar\varphi$ on the torus and call $d$ the distance on 
the one dimensional  torus $d_{\T^1}(\varphi,\psi)=\min_{k\in\Z}|\varphi-\psi-2\pi k|$
for $\varphi,\psi\in\T^1$. 
We want to show that there exists a time $T\in\R$ such that
$d_{\T^1}(\bar\varphi_n,\nu_n T)\leq \delta$ for every $n\in\cS_I.$
Let $n_0\in\cS_I$ large enough such that
 $|\nu_n| \sum_{j\in\cS_I,\,|j|>|n|} |\nu_j|^{-1}\leq \delta/2\pi$ for every $n\in\cS_I,\,|n|\geq |n_0|$.
Since $\nu^{(n_0)}$ is rationally independent there exists $T_0>0$ such that
$d_{\T^1}(\bar\varphi_j,\nu_j T_0)\leq \delta/2$ for every $j\in\cS_I, |j|\leq |n_0|$.
For every $j\in\cS_I,|j|>|n_0|$ there exists $t_j\in\R$ with $|t_j|\leq\pi/|\nu_j|$ such that $d_{\T^1}(\bar\varphi_j,\nu_j (T_0+t_j))=0$. Set $T:=T_0+T'$ where $T':= \sum_{j\in\cS_I,|j|>|n_0|}t_j.$
Note that $|T'|\leq \delta/2|\nu_{n_0}|.$
Then for $n\in\cS_I, |n|\leq |n_0|$ we have $d_{\T^1}(\bar\varphi_n,\nu_n T)\leq \delta/2+
|\nu_n T'|\leq \delta$. Finally for $n\in\cS_I, |n|> |n_0|$ we have
$d_{\T^1}(\bar\varphi_n,\nu_n T)\leq |\nu_n|\sum_{j\in\cS_I,|j|>|n_0|, j\neq n_0}|t_j|\leq \delta/2,$
concluding the proof.
\end{proof}
\noindent
Finally, since in our application to NLS $\nu_j\sim j^2$, if $\cS_I$ is sparse enough the condition of Lemma \ref{uccellagione}
is satisfied and the flow is dense.

%
%
%
	\newcommand{\etalchar}[1]{$^{#1}$}
	\def\cprime{$'$}

\end{document}